\documentclass[12pt,reqno]{amsart}
\usepackage{amsmath}
\usepackage{amssymb}
\usepackage{verbatim}
\usepackage{graphicx}
\usepackage[font=footnotesize]{caption}
\usepackage{subcaption}
\usepackage{epstopdf}
\usepackage{multirow}
\usepackage{url}
\usepackage{makecell}
\usepackage{tabularx}
\usepackage{longtable}
\usepackage{graphicx}
\usepackage{tikz}
\usepackage{pgfplots}
\usepgfplotslibrary{polar}
\usepackage{xcolor}
\usepackage{pgfplots}
\usepackage{pgfplotstable}
\allowdisplaybreaks

\usepackage{bm}
\usepackage[toc]{appendix}
\usepackage[capitalise]{cleveref}
\usepackage{color}
\allowdisplaybreaks

\usepackage[top=0.5in,bottom=0.5in,left=0.7in,right=0.7in]{geometry}

\newtheorem{lemma}{Lemma}

\newtheorem{theorem}[lemma]{Theorem}
\newtheorem{remark}[lemma]{Remark}

\newtheorem{example}[lemma]{Example}

\newcommand{\N}{\mathbb{N}}    
\newcommand{\NN}{\mathbb{N}_0} 
\newcommand{\R}{\mathbb{R}}    

\newcommand{\nv}{\vec{n}}
\newcommand{\bo}{\mathcal{O}}
\newcommand{\Op}{\Omega_+}
\newcommand{\Om}{\Omega_-}

\newcommand{\ka}{\textsf{k}}

\newcommand{\B}{\mathcal{B}}
\newcommand{\id}{\mathbf{I}_d}
\newcommand{\gd}{g}
\newcommand{\gn}{g_{\Gamma}}

\newcommand{\be}{ \begin{equation} }
	\newcommand{\ee}{ \end{equation} }
\newcommand{\odd}{\operatorname{odd}}
\newcommand{\ind}{\Lambda}

\numberwithin{equation}{section}
\numberwithin{lemma}{section}

\begin{document}
	
\title[Hybrid Finite Difference Methods for Elliptic Interface Problems]{Sixth-Order Hybrid Finite Difference Methods for Elliptic Interface Problems with Mixed Boundary Conditions
}

\author{Qiwei Feng, Bin Han, and Peter Minev}

\thanks{Research supported in part by Natural Sciences and Engineering Research Council (NSERC) of Canada under grants RGPIN-2019-04276 (Bin Han), RGPIN-2017-04152 (Peter Minev),  Digital Research Alliance of Canada (https://alliancecan.ca/en)}

\address{Department of Mathematical and Statistical Sciences,
	University of Alberta, Edmonton, Alberta, Canada T6G 2G1.
	\quad {\tt qfeng@ualberta.ca}
	\quad {\tt bhan@ualberta.ca}
	\quad {\tt minev@ualberta.ca}
}

\address{Department of Mathematics, Center for Research in Scientific Computation, North Carolina State University, Raleigh, NC, USA 27695.
	\quad {\tt qfeng4@ncsu.edu}
}

	\makeatletter \@addtoreset{equation}{section} \makeatother

	\begin{abstract}			
In this paper, we develop sixth-order hybrid finite difference methods (FDMs) for the elliptic interface problem $-\nabla \cdot( a\nabla u)=f$ in $\Omega\backslash \Gamma$, where $\Gamma$ is a smooth interface inside $\Omega$. The variable scalar coefficient $a>0$ and source $f$ are possibly discontinuous across $\Gamma$.
The hybrid FDMs utilize a $9$-point compact stencil at any interior regular points of the grid and a $13$-point stencil at irregular points near $\Gamma$.
For interior regular points away from $\Gamma$, we obtain a sixth-order $9$-point compact FDM satisfying the sign and sum conditions for ensuring the M-matrix property.
We also derive sixth-order compact ($4$-point for corners and $6$-point for edges) FDMs satisfying the sign and sum conditions
for the M-matrix property
at any boundary point subject to (mixed) Dirichlet/Neumann/Robin boundary conditions.
Thus, for the elliptic problem without interface (i.e., $\Gamma$ is empty), our compact FDM has the M-matrix property for any mesh size $h>0$ and consequently, satisfies the discrete maximum principle, which guarantees the theoretical sixth-order convergence.
For irregular points near $\Gamma$, we propose fifth-order $13$-point FDMs, whose stencil coefficients can be effectively calculated by recursively solving several small linear systems.
Theoretically, the proposed high order FDMs use high order (partial) derivatives of the coefficient $a$, the source term $f$, the interface curve $\Gamma$, the two jump functions along $\Gamma$, and the functions on $\partial \Omega$. Numerically, we always use function values to approximate all required high order (partial) derivatives in our hybrid FDMs without losing accuracy. Our proposed FDMs are independent of the choice representing $\Gamma$ and are also applicable if the jump conditions on $\Gamma$ only depend on the geometry (e.g., curvature) of the curve $\Gamma$.
Our numerical experiments confirm the sixth-order convergence in the $l_{\infty}$ norm of the proposed hybrid FDMs for the elliptic interface problem.
	\end{abstract}

	\keywords{Elliptic interface problems,  M-matrix for any $h$, high
		order consistency, mixed boundary conditions, corner treatments,   discontinuous and  scalar variable coefficients, complex interface curves}
	
	\subjclass[2010]{65N06, 35J15, 76S05, 41A58}
	\maketitle
	
	\pagenumbering{arabic}
	
	\section{Introduction and motivations}\label{Introdu:motiva}

Elliptic interface problems with discontinuous coefficients appear in many real-world applications: composite materials, fluid mechanics, nuclear waste disposal, and many others.
Consider the domain $\Omega=(l_1, l_2)\times(l_3, l_4)$ and a smooth two-dimensional function $\psi$. Define a smooth curve $\Gamma:=\{(x,y)\in \Omega: \psi(x,y)=0\}$, which partitions $\Omega$ into two subregions:
	$\Op:=\{(x,y)\in \Omega\; :\; \psi(x,y)>0\}$ and $\Om:=\{(x,y)\in \Omega\; : \; \psi(x,y)<0\}$. We also define
	$
	a_{\pm}:=a \chi_{\Omega_{\pm}}$, $f_{\pm}:=f \chi_{\Omega_{\pm}}$ and $u_{\pm}:=u \chi_{\Omega_{\pm}}.
	$
The model problem considered in this paper is defined to be:
	\be \label{Qeques2}
		\begin{cases}
		-\nabla \cdot( a\nabla u)=f &\text{in $\Omega \setminus \Gamma$},\\	
			\left[u\right]=\gd, \quad \left[a\nabla  u \cdot \nv \right]=\gn &\text{on $\Gamma$},\\
			\B_1 u =g_1 \text{ on } \Gamma_1:=\{l_{1}\} \times (l_3,l_4), & \B_2 u =g_2 \text{ on } \Gamma_2:=\{l_{2}\} \times (l_3,l_4),\\
\B_3 u =g_3 \text{ on } \Gamma_3:=(l_1,l_2) \times \{l_{3}\}, &
\B_4 u =g_4 \text{ on } \Gamma_4:=(l_1,l_2) \times \{l_{4}\},
		\end{cases}
	\ee
	where  $f$ is the source term, and for any point $(x_0,y_0)\in \Gamma$ on the interface $\Gamma$, we define
	\begin{align*}
		[u](x_0,y_0) & :=\lim_{(x,y)\in \Op, (x,y) \to (x_0,y_0) }u(x,y)- \lim_{(x,y)\in \Om, (x,y) \to (x_0,y_0) }u(x,y),\\
		[ a\nabla  u \cdot \nv](x_0,y_0) & :=  \lim_{(x,y)\in \Op, (x,y) \to (x_0,y_0) } a\nabla  u(x,y) \cdot \nv- \lim_{(x,y)\in \Om, (x,y) \to (x_0,y_0) } a \nabla  u(x,y) \cdot \nv,
	\end{align*}
where $\nv$ is the unit normal vector of $\Gamma$ pointing towards $\Op$.
In \eqref{Qeques2}, the boundary operators $\B_1,\ldots,\B_4  \in \{\id,\frac{\partial }{\partial \nv}+ \alpha \id\}$, where
$\id$ represents the Dirichlet boundary condition; when $\alpha=0$, $\frac{\partial }{\partial \nv}$ represents the Neumann boundary condition;  when $\alpha \ne 0$, $\frac{\partial }{\partial \nv}+\alpha \id$ represents the Robin boundary condition.   An example for the boundary conditions of  \eqref{Qeques2} is shown in \cref{fig:model:problem}.

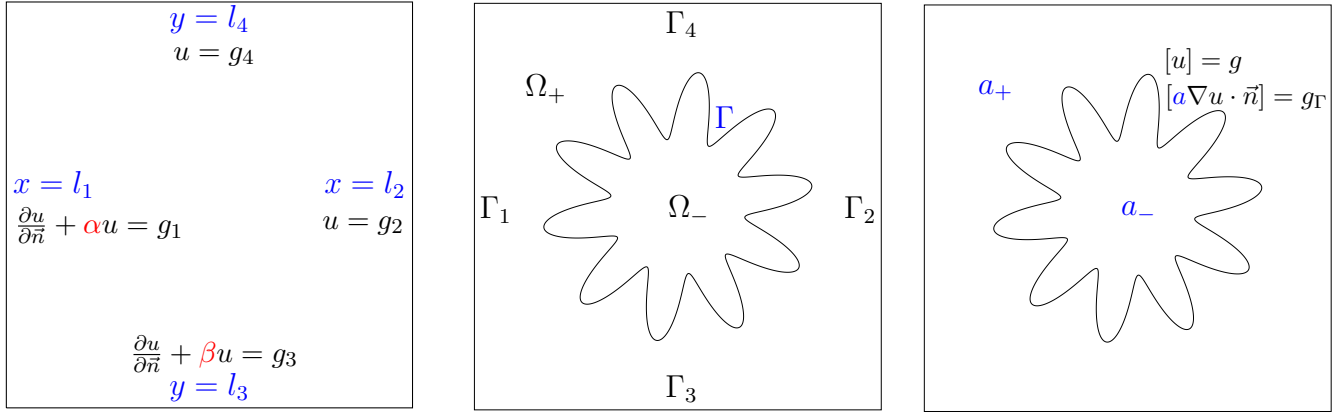
\begin{figure}[htbp]
	\centering	
	\begin{subfigure}[b]{0.3\textwidth}
	\hspace{-0.9cm}
	\begin{tikzpicture}[scale = 0.9]
		\draw
		(-3, -3) -- (-3, 3) -- (3, 3) -- (3, -3) --(-3,-3);
		\node (A) at (-2.3,0.3) {\color{blue}{$x=l_1$}};
		\node (A) at (-1.7,-0.3) {{\small{ $ \tfrac{\partial u}{\partial \nv}+{\color{red}{\alpha}} u=g_1$}}};
		\node (A) at (2.3,0.3) {\color{blue}{$x=l_2$}};
		\node (A) at (2.2,-0.3) {{\small{ $u=g_2$}}};
		\node (A) at (0,-2.7) {\color{blue}{$y=l_3$}};
		\node (A) at (0,-2.2) {{\small{ $\tfrac{\partial u}{\partial \nv}+{\color{red}{\beta}} u=g_3$}}};
		\node (A) at (0,2.75) {\color{blue}{$y=l_4$}};
		\node (A) at (0,2.2) {{\small{ $u=g_4$}}};
	\end{tikzpicture}
\end{subfigure}
	\begin{subfigure}[b]{0.3\textwidth}
		\begin{tikzpicture}[scale = 0.95]	
	\begin{axis}[
	xmin=-3, xmax=3,
	ymin=-3, ymax=3,
	axis equal image,
	xtick=\empty,
	ytick=\empty]
	\addplot[data cs=polar,black,domain=0:360,samples=360,smooth] (x,{(1.5+0.5*sin(10*\x)) });
\end{axis}
	\node (A) at (0.3,2.8) {$\Gamma_1$};
	\node (A) at (5.4,2.8) {$\Gamma_2$};
	\node (A) at (2.9,5.4) {$\Gamma_4$};
	\node (A) at (2.9,0.3) {$\Gamma_3$};
	\node (A) at (3.5,4.1) {\color{blue}{$\Gamma$}};
	\node (A) at (3,2.8) {$\Omega_-$};
	\node (A) at (1,4.5) {$\Omega_+$};
		\end{tikzpicture}
	\end{subfigure}
	\begin{subfigure}[b]{0.3\textwidth}
		\hspace{0.4cm}
		\begin{tikzpicture}[scale = 0.95]	
	\begin{axis}[
		xmin=-3, xmax=3,
		ymin=-3, ymax=3,
		axis equal image,
		xtick=\empty,
		ytick=\empty]
		\addplot[data cs=polar,black,domain=0:360,samples=360,smooth] (x,{(1.5+0.5*sin(10*\x)) });
	\end{axis}
	\node (A) at (3,2.8) {\color{blue}{$a_-$}};
	\node (A) at (1,4.5) {\color{blue}{$a_+$}};
	\node (A) at (3.9,4.9) {{\footnotesize{$\left[u\right]=g$}}};
	\node (A) at (4.45,4.4) { {\footnotesize{ $\left[{\color{blue}{a}}\nabla  u \cdot \nv \right]=g_{\Gamma}$}}};
\end{tikzpicture}
	\end{subfigure}
	\caption
	{ An example for the (mixed) boundary conditions of  \eqref{Qeques2}, where $\B_1u=\frac{\partial u}{\partial \nv}+\alpha u$, $\B_2u=u$,  $\B_3u=\frac{\partial u}{\partial \nv}+\beta u$ and  $\B_4u= u$. Note that $a_{+}(x,y)$ and $a_{-}(x,y)$ are 2D functions, $a_{+}(x,y)\ne a_{-}(x,y)$ can happen for $(x,y)\in \Gamma$, while $\alpha(y)$ and $\beta(x)$ are 1D functions for describing boundary conditions.}
	\label{fig:model:problem}
\end{figure}

On the one hand,
the extensively studied elliptic problem without interface corresponds to \eqref{Qeques2} with $\Gamma$ being empty (i.e., no interface $\Gamma$).
To solve  $-\nabla \cdot( a\nabla u)=f$ without interface for a scalar variable coefficient $a>0$,
the monotonicity is sufficient to guarantee that the corresponding scheme satisfies a discrete maximum principle which is used to prove the convergence rate  theoretically. Furthermore, the M-matrix property is sufficient to prove that the corresponding scheme is monotone. For the convection-diffusion problem, the stiffness matrix of the first-order  finite element method (FEM) in \cite{XuZikatanov1999} is an M-matrix  under some mild constraints on the finite element grids.
Note that even for the Poisson equation $-\Delta u=f$  almost all  high order schemes, except for some high order 9-point FDMs, do not result in an M-matrix, due to positive off-diagonal entries \cite{LiZhang2020}. For example, for the second-order FEM on a generic triangular mesh,  a very strong mesh constraint is required to satisfy the discrete maximum principle  for $-\Delta u=0$ \cite{Mittelmann1981}. For the third and higher order FEM  on regular triangular meshes,  \cite{Vejchodsky2009} showed that the discrete maximum principle could not hold for $-\Delta u=f$.
For the elliptic equation $-\nabla \cdot( a\nabla u)+cu=f$ with a scalar variable coefficient $a>0$ and $c\ge 0$,  \cite{LiZhang2020}  proposed a fourth-order FDM implementation of $C^0$--$Q^2$ FEM, where the corresponding  matrix is not an M-matrix but monotone under some suitable mesh constraints. The additional mesh constraints in \cite{LiZhang2020} can be satisfied for small $h$, but are not all satisfied for any $h$ even if $a$ is concave and $c=0$.
In this paper, we explicitly construct  a 9-point compact  scheme for the elliptic equation with a variable scalar coefficient, that has the sixth-order consistency, and satisfies the  M-matrix property for any $h>0$, without any mesh constraint.

At present, perhaps the most popular approach to handle elliptic problems with discontinuous coefficients
is the so-called  immersed interface method (IIM), proposed by LeVeque and Li (e.g., see \cite{LeLi94}).
It has been combined with FDM, FEM, and finite volume method (FVM) spatial discretizations, with various degrees of accuracy.
Some of the most important developments include: the second-order IIM \cite{CFL19,LeLi94},
the first-order immersed finite volume-element method \cite{EwingLLL99},
the second-order immersed FEM \cite{GongLiLi08,HeLL2011},
the second-order  fast iterative IIM  \cite{Li98}, the second-order explicit-jump IIM  \cite{WieBube00},
the third-order  9-point compact FDM  \cite{PanHeLi21}, and
fourth-order IIM  \cite{XiaolinZhong07}.
Another possible approach to handle irregular points is the matched interface and boundary method (MIB). The
related papers of MIB for the elliptic interface problems can be summarized as:  second-order MIB  \cite{YuZhouWei07},
fourth-order MIB  \cite{ZW06},  fourth-order augmented MIB with the FFT acceleration  \cite{FendZhao20pp109677,RenZhao2023}, sixth-order  MIB  \cite{YuWei073D,ZZFW06}.
For the anisotropic elliptic interface problems with discontinuous and matrix coefficients, \cite{DFL20} proposed a new second-order finite
element-finite difference (FE-FD) method.
		A relatively simple second-order finite volume technique for elliptic problems with discontinuous solutions was introduced in \cite{BochkovGibou2020}.  An attractive feature of this approach is that it yields a linear system with a bounded condition number. \cite{EganGibou2020} proposed the so-called xGFM (extended Ghost Fluid
		Method) to recover convergence of the fluxes. Another second-order method called Voronoi Interface Method, that yields a  symmetric positive definite matrix,  was introduced in \cite{GuittetLTG2015}.
\cite{FHM21b} developed a 9-point compact FDM for elliptic interface problems with discontinuous scalar coefficients, that is formally fourth-order consistent  away from the interface of singularity of the solution (regular points), and third-order consistent  in the vicinity of the interface (irregular points).
For the elliptic cross-interface
problem with a vertical and a horizontal straight line, we derived a sixth-order 9-point  compact FDM with the M-matrix property for the specific case (the internal interfaces coincide with some grid lines) in \cite{FHM2023}. For the general case (the interfaces are not matched by grid lines), we proposed a fourth/fifth-order 9-point compact FDM without the M-matrix property and a third-order 9-point compact  FDM with the M-matrix property in \cite{FHM2023}.
In the present paper  we derive a 9-point compact scheme that has the sixth-order consistency  at regular  points. Additionally, we derive
a 13-point  discretization at irregular points that achieves the fifth-order consistency. Further, the two discretizations are combined in a hybrid scheme that utilizes a 9-point stencil  with the sixth-order consistency for  regular points and a 13-point stencil with the fifth-order consistency for irregular points.   Our numerical experiments confirm the sixth-order convergence in the $l_{\infty}$ norm.
Furthermore, we also propose a recursive solver to efficiently derive the stencil coefficients of the proposed scheme. The resulting sixth-order hybrid scheme shows a significantly improved numerical performance with a slight increase in its complexity over the fourth-order scheme in \cite{FHM21b}. Theoretically, our hybrid FDMs use high order (partial) derivatives of the coefficient function $a$, the source term $f$, the interface curve $\Gamma$, the two jump functions $g,g_{\Gamma}$ on $\Gamma$, and boundary data. In this paper, we always use a numerical technique to employ only function values to estimate all required high order (partial) derivatives in our proposed hybrid FDMs without losing their accuracy and performance.

 A comprehensive literature review of the high order schemes for mixed boundary conditions  can be found in \cite{LiPan2021}. In addition, one should also mention the following literature concerned with the discretization of the boundary conditions for Poisson/elliptic/Helmholtz problems in rectangular domains:
the sixth-order  FDM for 1-side Neumann/Robin and 3-side Dirichlet boundary conditions of Helmholtz equations \cite{Nabavi07,TGGT13},
 the fourth-order FDM for flux boundary conditions for diffusion-advection/anisotropic equations \cite{LiPan2021},
4th-8th-order MIB methods with the FFT acceleration for mixed boundary conditions of Dirichlet/Neumann/Robin for Poisson/elliptic/Helmholtz equations  \cite{FendZhao20pp109677,FendZhao20pp109391}. Furthermore, \cite{ZhaoWei2009} proposed the MIB method to implement general boundary conditions in high order central FDMs in various differential equations.
For elliptic problems with various boundary conditions in non-rectangular domains, \cite{RenFengZhao2022} proposed
a fourth-order augmented MIB with the FFT acceleration,
\cite{WieBube00} developed a
second-order explicit-jump IIM, and \cite{ItoLiKyei05,LiIto06,PanHeLi21}
proposed third/fourth-order FDMs.
In \cite{FHM21Helmholtz}, we discussed sixth-order FDMs for various boundary conditions of the Helmholtz equation with constant wavenumbers. In this paper, we consider the elliptic equation with the variable coefficient $a$ and mixed combinations  of Dirichlet,  Neumann $\tfrac{\partial u}{\partial \nv}$, and Robin boundary conditions $\tfrac{\partial u}{\partial \nv}+\alpha u$, $\tfrac{\partial u}{\partial \nv}+\beta u$ with  variable functions $\alpha,\beta$ (see \cref{fig:model:problem} for an example of the mixed boundary conditions). Finally,	we  derive the $6$-point FDM for edge points and $4$-point FDM for corner points  with the sixth-order consistency and the M-matrix property for any $h>0$ if $\alpha\ge0$  and $\beta\ge0$ on $\partial \Omega$ (see \cref{fig:model:problem}).

In this paper, we consider the model  problem \eqref{Qeques2} under the following assumptions:
	\begin{itemize}
		
	\item[(A1)] The coefficient	$a$ is positive and has uniformly continuous partial derivatives of (total) orders up to six  in each of the subregions $\Op$ and $\Om$, but $a$ may be discontinuous across $\Gamma$.
		
		\item[(A2)] The solution $u$ and the source term $f$ have uniformly continuous partial derivatives of (total) orders up to seven and five respectively in each of the subregions $\Op$ and $\Om$. Both the solution $u$ and the source function $f$ can be discontinuous across the interface $\Gamma$.

		\item[(A3)] The interface curve $\Gamma$ is smooth: for each $(x^*,y^*)\in \Gamma$, there exists a local parametric equation of $\Gamma$: $(r(t),s(t))$  such that $(r(t^*),s(t^*))=(x^*,y^*)$	for some $t^*\in \R$, $(r'(t^*),s'(t^*)) \ne (0,0)$,  $r(t)$ and $s(t)$ have uniformly continuous derivatives of (total) order up to five for $t=t^*$.
		
		\item[(A4)] The (essentially 1D) interface jump functions $\gd$ and $\gn$ along $\Gamma$ have uniformly continuous derivatives of (total) orders up to five and four respectively on the interface $\Gamma$.
		
		\item[(A5)] All 1D boundary functions $g_1,\ldots,g_4$ in \eqref{Qeques2} and $\alpha, \beta$ in the Robin boundary conditions
have uniformly continuous derivatives of (total) order up to five on the boundary $\partial\Omega$.
	\end{itemize}

	The organization of this paper is as follows.
		In \cref{subsec:regular} we explicitly construct a 9-point discretization for interior regular points, with the sixth-order consistency, satisfying the M-matrix property for any $h>0$,  without any mesh constraints. We also extend this result to the boundary points. For the sake of readability, we give  in \cref{hybrid:sec:proofs} the corresponding 6-point and 4-point discretizations in the vicinity of $\partial \Omega$,
	 in \cref{thm:Robin:Gamma1,thm:Corner:1}, respectively.
		In \cref{hybrid:Irregular:points}, we provide the 13-point FDM with the fifth-order consistency for irregular points.
We explicitly derive a recursive solver that decomposes the original linear system in \cref{thm:13point:scheme} into several small linear systems for computing the stencil coefficients effectively.
In \cref{Estimate:derivatives}, we discuss how to estimate high order (partial) derivatives used in the computation of the stencil coefficients, only using the values of the corresponding function.
	 In \cref{sec:numerical}, we present some numerical examples which confirm the sixth-order  convergence of the proposed hybrid scheme in the $l_{\infty}$ norm.	
In \cref{sec:hybrid:Conclu}, we summarize the main contributions of this paper. Finally, in \cref{hybrid:sec:proofs}, we present the proofs of \cref{thm:regular:interior,thm:tranmiss:interface,thm:13point:scheme} in \cref{sec:sixord} and \cref{thm:Robin:Gamma1,thm:Corner:1} for sixth-order FDMs for mixed boundary conditions.
	\section{Hybrid FDMs on uniform Cartesian grids for the elliptic interface problem}
	\label{sec:sixord}

In this section we propose hybrid FDMs on uniform Cartesian grids by using $9$-point compact stencils at regular points and $13$-point stencils at irregular points near the interface $\Gamma$.

\subsection{Some auxiliary identities}
To present hybrid FDMs for the elliptic interface problem, we  introduce some auxiliary identities used in this paper.
First, we shall use the following notations for the $(m,n)$th partial derivatives:
	\be\label{ufmn}
		\begin{split}
		&a^{(m,n)}:=\frac{\partial^{m+n} a}{ \partial^m x \partial^n y},
	\qquad u^{(m,n)}:=\frac{\partial^{m+n} u}{ \partial^m x \partial^n y},\qquad
	f^{(m,n)}:=\frac{\partial^{m+n} f}{ \partial^m x \partial^n y}.
	\end{split}
	\ee
From $-\nabla \cdot( a\nabla u)=f$, we have $au_{xx}+a u_{yy}+a_x u_x+a_y u_y=-f$, which is just
\be \label{uderivx2}
u^{(2,0)}=-\left(\frac{f}{a}\right)-u^{(0,2)}-\left(\frac{a^{(1,0)}}{a}\right) u^{(1,0)}-\left(\frac{a^{(0,1)}}{a}\right) u^{(0,1)}.
\ee
%
Consequently, for any $m,n\in \NN$, applying the Leibniz differentiation formula to \eqref{uderivx2}, we obtain
\be \label{recursive}
\begin{split}
u^{(m+2,n)}=
&-\left(\tfrac{f}{a}\right)^{(m,n)}
-u^{(m,n+2)}\\
&-\sum_{\textsf{i}=0}^m \binom{m}{\textsf{i}} \sum_{\textsf{j}=0}^n \binom{n}{\textsf{j}} \left(\left(\tfrac{a^{(1,0)}}{a}\right)^{(m-\textsf{i},n-\textsf{j})}  u^{(\textsf{i}+1,\textsf{j})}+\left(\tfrac{a^{(0,1)}}{ a}\right)^{(m-\textsf{i},n-\textsf{j})} u^{(\textsf{i},\textsf{j}+1)}\right),
\end{split}
\ee
where $\binom{m}{\textsf{i}}:=\frac{m!}{\textsf{i}!(m-\textsf{i})!}$.
Note that the $x$-derivative (i.e., with respect to $x$) order on the right-hand side of
 \eqref{recursive} is always one order less than that on the left-hand side, i.e., though $u^{(m+2,n)}$ on the left-hand side of \eqref{recursive} has the $x$-derivative order $m+2$, all the derivatives $u^{(p,q)}$ on the right-hand side of \eqref{recursive} satisfying $p<m+2$ and $p+q\le m+n+2$.
We now define several index sets $\ind_{M+1}, \ind_{M+1}^{1}, \ind_{M+1}^{2}$, which are employed throughout the whole paper.
For $M+1\in \NN:=\N\cup\{0\}$, we define
\be \label{Sk}
\ind_{M+1}:=\{ (m,n)\in \NN^2 \; : \;  m+n\le M+1 \}, \qquad M+1\in \NN,
\ee
\be \label{indV12}
\ind_{M+1}^{ 2}:=\ind_{M+1}\setminus \ind_{M+1}^{ 1}\quad \mbox{with}\quad
\ind_{M+1}^{ 1}:=\{ (m,n)\in \ind_{M+1} \; : \; m=0,1\}.
\ee
Recursively applying \eqref{recursive} by reducing the derivative orders with respect to $x$ to less than $2$, we have
\be  \label{upq1}
\begin{split}
u^{(p,q)}=\sum_{(m,n)\in \ind_{p+q}^1}
a^{u}_{p,q,m,n}u^{(m,n)} +\sum_{(m,n)\in \ind_{p+q-2}}
	a^{f}_{p,q,m,n}f^{(m,n)} ,\qquad \forall\;
(p,q)\in \ind_{M+1},
\end{split}
\ee
where  $a^{u}_{p,q,m,n}$, $a^{f}_{p,q,m,n}$ are uniquely determined by $\{a^{(i,j)}: (i,j) \in \ind_{M}\}$ with $a^u_{p,q,m,n}=0$ for all $(m,n)\not\in \ind^1_{p+q}$ and $a^f_{p,q,m,n}=0$ for all $(m,n)\not\in \ind_{p+q-2}$,
and can be obtained uniquely by a recursive procedure using \eqref{recursive} as follows:
\be\label{recursive:au:initial}
\begin{split}
 a^u_{p,q,m,n}=\delta(m-p)\delta(n-q) \quad \text{if} \quad (p,q)\in \ind^1_{M+1}, \quad \text{and} \quad   a^u_{p,q,m,n}=0 \quad \text{if} \quad (m,n)\not\in \ind^1_{p+q},
\end{split}
\ee
\be \label{recursive:au}
\begin{split}
	& a^u_{p,q+1,m,n}=(a^u_{p,q,m,n})^{(0,1)}+a^u_{p,q,m,n-1},\\
	&a^u_{p+1,q,m,n}=
	\begin{cases}
	(a^u_{p,q,m,n})^{(1,0)}
	-a^u_{p,q,m+1,n-2}-\chi_{[1,p+q]}(n)
	\sum _{s=0}^{p+q-n} \binom{s+n-1}{n-1} a^u_{p,q,1,s+n-1} (\tfrac{a_y}{a})^{(0,s)},
		&\text{if } m=0,\\
	(a^u_{p,q,m,n})^{(1,0)}
	+a^u_{p,q,m-1,n}
	-\chi_{[0,p+q-1]}(n)\sum _{s=0}^{p+q-1-n} \binom{s+n}{n} a^u_{p,q,1,s+n} (\tfrac{a_x}{a})^{(0,s)}, &\text{if } m=1,\\
	\end{cases}
\end{split}
\ee
where $\delta(0)=1$, and $\delta(m)=0$ if $m\ne 0$, $\chi_{[1,p+q]}(n)=1$ if $1\le n \le p+q$, and $\chi_{[1,p+q]}(n)=0$ if $n<1$ or $n>p+q$,
and the recursive formula for $a^f_{p,q,m,n}$ is similar. See \cref{regular:interior:umn} for an illustration of \eqref{upq1} with $M=6$. By a direct calculation,  \eqref{upq1} can be rewritten as
	\be  \label{upq2}
	\begin{split}
		&u^{(p,q)}
		=(-1)^{\lfloor\frac{p}{2}\rfloor}
		u^{(\odd(p),q+p-\odd(p))}+
		\sum_{(m,n)\in \ind_{p+q-1}^1}
		a^{u}_{p,q,m,n}u^{(m,n)}  +\sum_{(m,n)\in \ind_{p+q-2}}
		a^{f}_{p,q,m,n}f^{(m,n)},\\
		& \text{for any }(p,q)\in \ind_{M+1}^2, \quad \text{and trivially} \quad u^{(p,q)}=u^{(p,q)}  \text{ for any } (p,q)\in \ind_{M+1}^1,
	\end{split}
	\ee
where $\odd(p)=1$ if $p$ is odd, $\odd(p)=0$ if $p$ is even, and the floor function $\lfloor x\rfloor$ is the largest integer less than or equal to $x\in \R$.
\begin{figure}[h]
	\centering	
	\begin{subfigure}[b]{0.3\textwidth}
		\hspace{-5mm}
		\begin{tikzpicture}[scale = 0.45]
			\node (A) at (0,7) {$u^{(0,0)}$};
			\node (A) at (0,6) {$u^{(0,1)}$};
			\node (A) at (0,5) {$u^{(0,2)}$};
			\node (A) at (0,4) {$u^{(0,3)}$};
			\node (A) at (0,3) {$u^{(0,4)}$};
			\node (A) at (0,2) {$u^{(0,5)}$};
			\node (A) at (0,1) {$u^{(0,6)}$};
			\node (A) at (0,0) {$u^{(0,7)}$};
			\node (A) at (2,7) {$u^{(1,0)}$};
			\node (A) at (2,6) {$u^{(1,1)}$};
			\node (A) at (2,5) {$u^{(1,2)}$};
			\node (A) at (2,4) {$u^{(1,3)}$};
			\node (A) at (2,3) {$u^{(1,4)}$};
			\node (A) at (2,2) {$u^{(1,5)}$};
			\node (A) at (2,1) {$u^{(1,6)}$};	
			\node (A) at (4,7) {$u^{(2,0)}$};
			\node (A) at (4,6) {$u^{(2,1)}$};
			\node (A) at (4,5) {$u^{(2,2)}$};
			\node (A) at (4,4) {$u^{(2,3)}$};
			\node (A) at (4,3) {$u^{(2,4)}$};
			\node (A) at (4,2) {$u^{(2,5)}$};
			\node (A) at (6,7) {$u^{(3,0)}$};
			\node (A) at (6,6) {$u^{(3,1)}$};
			\node (A) at (6,5) {$u^{(3,2)}$};
			\node (A) at (6,4) {$u^{(3,3)}$};
			\node (A) at (6,3) {$u^{(3,4)}$};
			\node (A) at (8,7) {$u^{(4,0)}$};
			\node (A) at (8,6) {$u^{(4,1)}$};
			\node (A) at (8,5) {$u^{(4,2)}$};
			\node (A) at (8,4) {$u^{(4,3)}$};
			\node (A) at (10,7) {$u^{(5,0)}$};
			\node (A) at (10,6) {$u^{(5,1)}$};
			\node (A) at (10,5) {$u^{(5,2)}$};
			\node (A) at (12,7) {$u^{(6,0)}$};
			\node (A) at (12,6) {$u^{(6,1)}$};
			\node (A) at (14,7) {$u^{(7,0)}$};
			\draw[ultra thick, blue][->] (7,3) to[left] (18,3);     	
		\end{tikzpicture}
	\end{subfigure}
	\begin{subfigure}[b]{0.3\textwidth}
		\hspace{2.5cm}
		\begin{tikzpicture}[scale = 0.45]
			\node (A) at (0,7) {$u^{(0,0)}$};
			\node (A) at (0,6) {$u^{(0,1)}$};
			\node (A) at (0,5) {$u^{(0,2)}$};
			\node (A) at (0,4) {$u^{(0,3)}$};
			\node (A) at (0,3) {$u^{(0,4)}$};
			\node (A) at (0,2) {$u^{(0,5)}$};
			\node (A) at (0,1) {$u^{(0,6)}$};
			\node (A) at (0,0) {$u^{(0,7)}$};
			\node (A) at (2,7) {$u^{(1,0)}$};
			\node (A) at (2,6) {$u^{(1,1)}$};
			\node (A) at (2,5) {$u^{(1,2)}$};
			\node (A) at (2,4) {$u^{(1,3)}$};
			\node (A) at (2,3) {$u^{(1,4)}$};
			\node (A) at (2,2) {$u^{(1,5)}$};
			\node (A) at (2,1) {$u^{(1,6)}$};	      	
		\end{tikzpicture}
	\end{subfigure}
	\caption
	{The illustration for \eqref{upq1}--\eqref{u:approx} with $M=6$. }
	\label{regular:interior:umn}
\end{figure}
For the sake of presentation, we plug $(x,y)=(x_i^*,y_j^*)\in \overline{\Omega}$ into \eqref{ufmn}, i.e., we use the following abbreviation
notations in the rest of this paper:
\be\label{ufmn2}
a^{(m,n)}:=\frac{\partial^{m+n} a}{ \partial^m x \partial^n y}(x_i^*,y_j^*),\qquad
u^{(m,n)}:=\frac{\partial^{m+n} u}{ \partial^m x \partial^n y}(x_i^*,y_j^*),\qquad
f^{(m,n)}:=\frac{\partial^{m+n} f}{ \partial^m x \partial^n y}(x_i^*,y_j^*).
\ee
Using the Taylor approximation at a base point $(x_i^*,y_j^*)\in \overline{\Omega}$, we have
\be \label{u:original}
u(x+x_i^*,y+y_j^*)
=
\sum_{(m,n)\in \ind_{M+1}}
u^{(m,n)}\frac{x^my^n}{m!n!}+\bo(h^{M+2}), \qquad x,y\in (-2h,2h).
\ee
Plugging \eqref{upq2} into \eqref{u:original} and rearranging terms of $u^{(m,n)}$ with $(m,n)\in \ind_{M+1}^{1}$ (see \cref{regular:interior:umn}), we have
\be \label{u:approx}
\begin{split}
u(x+x_i^*,y+y_j^*)
=
&\sum_{(m,n)\in \ind_{M+1}^{1}}
u^{(m,n)} G_{M+1,m,n}(x,y) +\sum_{(m,n)\in \ind_{M-1}}
f^{(m,n)} H_{M+1,m,n}(x,y)+\bo(h^{M+2}),
\end{split}
\ee
for $x,y\in (-2h,2h)$, where
%
%
%
\be\label{Gmn}
G_{M+1,m,n}(x,y):=\sum_{(p,q)\in \ind_{M+1} }\frac{a^{u}_{p,q,m,n}}{p!q!} x^{p} y^{q}=G_{m,n}(x,y)+\sum_{(p,q)\in \ind_{M+1}^{2} \setminus \ind_{m+n}^{2} }\frac{a^{u}_{p,q,m,n}}{p!q!} x^{p} y^{q},
\ee
\be\label{Hmn}
\begin{split}
	G_{m,n}(x,y):=\sum_{\ell=0}^{\lfloor \frac{n}{2}\rfloor}
	\frac{(-1)^\ell x^{m+2\ell} y^{n-2\ell}}{(m+2\ell)!(n-2\ell)!},\qquad
	H_{M+1,m,n}(x,y):=
	\sum_{(p,q)\in \ind_{M+1}^{2}  } \frac{a^{f}_{p,q,m,n}}{p!q!} x^{p} y^{q}.
\end{split}
\ee
%
%
%
%
 In particular, by a direct calculation, we obtain
\be\label{GM100}
G_{M+1,0,0}(x,y):=1 \quad \text{ for all } M+1\in \NN.
\ee
If $a$ is a constant function, then $G_{M+1,m,n}(x,y)=G_{m,n}(x,y)$ in \eqref{Gmn}.

\subsection{The M-matrix property}

Let $\Omega=(l_1,l_2)\times (l_3,l_4)$ and we assume $l_4-l_3=N_0 (l_2-l_1)$ for some $N_0 \in \N$. For any positive integer $N_1\in \N$, we define $N_2:=N_0 N_1$ and so the grid size is  $h:=(l_2-l_1)/N_1=(l_4-l_3)/N_2$.
	Let
	\be \label{xiyj}
	x_i=l_1+i h, \quad i=0,\ldots,N_1, \quad \text{and} \quad y_j=l_3+j h, \quad j=0,\ldots,N_2.
	\ee
We define $(u_h)_{i,j}$ to be the value of the numerical approximation  $u_h$ of the exact solution $u$ of the elliptic interface problem \eqref{Qeques2}, at the grid point $(x_i, y_j)$.
A  9-point compact stencil centered at a grid point $(x_i,y_j)$ contains nine points $(x_i+kh, y_j+\ell h)$ with stencil coefficients $C_{k,\ell}\in \R$ for $k,\ell\in \{-1,0,1\}$. Define
 \be\label{Compact:Set}
	\begin{split}
		& d_{i,j}^+:=\{(k,\ell) \; : \;
		k,\ell\in \{-1,0,1\}, \psi(x_i+kh, y_j+\ell h)> 0\}, \quad \mbox{and}\\
		& d_{i,j}^-:=\{(k,\ell) \; : \;
		k,\ell\in \{-1,0,1\}, \psi(x_i+kh, y_j+\ell h)\le 0\}.
	\end{split}
\ee
	Thus, the interface curve $\Gamma:=\{(x,y)\in \Omega \; :\; \psi(x,y)=0\}$ splits the nine points of the 9-point compact stencil into two disjoint sets  $\{(x_{i+k}, y_{j+\ell})\; : \; (k,\ell)\in d_{i,j}^+\} \subseteq \Op$ and
	$\{(x_{i+k}, y_{j+\ell})\; : \; (k,\ell)\in d_{i,j}^-\} \subseteq \Om \cup \Gamma$. We refer to a grid/center point $(x_i,y_j)$ as
	\emph{a regular point} if  $d_{i,j}^+=\emptyset$ or $d_{i,j}^-=\emptyset$.
	The center grid point $(x_i,y_j)$ of a stencil is \emph{regular} if all of its nine points are in $\Op$ (hence $d_{i,j}^-=\emptyset$) or in $\Om \cup \Gamma$ (i.e., $d_{i,j}^+=\emptyset$).
	Otherwise, if both $d_{i,j}^+$ and $d_{i,j}^-$ are nonempty, the center grid point $(x_i,y_j)$ of a stencil is referred to as \emph{an irregular point}.
	Now, let us pick and fix a base point $(x_i^*,y_j^*)$ inside the open square $(x_i-h,x_i+h)\times (y_j-h,y_j+h)$, which can be written as
	\be \label{base:pt}
	x_i^*=x_i-v_0h  \quad \mbox{and}\quad y_j^*=y_j-w_0h  \quad \mbox{with}\quad
	-1<v_0, w_0<1.
	\ee

We now discuss the M-matrix property.
An M-matrix is a real square matrix with non-positive off-diagonal entries and positive diagonal entries such that all row sums are non-negative with at least one row sum being positive.
The linear system with an M-matrix has the potential to construct the efficient iterative solver and preconditioner to obtain the solution accurately and effectively. Furthermore, an M-matrix is sufficient to guarantee that the corresponding scheme satisfies the discrete maximum principle which is used to prove the convergence rate  theoretically.
To form an M-matrix, we consider following sign and sum conditions for a scheme with stencil coefficients $\{C_{k,\ell}\}$:
\be\label{sign:condition}
\begin{cases}
	C_{k,\ell}>0, &\quad \mbox{if} \quad (k,\ell)=(0,0),\\
	C_{k,\ell}\le  0, &\quad \mbox{if} \quad (k,\ell)\ne(0,0),
\end{cases}
\ee
and
\be\label{sum:condition}
\sum_{k}\sum_{\ell} C_{k,\ell}\ge 0.
\ee
If all the stencil coefficients $C_{k,\ell}$ are polynomials (in terms of $h$) of degree at most $M+1\in \NN$:
	\be\label{Ckell}
	C_{k,\ell}:=\sum_{p=0}^{M+1} c_{k,\ell,p}h^p \quad\mbox{ with }\quad  c_{k,\ell,p}\in \R,
	\ee
then $\{C_{k,\ell}\}$
 satisfies the sign condition \eqref{sign:condition} for any mesh size $h>0$ if
\be\label{suff:nece:sign:cond}
\begin{cases}
	c_{k,\ell,0}> 0, &\quad \mbox{if} \quad (k,\ell)=(0,0),\\
	c_{k,\ell,0}\le  0, &\quad \mbox{if} \quad (k,\ell)\ne(0,0),
\end{cases}
\quad
\begin{cases}
	c_{k,\ell,p}\ge 0, &\quad \mbox{if} \quad (k,\ell)=(0,0),\\
	c_{k,\ell,p}\le  0, &\quad \mbox{if} \quad (k,\ell)\ne(0,0),
\end{cases}
\quad p=1,\dots, M+1,
\ee
and
$\{C_{k,\ell}\}$  satisfies the sum condition \eqref{sum:condition} for any mesh size $h>0$ if
\be\label{suff:nece:sum:cond}
\sum_{k}\sum_{\ell} c_{k,\ell,p} \ge 0,
\quad p=0,\dots, M+1.
\ee
In this paper, we say that the $\{C_{k,\ell}\}$ in \eqref{Ckell} is nontrivial if $c_{k,\ell,0}\ne0$ for at least one choice of $k,\ell$.
Under suitable boundary conditions such that at least one sum in \eqref{sum:condition} satisfies $\sum_{k}\sum_{\ell} C_{k,\ell}>0$ for any $h$ (such as the Dirichlet boundary condition is imposed on at least one grid point $(x_i,y_j)\in \partial \Omega$ implies $\sum_{k}\sum_{\ell} C_{k,\ell}=1$ for any $h$ on $(x_i,y_j)$),
it is well known that \eqref{suff:nece:sign:cond} and \eqref{suff:nece:sum:cond} together  guarantee the resulting coefficient matrix to be an M-matrix for any $h>0$.
For the sake of better readability, all technical proofs of \cref{sec:sixord} are provided in \cref{proof:Appendix}.


	\subsection{$9$-point compact stencils at regular points (interior)}
	\label{subsec:regular}
	In this subsection, we discuss how to construct 9-point FDMs with the sixth-order consistency and satisfying the M-matrix property for any $h>0$ for interior regular points.
 We choose $(x_i^*,y_j^*)$ to be the center point of the 9-point scheme, i.e., $(x_i^*,y_j^*)=(x_i,y_j)$ and $v_0=w_0=0$ in \eqref{base:pt}.
\begin{figure}[htbp]
	\begin{subfigure}[b]{0.3\textwidth}
			\hspace{-0.6cm}
		\begin{tikzpicture}[scale = 3.5]
			\draw[help lines,step = 0.5]
			(-1/2,-1/2) grid (1/2,1/2);
			\node at (-0.5,0.5)[circle,fill,inner sep=2pt,color=black]{};
			\node at (-0.5,0)[circle,fill,inner sep=2pt,color=black]{};
			\node at (-0.5,-0.5)[circle,fill,inner sep=2pt,color=black]{};
			\node at (0,0.5)[circle,fill,inner sep=2pt,color=black]{};
			\node at (0,0)[circle,fill,inner sep=2pt,color=red]{};
			\node at (0,-0.5)[circle,fill,inner sep=2pt,color=black]{};
			\node at (0.5,0.5)[circle,fill,inner sep=2pt,color=black]{};
			\node at (0.5,0)[circle,fill,inner sep=2pt,color=black]{};
			\node at (0.5,-0.5)[circle,fill,inner sep=2pt,color=black]{};
			\node (A) at (-0.28,0.55) {{{$u_{i-1,j+1}$}}};
			\node (A) at (-0.34,0.05) {{{$u_{i-1,j}$}}};
			\node (A) at (-0.28,-0.45) {{{$u_{i-1,j-1}$}}};
			\node (A) at (0.15,0.55) {{{$u_{i,j+1}$}}};
			\node (A) at (0.09,0.05) {{{$u_{i,j}$}}};
			\node (A) at (0.15,-0.45) {{{$u_{i,j-1}$}}};
			\node (A) at (0.72,0.55) {{{$u_{i+1,j+1}$}}};
			\node (A) at (0.67,0.05) {{{$u_{i+1,j}$}}};
			\node (A) at (0.72,-0.45) {{{$u_{i+1,j-1}$}}};
			\end{tikzpicture}
	\end{subfigure}	
	\begin{subfigure}[b]{0.3\textwidth}
			\hspace{1.6cm}
		\begin{tikzpicture}[scale = 3.5]
			\draw[help lines,step = 0.5]
			(-1/2,-1/2) grid (1/2,1/2);
			\node at (-0.5,0.5)[circle,fill,inner sep=2pt,color=black]{};
			\node at (-0.5,0)[circle,fill,inner sep=2pt,color=black]{};
			\node at (-0.5,-0.5)[circle,fill,inner sep=2pt,color=black]{};
			\node at (0,0.5)[circle,fill,inner sep=2pt,color=black]{};
			\node at (0,0)[circle,fill,inner sep=2pt,color=red]{};
			\node at (0,-0.5)[circle,fill,inner sep=2pt,color=black]{};
			\node at (0.5,0.5)[circle,fill,inner sep=2pt,color=black]{};
			\node at (0.5,0)[circle,fill,inner sep=2pt,color=black]{};
			\node at (0.5,-0.5)[circle,fill,inner sep=2pt,color=black]{};
			\node (A) at (-0.33,0.55) {{{$C_{-1,1}$}}};
			\node (A) at (-0.33,0.05) {{{$C_{-1,0}$}}};
			\node (A) at (-0.28,-0.45) {{{$C_{-1,-1}$}}};
			\node (A) at (0.12,0.55) {{{$C_{0,1}$}}};
			\node (A) at (0.12,0.05) {{{$C_{0,0}$}}};
			\node (A) at (0.16,-0.45) {{{$C_{0,-1}$}}};
			\node (A) at (0.63,0.55) {{{$C_{1,1}$}}};
			\node (A) at (0.63,0.05) {{{$C_{1,0}$}}};
			\node (A) at (0.66,-0.45) {{{$C_{1,-1}$}}};
		\end{tikzpicture}
	\end{subfigure}
	\caption
	{The illustration for the 9-point scheme in \eqref{stencil:regular} of \cref{thm:regular:interior}.}
	\label{Ckl:9:points}
\end{figure}
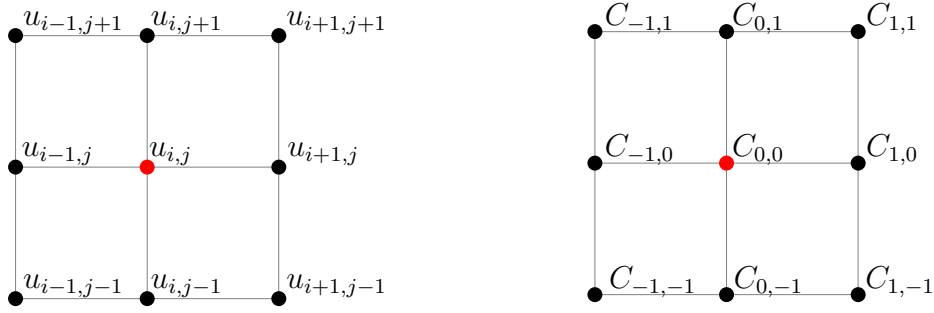

	 \begin{theorem}\label{thm:regular:interior}
Let a grid point $(x_i,y_j)$ be a regular point and $(x_i,y_j) \notin \partial \Omega$ and $(x_i^*,y_j^*)=(x_i,y_j)$.
	Then the following 9-point scheme centered at $(x_{i},y_j)$ (see \cref{Ckl:9:points}):
		\be\label{stencil:regular}
\begin{split}
	h^{-2}	\mathcal{L}_h u_h :=
&h^{-2} \sum_{k=-1}^{1}\sum_{\ell =-1}^{1} C_{k,\ell}(u_{h})_{i+k,j+\ell}
=h^{-2}\sum_{(m,n)\in \ind_{5}} f^{(m,n)} \sum_{k=-1}^1\sum_{\ell=-1}^1 C_{k,\ell} H_{7,m,n}(kh, \ell h),
\end{split}
		\ee
		achieves the sixth-order consistency for $-\nabla \cdot( a\nabla u)=f$ at the point $(x_i,y_j)$, where $\{c_{k,\ell,p}\}_{-1\le k,\ell\le 1, 0\le p\le 7}$ in the stencil coefficients $C_{k,\ell}:=\sum_{p=0}^{7} c_{k,\ell,p}h^p$
 is any nontrivial solution of
				 \be\label{Thm:regular:Equation:Form}
\begin{split}
& \sum_{k=-1}^1 \sum_{\ell=-1}^1  c_{k,\ell,d} G_{m,n}(k,\ell )=-\sum_{k=-1}^1 \sum_{\ell=-1}^1 \sum_{s=0}^{d-1} c_{k,\ell,s}
A^u_{k,\ell,m,n,s},\\
&\mbox{with} \quad
A^u_{k,\ell,m,n,s}:=\sum_{ \substack{(p,q)\in \ind_{7}^2 \setminus \ind_{m+n}^2 \\p+q=m+n+d-s}}\frac{a^{u}_{p,q,m,n}}  {p!q!} k^{p} \ell ^{q}\qquad\quad
\mbox{ for all } (m,n)\in \ind_{7-d}^1,
		\end{split}
		\ee	
for all $d=0,\ldots, 7$, where $ G_{m,n}$ and $ H_{7,m,n}$ are defined in  \eqref{Hmn}, $a^{u}_{p,q,m,n}$ can be computed through \eqref{recursive:au:initial}--\eqref{recursive:au}.
By the symbolic calculation,
 the linear system in \eqref{Thm:regular:Equation:Form} always has nontrivial solutions such that
\begin{itemize}
	\item[(i)]   All nontrivial solutions of  \eqref{Thm:regular:Equation:Form}  satisfy the sum condition \eqref{sum:condition} for any mesh size $h>0$;
	\item[(ii)] There must exist a nontrivial solution of  \eqref{Thm:regular:Equation:Form}   such that
	$\{C_{k,\ell}\}_{k,\ell=-1,0,1}$  satisfies the sign condition \eqref{sign:condition} for any mesh size $h>0$.
\end{itemize}
\end{theorem}
\noindent
\textbf{An efficient way to compute $\{C_{k,\ell}\}$ in \cref{thm:regular:interior}:}	
Obviously, the systems of linear equations in \eqref{Thm:regular:Equation:Form} for $d=0,\ldots, 7$ can be equivalently expressed in matrix forms:
\be\label{Thm:regular:Matrix:Form}
\begin{split}
&A_0 C_0=\textbf{0},\quad A_d C_d=-\sum_{s=0}^{d-1} B_{d,s}C_s, \quad \text{for} \quad 1\le d\le 6,
\quad \mbox{and}\quad A_7 C_7=0,\\
&\mbox{with}\quad C_d=\left(c_{-1,-1,d}, c_{-1,0,d}, c_{-1,1,d}, c_{0,-1,d},
c_{0,0,d}, c_{0,1,d}, c_{1,-1,d}, c_{1,0,d}, c_{1,1,d}\right)^{\textsf{T}},\quad d=0,\ldots,7,
\end{split}
\ee
where $B_{d,s}:=[A^u_{k,\ell,m,n,s}]_{(k,\ell)\in \{-1,0,1\}^2, (m,n)\in \ind_{7-d}^1}$ and
the $15\times 9$ matrix $A_0:=[G_{m,n}(k,\ell)]_{(k,\ell)\in \{-1,0,1\}^2, (m,n)\in \ind_7^1}$ in \eqref{Thm:regular:Matrix:Form} is given by
 {\tiny{
 		\be\label{A0:Regular}
 		A_0=\begin{pmatrix}
 1& 1& 1& 1& 1& 1& 1& 1& 1\\			
 -1& 0& 1& -1& 0& 1& -1& 0& 1\\
 -1& -1& -1& 0& 0& 0& 1& 1& 1\\
 0& -1/2& 0& 1/2& 0& 1/2& 0& -1/2& 0\\
 1& 0& -1& 0& 0& 0& -1& 0& 1\\			
1/3& 0& -1/3& -1/6& 0& 1/6& 1/3& 0& -1/3\\
-1/3& 1/6& -1/3& 0& 0& 0& 1/3& -1/6& 1/3\\
-1/6& 1/24& -1/6& 1/24& 0& 1/24& -1/6& 1/24& -1/6\\
0 & 0& 0 & 0& 0 & 0&  0& 0 & 0 \\
1/30& 0& -1/30& -1/120& 0& 1/120& 1/30& 0& -1/30\\
 1/30& -1/120& 1/30& 0& 0& 0& -1/30& 1/120& -1/30\\
 0& -1/720& 0& 1/720& 0& 1/720& 0& -1/720& 0\\
 -1/90& 0& 1/90& 0& 0& 0& 1/90& 0& -1/90\\
 -1/630& 0& 1/630& -1/5040& 0& 1/5040& -1/630& 0& 1/630\\
1/630& 1/5040& 1/630& 0& 0& 0& -1/630& -1/5040& -1/630\\
 		\end{pmatrix},
 		\ee}}	
and all other matrices $A_1,\ldots, A_7$ are sub-matrices of $A_0$ by deleting some rows of $A_0$ as follows:
 \be\label{A1:A7:regular}
 \begin{split}
&
A_1=A_0(1:13,:),\quad A_2=A_0(1:11,:),\quad A_3=A_0(1:9,:),\quad
  A_4=A_0(1:7,:),\qquad\\
&A_5=A_0(1:5,:), \quad\; A_6=A_0(1:3,:), \qquad A_7=A_0(1,:)=(1,\ldots,1),
  \end{split}
 \ee
where the submatrix $A_0(1:n,:)$ consists of the first $n$ rows of $A_0$.
 All nontrivial solutions of $A_0C_0=\textbf{0}$ in \eqref{Thm:regular:Matrix:Form} are given by $c_{-1,-1,0}=c_{-1,1,0}=c_{1,-1,0}=c_{1,1,0}$, $c_{-1,0,0}= c_{1,0,0}= c_{0,-1,0}=c_{0,1,0}=4c_{1,1,0}$, $c_{0,0,0}=-20c_{1,1,0}$ with the free parameter $c_{1,1,0}\in \R\backslash \{0\}$. We simply choose the trivial solution $C_7=\textbf{0}$ of $A_7C_7=0$ in \eqref{Thm:regular:Matrix:Form} so that all the nine stencil coefficients $\{C_{k,\ell}\}$ are polynomials (in terms of $h$) of degree at most $6$.

\noindent
\textbf{Stencil coefficients $\{C_{k,\ell}\}$ in \cref{thm:regular:interior} forming an M-matrix:} Because the solutions of $\{C_{k,\ell}\}_{k,\ell=-1,0,1}$ to \eqref{Thm:regular:Matrix:Form} are not unique, for our numerical experiments and for achieving the M-matrix property, we set some free parameters in $\{c_{k,\ell,d}\}$ in advance for uniqueness as follows:
\begin{enumerate}
\item[(S1)] Set $C_7=\textbf{0}$, pick a particular solution $C_0$ of $A_0C_0=\textbf{0}$ in
\eqref{Thm:regular:Matrix:Form} by
\[ c_{-1,-1,0}=c_{-1,1,0}=c_{1,-1,0}=c_{1,1,0}=-1, \quad  c_{-1,0,0}= c_{1,0,0}= c_{0,-1,0}=c_{0,1,0}=-4, \quad c_{0,0,0}=20,
\]
and further reduce some free parameters by artificially imposing
{\small{
\[
c_{1,0,4}=c_{1,1,4}, \quad  c_{0,1,5}=c_{1,-1,5}=c_{1,0,5}=c_{1,1,5}, \quad c_{-1,1,6}=c_{0,1,6}=c_{1,-1,6}=c_{1,0,6}=c_{1,1,6},\quad c_{0,0,6}=-8c_{1,1,6}.
\]}}

\item[(S2)] Recursively obtain $C_d:=\left(c_{-1,-1,d}, c_{-1,0,d}, c_{-1,1,d}, c_{0,-1,d},
c_{0,0,d}, c_{0,1,d}, c_{1,-1,d}, c_{1,0,d}, c_{1,1,d}\right)^{\textsf{T}}$ in the order $d=1,\ldots, 6$ by
solving $A_dC_d=b_d$ in
\eqref{Thm:regular:Matrix:Form} with the free parameter  $c_{1,1,d}\in \R$ and then
choose the free parameter $c_{1,1,d}$ to be the maximum value such that
	\be\label{regular:ckld}
	\begin{cases}
		c_{k,\ell,d}\ge 0, &\quad \mbox{if} \quad (k,\ell)=(0,0),\\
		c_{k,\ell,d}\le  0, &\quad \mbox{if} \quad (k,\ell)\ne(0,0).
	\end{cases}
	\ee
\end{enumerate}
The proof of \cref{thm:regular:interior} in \cref{proof:Appendix} guarantees the existence of  $c_{1,1,d}$ satisfying \eqref{regular:ckld}.
The proof of \cref{thm:regular:interior} further guarantees that the above unique 9-point FDM
must satisfy \eqref{sign:condition} and \eqref{sum:condition} for any mesh size $h>0$ to achieve the M-matrix property and achieve the sixth-order consistency at regular points.
For the special case that $a^{(m,n)}=0$ for all $2\le m+n\le 6$,
all the above nine unique stencil coefficients $C_{k,\ell}:=\sum_{p=0}^{7} c_{k,\ell,p}h^p$ in \cref{thm:regular:interior} are explicitly presented
in \eqref{particular:0}-\eqref{particular:6}
with the sixth-order consistency (but only second order if $a^{(m,n)}\ne 0$ for some $2\le m+n\le 6$) and satisfying the sign condition \eqref{sign:condition} and the sum condition \eqref{sum:condition} for any mesh size $h>0$.

For the sake of better readability,  the corresponding 4-point and 6-point FDMs on the boundary $\partial \Omega$ for various boundary conditions and their proofs are provided in \cref{Appendix:Boundary}.

\subsection{$13$-point stencils at irregular points}\label{hybrid:Irregular:points}

	We now discuss how to construct a $13$-point FDM with the fifth-order consistency at irregular points and derive the recursive solver to obtain stencil coefficients effectively.	
Let $(x_i,y_j)$ be an irregular point (i.e., both $d_{i,j}^+$ and $d_{i,j}^-$ are nonempty, see the left panel of \cref{fig:transfer}) and choose the base point $(x^*_i,y^*_j)\in \Gamma \cap (x_i-h,x_i+h)\times (y_j-h,y_j+h)$.
	By \eqref{base:pt},
	\be \label{base:pt:gamma}
		x_i^*=x_i-v_0h  \quad \mbox{and}\quad y_j^*=y_j-w_0h  \quad \mbox{with}\quad
		-1<v_0, w_0<1 \quad \mbox{and}\quad (x_i^*,y_j^*)\in \Gamma.
	\ee
	Let $a_{\pm}$, $u_{\pm}$ and $f_{\pm}$ represent the diffusion coefficient $a$, the solution $u$ and the source term $f$ in $\Omega_{\pm}$.
	Similarly to \eqref{ufmn2}, we define
	\begin{align*}
		& a_{\pm}^{(m,n)}:=\frac{\partial^{m+n} a_{\pm}}{ \partial^m x \partial^n y}(x^*_i,y^*_j),\qquad u_{\pm}^{(m,n)}:=\frac{\partial^{m+n} u_{\pm}}{ \partial^m x \partial^n y}(x^*_i,y^*_j),\qquad f_{\pm}^{(m,n)}:=\frac{\partial^{m+n} f_{\pm}}{ \partial^m x \partial^n y}(x^*_i,y^*_j).
	\end{align*}
	Similarly to \eqref{u:approx}, we have
	\be
	\begin{split}	 \label{u:approx:ir:key:2}
		u_\pm (x+x_i^*,y+y_j^*)
		& =\sum_{(m,n)\in \ind_{M}^{1}}
		u_\pm^{(m,n)} G^{\pm}_{M,m,n}(x,y) +\sum_{(m,n)\in \ind_{M-2}}
		f_\pm ^{(m,n)} H^{\pm}_{M,m,n}(x,y)+\bo(h^{M+1}),
	\end{split}
    \ee
	for $x,y\in (-2h,2h)$, where $\ind_{M-2}$ and $\ind_{M}^{1}$ are defined in \eqref{Sk} and  \eqref{indV12} respectively, $G^{\pm}_{M,m,n}(x,y)$ and $H^{\pm}_{M,m,n}(x,y)$ are obtained by  replacing $\{a^{(m,n)}: (m,n) \in \ind_{M}\}$ by $\{a_{\pm}^{(m,n)}: (m,n) \in \ind_{M}\}$  and replacing $M$ by $M-1$ in \eqref{Gmn} and \eqref{Hmn}. For the sake of presentation, we can replace $a^{u}_{p,q,m,n}$ by $a^{u_{\pm}}_{p,q,m,n}$ in \eqref{Gmn} and replace $a^{f}_{p,q,m,n}$ by $a^{f_{\pm}}_{p,q,m,n}$  in \eqref{Hmn}.
Near the point $(x_i^*,y_j^*)\in \Gamma$, the parametric equation $(r(t),s(t))$ of the interface $\Gamma$ can be written as:
	\be \label{parametric}
	x=r(t),\quad y=s(t), \quad	 x^*_i=r(t_k^*),\quad y^*_j=s(t_k^*), \quad (r'(t_k^*))^2+(s'(t_k^*))^2\ne0 \quad \mbox{for some } k\in \N,
	\ee
	where $r(t)$ and $s(t)$ are smooth functions. Similarly to the definition of the  9-point compact stencil in \eqref{Compact:Set}, we define the following extra 4-point set for the 13-point scheme (see \cref{Ckl:13:points}):
	 \be\label{13point:Set}
	\begin{split}
		& e_{i,j}^+:=\{(k,\ell) \; : \;
		(k,\ell)\in \{(-2,0),(2,0),(0,-2),(0,2)\}, \ \psi(x_i+kh, y_j+\ell h)> 0\},\\
		& e_{i,j}^-:=\{(k,\ell) \; : \;
		(k,\ell)\in  \{(-2,0),(2,0),(0,-2),(0,2)\}, \ \psi(x_i+kh, y_j+\ell h)\le 0\}.
	\end{split}
	\ee

In the next \cref{thm:tranmiss:interface}, we present the transmission equation \eqref{tranmiss:cond} to transfer $u_{-}$ to $u_{+}$ at an irregular point (see \cref{fig:transfer}). Furthermore, we provide some results of the transmission coefficients $T^{u_{+}}_{m',n',m,n}$ in \eqref{tranmiss:cond} which are used to develop an efficient recursive way to obtain the stencil coefficients $\{C_{k,\ell}\}$ of the 13-point scheme with the fifth-order consistency in \cref{thm:13point:scheme}.
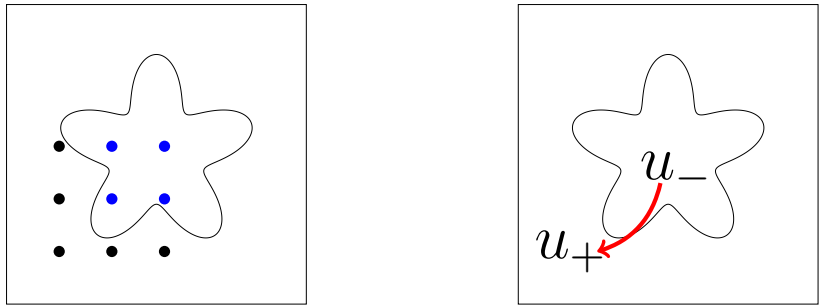
\begin{figure}[htbp]
	\centering	
	\begin{subfigure}[b]{0.3\textwidth}
		\hspace{-0.2cm}
		\begin{tikzpicture}[scale = 0.7]
				\begin{axis}[
				xmin=-3, xmax=3,
				ymin=-3, ymax=3,
				axis equal image,
				xtick=\empty,
				ytick=\empty]
				\addplot[data cs=polar,black,domain=0:360,samples=360,smooth] (x,{(1.5+0.5*sin(5*\x)) });
			\end{axis}
		    \node at (1,3)[circle,fill,inner sep=1.5pt,color=black]{};
			\node at (1,2)[circle,fill,inner sep=1.5pt,color=black]{};
			\node at (1,1)[circle,fill,inner sep=1.5pt,color=black]{};
			\node at (2,3)[circle,fill,inner sep=1.5pt,color=blue]{};
			\node at (2,2)[circle,fill,inner sep=1.5pt,color=blue]{};
			\node at (2,1)[circle,fill,inner sep=1.5pt,color=black]{};
		    \node at (3,3)[circle,fill,inner sep=1.5pt,color=blue]{};
			\node at (3,2)[circle,fill,inner sep=1.5pt,color=blue]{};
			\node at (3,1)[circle,fill,inner sep=1.5pt,color=black]{};
		\end{tikzpicture}
	\end{subfigure}		
	\begin{subfigure}[b]{0.3\textwidth}
		\hspace{1cm}
		\begin{tikzpicture}[scale = 0.7]
				\begin{axis}[
	xmin=-3, xmax=3,
	ymin=-3, ymax=3,
	axis equal image,
	xtick=\empty,
	ytick=\empty]
	\addplot[data cs=polar,black,domain=0:360,samples=360,smooth] (x,{(1.5+0.5*sin(5*\x)) });
\end{axis}

			\node (A)[scale=2] at (1,1) {$u_{+}$};
			\node (A)[scale=2] at (3,2.5) {$u_{-}$};
			\draw[ultra thick, red][->] (2.7,2.3) to[bend left] (1.5,1);
		\end{tikzpicture}
	\end{subfigure}
		\caption
	{An example for an irregular point (left) and the illustration for the transmission equation \eqref{tranmiss:cond} in \cref{thm:tranmiss:interface} (right).}
	\label{fig:transfer}
\end{figure}
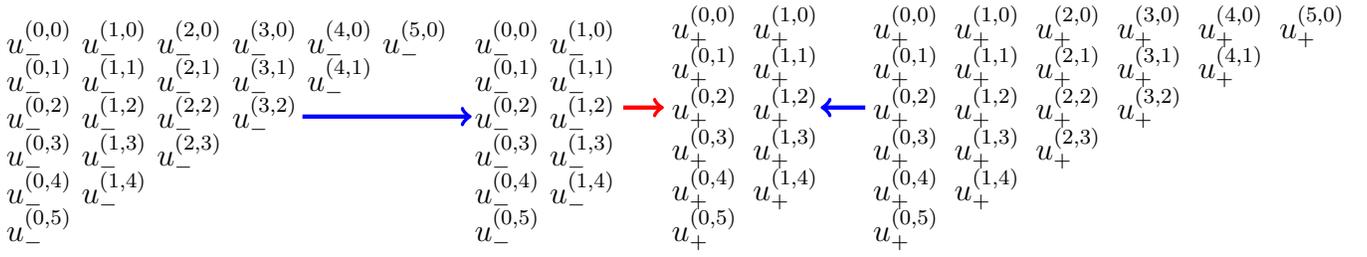
\begin{figure}[h]
	\centering
			\hspace{-30mm}
	\begin{subfigure}[b]{0.23\textwidth}
		\begin{tikzpicture}[scale = 0.50]
			\node (A) at (0,7) {$u_-^{(0,0)}$};
			\node (A) at (0,6) {$u_-^{(0,1)}$};
			\node (A) at (0,5) {$u_-^{(0,2)}$};
			\node (A) at (0,4) {$u_-^{(0,3)}$};
			\node (A) at (0,3) {$u_-^{(0,4)}$};
			\node (A) at (0,2) {$u_-^{(0,5)}$};
			\node (A) at (2,7) {$u_-^{(1,0)}$};
			\node (A) at (2,6) {$u_-^{(1,1)}$};
			\node (A) at (2,5) {$u_-^{(1,2)}$};
			\node (A) at (2,4) {$u_-^{(1,3)}$};
			\node (A) at (2,3) {$u_-^{(1,4)}$};
			\node (A) at (4,7) {$u_-^{(2,0)}$};
			\node (A) at (4,6) {$u_-^{(2,1)}$};
			\node (A) at (4,5) {$u_-^{(2,2)}$};
			\node (A) at (4,4) {$u_-^{(2,3)}$};
			\node (A) at (6,7) {$u_-^{(3,0)}$};
			\node (A) at (6,6) {$u_-^{(3,1)}$};
			\node (A) at (6,5) {$u_-^{(3,2)}$};
			\node (A) at (8,7) {$u_-^{(4,0)}$};
			\node (A) at (8,6) {$u_-^{(4,1)}$};
			\node (A) at (10,7) {$u_-^{(5,0)}$};
			\draw[ultra thick, blue][->] (7,5) to[left] (11.5,5);     	
		\end{tikzpicture}
	\end{subfigure}
	\begin{subfigure}[b]{0.23\textwidth}
		\hspace{18mm}
		\begin{tikzpicture}[scale = 0.50]
			\node (A) at (0,7) {$u_-^{(0,0)}$};
			\node (A) at (0,6) {$u_-^{(0,1)}$};
			\node (A) at (0,5) {$u_-^{(0,2)}$};
			\node (A) at (0,4) {$u_-^{(0,3)}$};
			\node (A) at (0,3) {$u_-^{(0,4)}$};
			\node (A) at (0,2) {$u_-^{(0,5)}$};
			\node (A) at (2,7) {$u_-^{(1,0)}$};
			\node (A) at (2,6) {$u_-^{(1,1)}$};
			\node (A) at (2,5) {$u_-^{(1,2)}$};
			\node (A) at (2,4) {$u_-^{(1,3)}$};
			\node (A) at (2,3) {$u_-^{(1,4)}$};	      	
		\end{tikzpicture}
		\end{subfigure}
			\hspace{-4mm}
		\begin{subfigure}[b]{0.23\textwidth}
		\begin{tikzpicture}[scale = 0.54]
			\node (A) at (0,7) {$u_+^{(0,0)}$};
			\node (A) at (0,6) {$u_+^{(0,1)}$};
			\node (A) at (0,5) {$u_+^{(0,2)}$};
			\node (A) at (0,4) {$u_+^{(0,3)}$};
			\node (A) at (0,3) {$u_+^{(0,4)}$};
			\node (A) at (0,2) {$u_+^{(0,5)}$};
			\node (A) at (2,7) {$u_+^{(1,0)}$};
			\node (A) at (2,6) {$u_+^{(1,1)}$};
			\node (A) at (2,5) {$u_+^{(1,2)}$};
			\node (A) at (2,4) {$u_+^{(1,3)}$};
			\node (A) at (2,3) {$u_+^{(1,4)}$};	
				\draw[ultra thick, red][->] (-2,5) to[left] (-1,5);     	
		\end{tikzpicture}
	\end{subfigure}
			\hspace{-18mm}
		\begin{subfigure}[b]{0.23\textwidth}
		\begin{tikzpicture}[scale = 0.54]
			\node (A) at (0,7) {$u_+^{(0,0)}$};
			\node (A) at (0,6) {$u_+^{(0,1)}$};
			\node (A) at (0,5) {$u_+^{(0,2)}$};
			\node (A) at (0,4) {$u_+^{(0,3)}$};
			\node (A) at (0,3) {$u_+^{(0,4)}$};
			\node (A) at (0,2) {$u_+^{(0,5)}$};
			\node (A) at (2,7) {$u_+^{(1,0)}$};
			\node (A) at (2,6) {$u_+^{(1,1)}$};
			\node (A) at (2,5) {$u_+^{(1,2)}$};
			\node (A) at (2,4) {$u_+^{(1,3)}$};
			\node (A) at (2,3) {$u_+^{(1,4)}$};
			\node (A) at (4,7) {$u_+^{(2,0)}$};
			\node (A) at (4,6) {$u_+^{(2,1)}$};
			\node (A) at (4,5) {$u_+^{(2,2)}$};
			\node (A) at (4,4) {$u_+^{(2,3)}$};
			\node (A) at (6,7) {$u_+^{(3,0)}$};
			\node (A) at (6,6) {$u_+^{(3,1)}$};
			\node (A) at (6,5) {$u_+^{(3,2)}$};
			\node (A) at (8,7) {$u_+^{(4,0)}$};
			\node (A) at (8,6) {$u_+^{(4,1)}$};
			\node (A) at (10,7) {$u_+^{(5,0)}$};
			\draw[ultra thick, blue][->] (-1,5) to[left] (-2.1,5);     	
		\end{tikzpicture}
	\end{subfigure}
	\caption
	{The illustration for the transmission equation \eqref{tranmiss:cond} in \cref{thm:tranmiss:interface} with $M=5$.}
	\label{fig:transfer:detail}
\end{figure}
	 \begin{theorem}\label{thm:tranmiss:interface}
		Let $u$ be the solution to the elliptic interface problem \eqref{Qeques2} and let  $\Gamma$ be parameterized near $(x_i^*,y_j^*)$ by \eqref{parametric}.
		Then the transmission equation for $u_{\pm}$ at $(x_i^*,y_j^*)$ (see \cref{fig:transfer,fig:transfer:detail}) is
\be \label{tranmiss:cond}
\begin{split}
	u_-^{(m',n')}&=\sum_{  \substack{ (m,n)\in \ind_{M}^{1} \\  m+n \le m'+n'} }T^{u_{+}}_{m',n',m,n}u_+^{(m,n)}+\sum_{(m,n)\in \ind_{M-2}} \left(T^+_{m',n',m,n} f_+^{(m,n)}
	+ T^-_{m',n',m,n} f_{-}^{(m,n)}\right)\\
	&+\sum_{p=0}^M T^{g}_{m',n',p} g^{(p)}
	+\sum_{p=0}^{M-1} T^{g_\Gamma}_{m',n',p} g_{\Gamma}^{(p)},\qquad \forall\; (m',n')\in \ind_{M}^{1},
\end{split}
\ee		
			\be\label{gp:gGammap}
			\begin{split}
&	g^{(p)}:=\frac{1}{p!} \frac{d^p}{dt^p} (g(t))
	\Big|_{t=t_k^*},\qquad g^{(p)}_{\Gamma}:=\frac{1}{p!} \frac{d^p}{dt^p} \left(g_\Gamma(t)\sqrt{(r'(t))^2+(s'(t))^2}\right)
	\Big|_{t=t_k^*},  \\
	&	r^{(p)}:= \frac{d^p}{dt^p} (r(t))
	\Big|_{t=t_k^*}, \qquad s^{(p)}:= \frac{d^p}{dt^p} (s(t))
	\Big|_{t=t_k^*},
	\end{split}
	\ee
		where all the transmission coefficients $T^{u_+}, T^{\pm}, T^{\gd}, T^{\gn}$ are uniquely determined by
		$r^{(p)}$, $s^{(p)}$  for $p=0,\ldots,M$ and $\{a_{\pm}^{(m,n)}: (m,n)\in \ind_{M-1}\}$.	Moreover, let  $T^{u_+}_{m',n',m,n}$ be the transmission coefficient of $u_+^{(m,n)}$ in \eqref{tranmiss:cond} with $(m,n)\in \ind_{M}^{1}$, $m+n=m'+n'$ and $(m',n')\in \ind_{M}^{1}$. Then $T^{u_+}_{m',n',m,n}$ only depends on $a_{\pm}^{(0,0)}$ and $(r'(t_k^*), s'(t_k^*))$ of \eqref{parametric}. Particularly, 	
		\be\label{T0000}
T^{u_+}_{0,0,0,0}=1
\quad \mbox{and}\quad
T^{u_+}_{m',n',0,0}=0 \quad \mbox{if }  (m',n')\ne (0,0).
		\ee	
	\end{theorem}

By the proof of \cref{thm:tranmiss:interface},
the transmission coefficients $T^{u_+}, T^{\pm}, T^{\gd}, T^{\gn}$ are uniquely determined by solving  \eqref{interface:flux:1} and \eqref{interface:u:1} recursively in the order $p=1,\dots,M$,
 and  $u_{-}^{(0,0)}
=u_{+}^{(0,0)}-g^{(0)}$.
	Next, we provide the 13-point FDM  with the fifth-order consistency for interior irregular points.

\begin{figure}[h]
	\centering
	\hspace{-0.4cm}	
	\begin{subfigure}[b]{0.3\textwidth}
		\begin{tikzpicture}[scale = 2.2]
			\draw[help lines,step = 0.5]
			(-1,-1) grid (1,1);
			\draw[line width=1.5pt, blue]  plot [smooth,tension=0.8]
			coordinates {(-1,-0.4) (-0.5,0.15) (0,0.2) (0.3,1)};
			\node (A) at (0.3,-0.3) {\color{red}{$a_-$}};
			\node (A) at (-0.7,0.3) {\color{red}{$a_+$}};
			\node (A) at (0.37,0.3) {{\footnotesize{$\left[u\right]=g$}}};
			\node (A) at (0.51,0.1) { {\footnotesize{ $\left[{\color{red}{a}}\nabla  u \cdot \nv \right]=g_{\Gamma}$}}};
			\node (A) at (-0.7,-0.2) {\color{blue}$\Gamma$};
		\end{tikzpicture}
	\end{subfigure}
	\begin{subfigure}[b]{0.3\textwidth}
		\hspace{0.6cm}
		\begin{tikzpicture}[scale = 2.2]
			\draw[help lines,step = 0.5]
			(-1,-1) grid (1,1);
			\node at (-0.5,0.5)[circle,fill,inner sep=2pt,color=black]{};
			\node at (-0.5,0)[circle,fill,inner sep=2pt,color=black]{};
			\node at (-0.5,-0.5)[circle,fill,inner sep=2pt,color=black]{};
			\node at (0,0.5)[circle,fill,inner sep=2pt,color=black]{};
			\node at (0,0)[circle,fill,inner sep=2pt,color=red]{};
			\node at (0,-0.5)[circle,fill,inner sep=2pt,color=black]{};
			\node at (0.5,0.5)[circle,fill,inner sep=2pt,color=black]{};
			\node at (0.5,0)[circle,fill,inner sep=2pt,color=black]{};
			\node at (0.5,-0.5)[circle,fill,inner sep=2pt,color=black]{};
			\node (A) at (-0.31,0.6) {{\tiny{$C_{-1,1}$}}};
			\node (A) at (-0.31,0.1) {{\tiny{$C_{-1,0}$}}};
			\node (A) at (-0.26,-0.4) {{\tiny{$C_{-1,-1}$}}};
			\node (A) at (0.14,0.6) {{\tiny{$C_{0,1}$}}};
			\node (A) at (0.14,0.1) {{\tiny{$C_{0,0}$}}};
			\node (A) at (0.18,-0.4) {{\tiny{$C_{0,-1}$}}};
			\node (A) at (0.65,0.6) {{\tiny{$C_{1,1}$}}};
			\node (A) at (0.65,0.1) {{\tiny{$C_{1,0}$}}};
			\node (A) at (0.68,-0.4) {{\tiny{$C_{1,-1}$}}};
			\draw[line width=1.5pt, blue]  plot [smooth,tension=0.8]
			coordinates {(-1,-0.4) (-0.5,0.15) (0,0.2) (0.3,1)};
			\node (A) at (-0.7,-0.2) {\color{blue}$\Gamma$};
		\end{tikzpicture}
	\end{subfigure}	
	\begin{subfigure}[b]{0.3\textwidth}
		\hspace{1.1cm}
		\vspace{-0.1cm}
		\begin{tikzpicture}[scale = 2.2]
			\draw[help lines,step = 0.5]
			(-1,-1) grid (1,1);
			\node at (-1,0)[circle,fill,inner sep=2pt,color=black]{};
			\node at (0,-1)[circle,fill,inner sep=2pt,color=black]{};
			\node at (0,1)[circle,fill,inner sep=2pt,color=black]{};
			\node at (1,0)[circle,fill,inner sep=2pt,color=black]{};
			\node (A) at (-0.8,0.1) {{\tiny{$C_{-2,0}$}}};
			\node (A) at (0.17,-0.9) {{\tiny{$C_{0,-2}$}}};
			\node (A) at (0.14,0.9) {{\tiny{$C_{0,2}$}}};
			\node (A) at (0.85,0.1) {{\tiny{$C_{2,0}$}}};
			\draw[line width=1.5pt, blue]  plot [smooth,tension=0.8]
			coordinates {(-1,-0.4) (-0.5,0.15) (0,0.2) (0.3,1)};
			\node (A) at (-0.7,-0.2) {\color{blue}$\Gamma$};
		\end{tikzpicture}
	\end{subfigure}
	\caption
	{ The illustration for the 13-point scheme \eqref{13point:scheme} in \cref{thm:13point:scheme}. }
	\label{Ckl:13:points}
\end{figure}
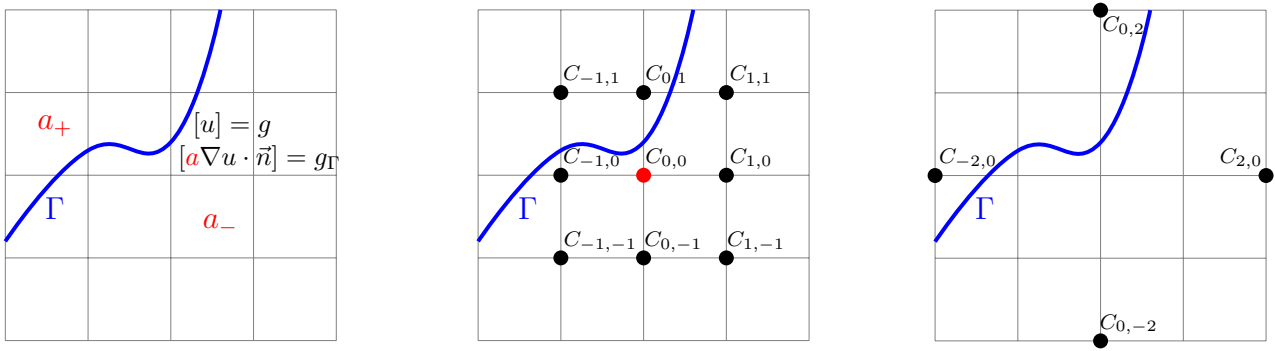	
	 \begin{theorem}\label{thm:13point:scheme}
Let $(x_i,y_j)=(x_i^*+v_0h,y_j^*+w_0h)$ be an interior irregular point with
		$(x_i^*,y_j^*)\in \Gamma$ and $-1<v_0,w_0<1$.
		Then
		the following 13-point scheme centered at  $(x_i,y_j)$ (see \cref{Ckl:13:points}):
		\be
		\begin{split}\label{13point:scheme}
		h^{-1}	\mathcal{L}_hu_h  & :=
h^{-1}	\Big( \sum_{k=-1}^{1}\sum_{\ell =-1}^{1} C_{k,\ell}(u_{h})_{i+k,j+\ell}  +  \sum_{k=-2,2} C_{k,0}(u_{h})_{i+k,j}+\sum_{\ell=-2,2} C_{0,\ell }(u_{h})_{i,j+\ell} \Big)\\
			&=\sum_{(m,n)\in \ind_{3} } f_+^{(m,n)}J^{+}_{m,n} + \sum_{(m,n)\in \ind_{3}} f_-^{(m,n)}J^{-}_{m,n} +\sum_{p=0}^5 \gd^{(p)}J^{\gd}_{p} + \sum_{p=0}^4 \gn^{(p)}J^{\gn}_{p},
		\end{split}
	\ee
		achieves the fifth-order consistency for $\left[u\right]=g$ and $\left[a\nabla  u \cdot \nv \right]=g_{\Gamma}$ at $(x_i^*,y_j^*)\in \Gamma$, where $\{C_{k,\ell}\}_{k,\ell=-1,0,1}$, $\{C_{k,0}\}_{k=-2,2}$, and $\{C_{0,\ell}\}_{\ell=-2,2}$ are given by
$C_{k,\ell}:=\sum_{p=0}^{5} c_{k,\ell,p}h^p$ with $c_{k,\ell,p}\in \R$ such that $\{c_{k,\ell,p}\}_{k,\ell=-1,0,1}$, $\{c_{k,0,p}\}_{k=-2,2}$, and $\{c_{0,\ell,p}\}_{\ell=-2,2}$ with $p=0,\ldots,5$ is any nontrivial solution of the linear system induced by the following equations
		\be\label{EQ:irregular}
	I^+_{m,n}+	\sum_{  \substack{ (m',n')\in \ind_{5}^{1} \\  m'+n' \ge m+n}} I^-_{m',n'} T^{u_+}_{m',n',m,n}=\bo(h^{6}), \quad \mbox{ for all }\; (m,n)\in \ind_{5}^1,
		\ee
		with
		\be\label{right:irregular}
			\begin{split}
				& I^{\pm}_{m,n}:=\sum_{(k,\ell)\in \tilde{d}_{i,j}^\pm}
				C_{k,\ell} G^{\pm}_{5,m,n}(v_1h,w_1 h), \quad J^{\pm}_{m,n}:=
				 J_{m,n}^{\pm,0}+J^{\pm,T}_{m,n},\quad v_1=v_0+k, \qquad w_1=w_0+\ell,\\
				& J^{\pm,0}_{m,n}:=h^{-1} \sum_{(k,\ell)\in \tilde{d}_{i,j}^\pm} C_{k,\ell} H^{\pm}_{5,m,n}(v_1h,w_1h), \qquad J^{\pm,T}_{m,n}:= h^{-1}
				\sum_{  \substack{ (m',n')\in \ind_{5}^{1} }} I^{-}_{m',n'} T^{\pm}_{m',n',m,n},\\
				& J^{\gd}_{p}:= h^{-1}
				\sum_{ \substack{ (m',n')\in \ind_{5}^{1} }} I^{-}_{m',n'} T^{\gd}_{m',n',p}, \qquad
				J^{\gn}_{p}:= h^{-1}
				\sum_{ \substack{ (m',n')\in \ind_{5}^{1} }} I^{-}_{m',n'} T^{\gn}_{m',n',p}, \qquad \tilde{d}_{i,j}^\pm =d_{i,j}^\pm \cup e_{i,j}^\pm.
			\end{split}
		\ee
	\end{theorem}
%
\noindent	By the symbolic calculation,
the linear system in \eqref{EQ:irregular} always has nontrivial solutions. \\

\noindent
\textbf{An efficient way to compute $\{C_{k,\ell}\}$ in \cref{thm:13point:scheme}:}
From the proof of \cref{thm:13point:scheme},  we observe that \eqref{EQ:irregular}  can be equivalently expressed as
{\footnotesize{
\be\label{irregular:simplify3}
\begin{split}
	& \sum_{(k,\ell)\in \tilde{d}_{i,j}^+}c_{k,\ell,d} G_{m,n}(v_1,w_1)+	 \sum_{(k,\ell)\in \tilde{d}_{i,j}^-}  c_{k,\ell,d} \sum_{\substack{(m',n')\in \ind_{5}^1 \\  m'+n'=m+n} }
	T^{u_+}_{m',n',m,n} G_{m',n'}(v_1,w_1) \\
	&=-  \sum_{(k,\ell)\in \tilde{d}_{i,j}^+} 	 \sum_{s=0}^{d-1} c_{k,\ell,s} 	
	\sum_{ \substack{ (p,q)\in \ind_{5}^2 \setminus \ind_{m+n}^2 \\ p+q=m+n+d-s}} \frac{a^{u_{+}}_{p,q,m,n}}{p!q!} v_1^{p}w_1^{q}  -	 \sum_{(k,\ell)\in \tilde{d}_{i,j}^-} 	 \sum_{s=0}^{d-1} c_{k,\ell,s}	 \sum_{\substack{(m',n')\in \ind_{5}^1 \\ m'+n'=m+n+d-s} }
	T^{u_+}_{m',n',m,n} G_{m',n'}(v_1,w_1) \\
	&\quad -	\sum_{(k,\ell)\in \tilde{d}_{i,j}^-} 	 \sum_{s=0}^{d-1} c_{k,\ell,s}	 \sum_{\substack{(m',n')\in \ind_{5}^1 \\ m'+n'\ge m+n } }
	T^{u_+}_{m',n',m,n}  \sum_{ \substack{ (p,q)\in \ind_{5}^2 \setminus \ind_{m'+n'}^2 \\ p+q=m+n+d-s}} \frac{a^{u_{-}}_{p,q,m',n'}}{p!q!}v_1^{p}w_1^{q}, \quad \mbox{for all } (m,n)\in \ind_{5-d}^1,  0\le d \le 5.
\end{split}
\ee}}
%
%
%
The system of linear equations in \eqref{irregular:simplify3} can be further equivalently expressed as follows:
\be\label{Matrix:Form:irregular}
\begin{split}
&A_0C_0=\textbf{0},\quad
A_1 C_1=b_1,\quad A_2 C_2=b_2,\quad A_3 C_3=b_3,\quad  A_4 C_4=b_4,\quad A_5 C_5=0\quad \mbox{with }\\
&C_d=\left(c_{-1,-1,d},
c_{-1,0,d},
c_{-1,1,d},
c_{0,-1,d},
c_{0,0,d},
c_{0,1,d},
c_{1,-1,d},
c_{1,0,d},
c_{1,1,d},
c_{-2,0,d},
c_{2,0,d},
c_{0,-2,d},
c_{0,2,d}\right)^{\textsf{T}},
\end{split}
\ee
where $b_d, d=1,\ldots,4$ depend on $\{ C_i\}_{i=0}^{d-1}$, $\{a_{\pm}^{(i,j)}\}_{(i,j) \in \ind_{4}}$,  $\{ (r^{(p)}, s^{(p)})\}_{ 1\le p\le 5}$ in  \eqref{gp:gGammap} and $(v_0,w_0)$ in \eqref{base:pt:gamma}.
\cref{thm:tranmiss:interface} shows that $T^{u_+}_{m',n',m,n}$ only depends on $a_{\pm}^{(0,0)}$ and $(r'(t_k^*), s'(t_k^*))$ in \eqref{parametric} if $(m,n)\in \ind_{M}^{1}$, $m+n=m'+n'$ and $(m',n')\in \ind_{M}^{1}$. So the left-hand side of \eqref{irregular:simplify3} implies
 all matrices $A_0, \ldots, A_5$ in \eqref{Matrix:Form:irregular} only depend on $a_{\pm}^{(0,0)}$, $(r'(t_k^*), s'(t_k^*))$ in \eqref{parametric} and $(v_0,w_0)$ in \eqref{base:pt:gamma} such that
$A_0$ is a $11\times 13$ matrix and all other matrices $A_1,\ldots, A_5$ are sub-matrices of $A_0$ by deleting some rows of $A_0$ as follows:
\be\label{All:Ad:irregular}
\begin{split}
&A_1=A_0(1:9, :),\quad A_2=A_0(1:7,:),\quad
A_3=A_0(1:5,:),\quad A_4=A_0(1:3,:),\quad\\
&A_5=A_0(1,:)=(1 \quad 1 \quad  1 \quad 1 \quad 1 \quad  1 \quad 1 \quad 1 \quad  1 \quad 1 \quad 1 \quad  1 \quad 1).
\end{split}
\ee
%
%
\noindent
\textbf{The explicit expression of $A_0$ in \eqref{Matrix:Form:irregular}--\eqref{All:Ad:irregular}:}
Recall that $v_1=v_0+k$, $w_1=w_0+\ell$ in \eqref{right:irregular}. Let $\xi_1=\tfrac{1}{2}(w_{1}^2-v_{1}^2)$, $\xi_2=v_{1}w_{1}$, $\xi_3=\tfrac{1}{6}w_{1}(w_{1}^2-3v_{1}^2)$, $\xi_4=\tfrac{1}{6}v_1(3w_{1}^2-v_{1}^2)$, $\xi_5=\tfrac{1}{24}(w_{1}^4-6v_{1}^2w_{1}^2+v_{1}^4)$, $\xi_6=\tfrac{1}{6}v_{1}w_{1}(w_{1}^2-v_{1}^2)$, $\xi_7=\tfrac{1}{120}w_{1}(w_{1}^4-10v_{1}^2w_{1}^2+5v_{1}^4)$, $\xi_8=\tfrac{1}{120}v_{1}(v_{1}^4-10v_{1}^2w_{1}^2+5w_{1}^4)$, and define two $11\times 1$ matrices  as follows
{\footnotesize{
\be\label{Apm:irregular}
\textsf{A}^{+}_{k,\ell}:=\begin{pmatrix}
	1\\
	w_{1}\\
	v_{1}\\
	\xi_1\\
	\xi_2\\
	\xi_3\\
	\xi_4\\
	\xi_5\\
	\xi_6\\
	\xi_7\\
	\xi_8\\
\end{pmatrix}, \qquad \textsf{A}^{-}_{k,\ell}:=\begin{pmatrix}
1\\
T^{u_{+}}_{0,1,0,1}w_{1}+T^{u_{+}}_{1,0,0,1}v_{1}\\
T^{u_{+}}_{0,1,1,0}w_{1}+T^{u_{+}}_{1,0,1,0}v_{1}\\
T^{u_{+}}_{0,2,0,2}\xi_{1}+T^{u_{+}}_{1,1,0,2}\xi_{2}\\
T^{u_{+}}_{0,2,1,1}\xi_{1}+T^{u_{+}}_{1,1,1,1}\xi_{2}\\
T^{u_{+}}_{0,3,0,3}\xi_{3}+T^{u_{+}}_{1,2,0,3}\xi_{4}\\
T^{u_{+}}_{0,3,1,2}\xi_{3}+T^{u_{+}}_{1,2,1,2}\xi_{4}\\
T^{u_{+}}_{0,4,0,4}\xi_{5}+T^{u_{+}}_{1,3,0,4}\xi_{6}\\
T^{u_{+}}_{0,4,1,3}\xi_{5}+T^{u_{+}}_{1,3,1,3}\xi_{6}\\
T^{u_{+}}_{0,5,0,5}\xi_{7}+T^{u_{+}}_{1,4,0,5}\xi_{8}\\
T^{u_{+}}_{0,5,1,4}\xi_{7}+T^{u_{+}}_{1,4,1,4}\xi_{8}\\
\end{pmatrix},
\ee}}
where $T^{u_{+}}$ is the coefficient of $u_{+}$ in \eqref{tranmiss:cond}. By \cref{thm:tranmiss:interface}, all $T^{u_{+}}$ in \eqref{Apm:irregular}  only depend on $a_{\pm}^{(0,0)}$ and $(r'(t_k^*), s'(t_k^*))$.
Now, by the left-hand side of \eqref{irregular:simplify3}, the $11\times 13$ matrix $A_0$ for the irregular point in \cref{Ckl:13:points}  is
{\small{
\be\label{A0:example}
A_0=(\textsf{A}^{-}_{-1,-1} \quad \textsf{A}^{-}_{-1,0}\quad \textsf{A}^{+}_{-1,1} \quad \textsf{A}^{-}_{0,-1} \quad \textsf{A}^{-}_{0,0} \quad \textsf{A}^{+}_{0,1} \quad \textsf{A}^{-}_{1,-1} \quad \textsf{A}^{-}_{1,0} \quad \textsf{A}^{-}_{1,1} \quad \textsf{A}^{+}_{-2,0} \quad \textsf{A}^{-}_{2,0} \quad \textsf{A}^{-}_{0,-2} \quad \textsf{A}^{+}_{0,2}).
\ee}}
The matrix $A_0$ for other irregular points can be obtained straightforwardly by using $\textsf{A}^{\pm}_{k,\ell}$ in \eqref{Apm:irregular}.

\noindent
\textbf{Stencil coefficients $\{C_{k,\ell}\}$ in \cref{thm:13point:scheme} for numerical tests:}
	To verify the 13-point scheme \eqref{13point:scheme} of  \cref{thm:13point:scheme} for numerical tests in \cref{sec:numerical}, we obtain a unique solution $\{C_{k,\ell}\}$ by
\begin{itemize}	
	\item choosing $(x_i^*,y_j^*)\in \Gamma$ to be the orthogonal projection of  $(x_i,y_j)$ (see the right panel of \cref{fig:3choices});
	\item setting $c_{0,0,0}=1$ and $C_5=\textbf{0}$;
\item first solving $A_0C_0=0$ for $C_0$, and then solving $A_d C_d=b_d$ in \eqref{Matrix:Form:irregular} recursively in the order $d=1,\ldots,4$ by the MATLAB Package $\texttt{mldivide}(\textsf{A},\textsf{b})$.
\end{itemize}
Note that if we use the MATLAB Package $\texttt{mldivide}(\textsf{A},\textsf{b})$ to
solve $\textsf{A}\textsf{x}=\textsf{b}$ with infinitely many solutions, then it automatically sets free parameters to be $0$.

\begin{figure}[htbp]
	\centering	
	\hspace{-0.9cm}
	\begin{subfigure}[b]{0.3\textwidth}
		\begin{tikzpicture}[scale = 2.1]
			\draw[help lines,step = 1]
			(1,1) grid (3,3);
			\node at (2,2)[circle,fill,inner sep=2pt,color=black]{};
			\node (A) at (0.9,1.8) {\color{blue}$\Gamma$};	
			\node (A) at (2.3,1.8) {$(x_i,y_j)$};
			\draw[line width=1.5pt, blue]  plot [smooth,tension=0.8]
			coordinates { (1,1.8) (2,2.7) (2.3,3)};
			\node (A) at (1.2,2.4) {$\Omega_{+}$};	
			\node (A) at (2.5,1.2) {$\Omega_{-}$};
			\node at (1.25,2)[circle,fill,inner sep=2pt,color=red]{};
			\node (A) at (1.5,1.8) {$(x_i^*,y_j^*)$};
			 \draw[decorate,decoration={brace,amplitude=2.5mm},xshift=0pt,yshift=10pt, thick] (1.25,1.65) -- node [black,midway,yshift=0.6cm]{} (2,1.65);
			\node (A) at (1.8,2.2) {$|v_0|h$};	
			\node (A) at (2.3,2.35) {$w_0=0$};
		\end{tikzpicture}
	\end{subfigure}
	\begin{subfigure}[b]{0.3\textwidth}
		\hspace{0.4cm}
		\begin{tikzpicture}[scale = 2.1]
			\draw[help lines,step = 1]
			(1,1) grid (3,3);
			\node at (2,2)[circle,fill,inner sep=2pt,color=black]{};
			\node (A) at (0.9,1.8) {\color{blue}$\Gamma$};	
			\node (A) at (2.3,1.8) {$(x_i,y_j)$};
			\draw[line width=1.5pt, blue]  plot [smooth,tension=0.8]
			coordinates { (1,1.8) (2,2.7) (2.3,3)};
			\node (A) at (1.2,2.4) {$\Omega_{+}$};	
			\node (A) at (2.5,1.2) {$\Omega_{-}$};
			\node at (2,2.7)[circle,fill,inner sep=2pt,color=red]{};
			\node (A) at (1.7,2.75) {$(x_i^*,y_j^*)$};	
			 \draw[decorate,decoration={brace,amplitude=2.5mm},xshift=0pt,yshift=10pt,thick] (2,2.35) -- node [black,midway,yshift=0.6cm]{} (2,1.65);
			\node (A) at (2.4,2.35) {$|w_0|h$};
			\node (A) at (1.6,1.9) {$v_0=0$};		
		\end{tikzpicture}
	\end{subfigure}
	\begin{subfigure}[b]{0.3\textwidth}
		\hspace{0.8cm}
		\begin{tikzpicture}[scale = 2.1]
			\draw[help lines,step = 1]
			(1,1) grid (3,3);
			\node at (2,2)[circle,fill,inner sep=2pt,color=black]{};
			\node (A) at (0.9,1.8) {\color{blue}$\Gamma$};	
			\node (A) at (2.3,1.8) {$(x_i,y_j)$};
			\draw[line width=1.5pt, blue]  plot [smooth,tension=0.8]
			coordinates { (1,1.8) (2,2.7) (2.3,3)};
			\node (A) at (1.2,2.4) {$\Omega_{+}$};	
			\node (A) at (2.5,1.2) {$\Omega_{-}$};
			\draw    (2,2) -- (1.67,2.4);
			\node at (1.67,2.4)[circle,fill,inner sep=2pt,color=red]{};
			\node (A) at (1.5,2.55) {$(x_i^*,y_j^*)$};	
			\draw    (1.67,2.4) -- (1.67,2);
			 \draw[decorate,decoration={brace,amplitude=2.5mm},xshift=0pt,yshift=10pt,thick] (2,2.05) -- node [black,midway,yshift=0.6cm]{} (2,1.65);
			\node (A) at (2.4,2.2) {$|w_0|h$};
			\draw    (1.67,2.4) --  (2,2.4);
			 \draw[decorate,decoration={brace,mirror,amplitude=2mm},xshift=0pt,yshift=10pt, thick] (1.67,1.65) -- node [black,midway,yshift=0.6cm]{} (2,1.65);
			\node (A) at (1.75,1.75) {$|v_0|h$};
		\end{tikzpicture}
	\end{subfigure}
	\caption
	{Three choices of $(x_i^*,y_j^*)$: $(x_i^*,y_j^*)=(x_i^*,y_j)$ (left), $(x_i^*,y_j^*)=(x_i,y_j^*)$ (middle) and $(x_i^*,y_j^*)\in \Gamma$ is the orthogonal projection of $(x_i,y_j) $ (right).}
	\label{fig:3choices}
\end{figure}
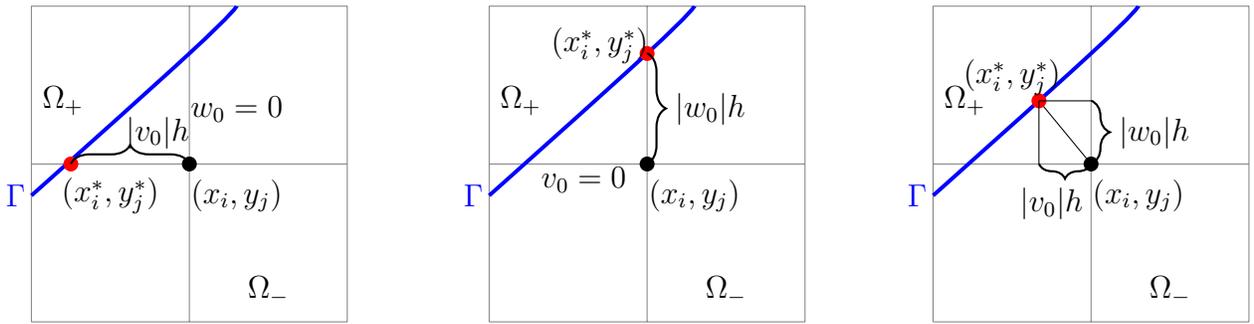

		\begin{remark}		
In this paper, we use $(x_i^*,y_j^*)\in \Gamma$ as the orthogonal projection of  $(x_i,y_j)$ (see the right panel of \cref{fig:3choices}). To obtain the orthogonal projection, we discretize the interface curve $\Gamma$ by a mesh of size $\frac{h}{2^4}$, and select the point on $\Gamma$ that is closest to   $(x_i,y_j)$. Theoretically, the choice of $(x_i^*,y_j^*) \in \Gamma$ within the stencil centered at $(x_i,y_j)$ will not affect the accuracy order. Two simple choices are $(x_i^*,y_j^*)=(x_i-v_0h, y_j)$ with $w_0=0$ or
$(x_i^*,y_j^*)=(x_i, y_j-w_0 h)$ with $v_0=0$, see left and middle panels of \cref{fig:3choices}. Among these two points, we can choose the one closest to the point $(x_i, y_j)$.
\end{remark}

\subsection{Estimate high order (partial) derivatives only using function values}\label{Estimate:derivatives}

To achieve the high order consistency for all grid points, our FDMs in \cref{subsec:regular,hybrid:Irregular:points,Appendix:Boundary} theoretically
use high order (partial) derivatives of the coefficient function $a$, the source term $f$, the interface curve $\Gamma$, two jump functions $g,g_{\Gamma}$, and functions on $\partial \Omega$. To avoid explicitly and symbolically computing such derivatives, numerically (but without losing accuracy and performance), we always use function values to estimate required high order (partial) derivatives in  this paper by the moving least-squares method in \cite{DLevin1998}.
Let $\textsf{Z}:=(\textsf{z}_1,\dots,\textsf{z}_{\textsf{K}})$, where $\textsf{K}\in \N$ and every $\textsf{z}_{\textsf{k}}$ is a 1D/2D point for $1\le \textsf{k}\le \textsf{K}$. For a 1D/2D point $\textsf{z}^*$, we define a $\textsf{K} \times \textsf{K}$ diagonal matrix  $\textsf{D}:=2 \textsf{diag}(\eta(\|\textsf{z}_1-\textsf{z}^*\|),\dots,\eta(\|\textsf{z}_{\textsf{K}}-\textsf{z}^*\|))$ with  $\eta(\textsf{r})=\exp(\textsf{r}^2/h^2)$,
and  define the space of polynomials of total degree $\le \textsf{M}$ as follows:
\[
\textsf{P}_\textsf{M}:=\begin{cases}
\textsf{span} \{ (x-\textsf{x}^*)^\textsf{m} :  0\le  \textsf{m}\le \textsf{M} \}, &\qquad \mbox{for the 1D case},\\
\textsf{span} \{ (x-\textsf{x}^*)^\textsf{m}(y-\textsf{y}^*)^\textsf{n} : 0\le \textsf{n}\le \textsf{M}-\textsf{m}  \text{ and } 0\le \textsf{m}\le \textsf{M}\} , &\qquad \mbox{for the 2D case}.
\end{cases}
\]
For the sake of better readability, we identify the linear space $\textsf{P}_\textsf{M}$ with a vector of the above basis elements $(\textsf{p}_1,\dots,\textsf{p}_\textsf{J})^{\textsf{T}}$ where $\textsf{J} \in \N$ is the dimension of $\textsf{P}_{\textsf{M}}$. Then we define a $\textsf{K} \times \textsf{J}$ matrix $\textsf{E}$ by
\[
\textsf{E}_{\textsf{k},\textsf{j}}:=\textsf{p}_\textsf{j}(\textsf{z}_\textsf{k}), \qquad \textsf{k}=1,\dots, \textsf{K}, \quad \textsf{j}=1,\dots, \textsf{J}.
\]
By \cite{DLevin1998}, the $\omega$th (partial) derivative of the 1D/2D function $\textsf{f}(\textsf{z})$ at the point $\textsf{z}^*$ can be approximated by
\be\label{est:deriva}
\textsf{f}^{(\omega)}(\textsf{z}^*)=(\textsf{f}(\textsf{z}_1),\dots,\textsf{f}(\textsf{z}_\textsf{K}))\textsf{D}^{-1}\textsf{E}(\textsf{E}^{\textsf{T}}\textsf{D}^{-1}\textsf{E})^{-1} (\textsf{p}_1^{(\omega)}(\textsf{z}^*),\dots,\textsf{p}_\textsf{J}^{(\omega)}(\textsf{z}^*))^{\textsf{T}}.
\ee
\noindent
\textbf{Details about our concrete choices of $\textsf{Z}, \textsf{P}_\textsf{M}, \textsf{z}^*$, and $(\textsf{x}^*,\textsf{y}^*)$ used in our numerical experiments:}
For numerical experiments in \cref{sec:numerical}, we shall employ the following settings to estimate high order (partial) derivatives using only function values:
\begin{itemize}	
	
	\item For 2D functions $a(x,y)$ and $f(x,y)$ used in \cref{thm:regular:interior}.	 Let $\textsf{Z}=\{(x_i\pm \textsf{i}\textsf{h},y_j\pm \textsf{j}\textsf{h}) : 0\le \textsf{i},\textsf{j}\le 2^2,  \textsf{h}=h/2^2\}$,  $\textsf{z}^*=(\textsf{x}^*,\textsf{y}^*)=(x_i,y_j)$. Then we use \eqref{est:deriva} with $\textsf{P}_\textsf{6}$ and  $\textsf{P}_\textsf{5}$ to approximate $\{a^{(m,n)}\}_{ (m,n)\in \ind_{6}}$ and $\{f^{(m,n)}\}_{ (m,n)\in  \ind_{5}}$,   respectively.

	\item For 2D functions $a(x,y)$ and $f(x,y)$ used in \cref{thm:13point:scheme}.	 Let $\textsf{Z}=\Omega_{\pm}\cap \{(x_i\pm \textsf{i}\textsf{h},y_j\pm \textsf{j}\textsf{h}) : 0\le \textsf{i},\textsf{j}\le 2^5,  \textsf{h}=h/2^5\}$,  $\textsf{z}^*=(x_i^*,y_j^*)$ and $(\textsf{x}^*,\textsf{y}^*)=(x_i,y_j)$. Then we use \eqref{est:deriva} with $\textsf{P}_\textsf{4}$ and  $\textsf{P}_\textsf{3}$ to approximate $\{a_{\pm}^{(m,n)} : (m,n)\in  \ind_{4}\}$ and $\{f_{\pm}^{(m,n)} : (m,n)\in  \ind_{3}\}$, respectively.

	\item For 1D functions $r(t),s(t),g(t)$ and $g_\Gamma(t)$ with $t\in \Gamma$ used in \cref{thm:13point:scheme}.
	
	\begin{itemize}	\item[1.] If the parametric equation $(r(t),s(t))$ of $\Gamma$ is equal to $(x,s(x))$.
		Let $\textsf{Z}=\{x_i^*\pm \textsf{i}\textsf{h} : 0\le \textsf{i}\le 5,  \textsf{h}=h/2^4\}$,  $\textsf{z}^*=\textsf{x}^*=x_i^*$. Then we use  \eqref{est:deriva} with $\textsf{P}_\textsf{6}$ to approximate  $\{ r^{(p)}\}_{p=0}^5$, $\{ s^{(p)}\}_{p=0}^5$, $\{ g^{(p)}\}_{p=0}^5$ and use \eqref{est:deriva} with $\textsf{P}_\textsf{5}$ to approximate $\{ g_{\Gamma}^{(p)}\}_{p=0}^4$. The corresponding procedure for  $(r(t),s(t))=(r(y),y)$ is straightforward.

		\item[2.] If  $(r(t),s(t))=(r(\theta),s(\theta))$ with $x^*_i=r(\theta^*)$ and $y^*_j=s(\theta^*)$.
		Let $\textsf{Z}=\{\theta^*\pm \textsf{i} \textsf{h} : 0\le \textsf{i}\le 5,  \textsf{h}=h/2^4\}$,  $\textsf{z}^*=\textsf{x}^*=\theta^*$. Then we  use \eqref{est:deriva} with $\textsf{P}_\textsf{6}$ to approximate $\{ r^{(p)}\}_{p=0}^5$, $\{ s^{(p)}\}_{p=0}^5$, $\{ g^{(p)}\}_{p=0}^5$ and  use \eqref{est:deriva} with $\textsf{P}_\textsf{5}$ to approximate $\{ g_{\Gamma}^{(p)}\}_{p=0}^4$.
		
\end{itemize}

	\item For  1D and 2D functions used in \cref{thm:Robin:Gamma1}.	Let $\textsf{Z}=\{(x_0+ \textsf{i}\textsf{h},y_j\pm \textsf{j}\textsf{h}) : 0\le \textsf{i},\textsf{j}\le 2^3,  \textsf{h}=h/2^3\}$,  $\textsf{z}^*=(\textsf{x}^*,\textsf{y}^*)=(x_0,y_j)$. Then we  use \eqref{est:deriva} with $\textsf{P}_\textsf{5}$ and  $\textsf{P}_\textsf{4}$ to approximate $\{a^{(m,n)}\}_{ (m,n)\in \ind_{5}}$ and $\{f^{(m,n)}\}_{ (m,n)\in  \ind_{4}}$  respectively,   use \eqref{est:deriva} with  $\Gamma_1\cap\textsf{Z}$ and $\textsf{P}_\textsf{5}$ to approximate $\{\alpha^{(n)}\}_{n=0}^5$ and $\{g_1^{(n)}\}_{n=0}^5$.

	\item For 1D and 2D functions used in \cref{thm:Corner:1}.	 Let $\textsf{Z}=\{(x_0+ \textsf{i}\textsf{h},y_0+ \textsf{j}\textsf{h}) : 0\le \textsf{i},\textsf{j}\le 2^4,  \textsf{h}=h/2^4\}$,  $\textsf{z}^*=(\textsf{x}^*,\textsf{y}^*)=(x_0,y_0)$. Then we  use \eqref{est:deriva} with $\textsf{P}_\textsf{5}$ and  $\textsf{P}_\textsf{4}$ to approximate $\{a^{(m,n)}\}_{ (m,n)\in \ind_{5}}$ and $\{f^{(m,n)}\}_{ (m,n)\in  \ind_{4}}$  respectively,   use \eqref{est:deriva} with  $(\Gamma_1\cup \Gamma_3 \cup (x_0,y_0))\cap\textsf{Z}$ and $\textsf{P}_\textsf{5}$ to approximate $\{\alpha^{(n)}\}_{n=0}^5$,  $\{g_1^{(n)}\}_{n=0}^5$, $\{\beta^{(m)}\}_{m=0}^5$, and $\{g_3^{(m)}\}_{m=0}^5$.
\end{itemize}

	\section{Numerical experiments}
	\label{sec:numerical}
Let $\Omega=(l_1,l_2)\times(l_3,l_4)$ with
$l_4-l_3=N_0(l_2-l_1)$ for some positive integer $N_0$. For a given $J\in \NN$, we define
$h:=(l_2-l_1)/N_1$ with $N_1:=2^J$ and let
$x_i=l_1+ih$ and
$y_j=l_3+jh$ for $i=0,1,\dots,N_1$ and $j=0,1,\dots,N_2$ with $N_2:=N_0N_1$.
Let
$u(x,y)$ be the exact solution of \eqref{Qeques2} and $(u_{h})_{i,j}$ be a numerical approximated solution at $(x_i, y_j)$ using the mesh size $h$.
Then we  quantify the order of convergence  of the proposed hybrid FDM  by the following errors
\[
\begin{split}
& \|u_h-u\|_\infty
:=\max_{0\le i\le N_1, 0\le j\le N_2} \left|(u_h)_{i,j}-u(x_i,y_j)\right|,\quad \text{if the exact solution $u$ is known},\\
& \|u_{h}-u_{h/2}\|_\infty:=\max_{0\le i\le N_1,0\le j\le N_2} \left|(u_{h})_{i,j}-(u_{h/2})_{2i,2j}\right|, \quad \text{if the exact solution $u$ is unknown}.
\end{split}
\]
Before presenting several numerical examples, we make some remarks. First of all, to set up our FDMs at an irregular point near a base point $(x_i^*, y_j^*)\in \Gamma$, we only need a local parametric equation describing $\Gamma$ near $(x_i^*,y_j^*)$, and
the uniqueness of \eqref{tranmiss:cond} guarantees that our proposed FDMs in this paper are independent of the choice of the local parametric equations of $\Gamma$.
Hence, the essential 1D data $\gd, \gn$ on the interface $\Gamma$ in \eqref{Qeques2} can be given by any chosen local parametric equation of $\Gamma$. Second, in some applications, the 1D data $\gd,\gn$ along $\Gamma$ only depend on the geometry (such as the curvature) of $\Gamma$. Our proposed FDMs can handle it easily.
We present \cref{hybrid:ex3}, where $\gd,\gn$ at any point $p\in \Gamma$ are functions of the curvature of $\Gamma$ at $p\in \Gamma$.
Third, though theoretically our proposed FDMs employ high order (partial) derivatives of given/known data, for all our numerical examples, we always use the
numerical technique stated in \cref{Estimate:derivatives} to estimate all needed high order (partial) derivatives by only using function values without losing accuracy and performance.
Fourth, we provide \cref{hybrid:ex1} where $\Gamma$ is described by a level set $\psi(x,y)=0$. The Implicit Function Theorem theoretically guarantees a local parametric equation near a base point and their associated derivatives can be computed without explicitly solving $\psi(x,y)=0$. Without using the Implicit Function Theorem, we can easily obtain some points $(x,y)\in \Gamma$ satisfying $\psi(x,y)=0$ (e.g., if $x$ is given, then the $y$ value(s) can be computed from $\psi(x,y)=0$ by Newton method) and then we can use the function values at these points to approximate the needed derivatives for our FDMs.

\subsection{Two numerical examples with known $u$}
%
%
%
\begin{example}\label{hybrid:ex1}
	\normalfont
Let $\Omega=(-2.5,2.5)^2$ and
the functions in \eqref{Qeques2} are given by
\begin{align*}
	& \Gamma=\{ (x,y) : \ x^4+2y^4-2=0\}, \\
	& \Omega_{+}=\{(x,y)\in \Omega : x^4+2y^4-2>0\},\qquad a_{+}=2+\sin(x)\sin(y), \\
	&\Omega_{-}=\{(x,y)\in \Omega : x^4+2y^4-2<0\},
	\qquad a_{-}=10^3(2+\sin(x)\sin(y)), \qquad g=-30, \qquad g_{\Gamma}=0,\\
	&u_{+}=\sin(2 x)\sin(2 y)(x^4+2y^4-2)+1,
	\qquad u_{-}=10^{-3}\sin(2 x)\sin(2 y)(x^4+2y^4-2)+31,\\
	& \tfrac{\partial u}{\partial \nv}+(\cos(y)+2)u =g_1 \text{ on } \Gamma_1, \quad u =g_2 \text{ on } \Gamma_2,\quad \tfrac{\partial u}{\partial \nv} +(\sin(x)+2) u=g_3 \text{ on } \Gamma_3, \quad
	u =g_4 \text{ on } \Gamma_4,	
\end{align*}
the other functions $f_{+}=f_{-}$, $g_1, \ldots,g_4$ in \eqref{Qeques2} can be obtained by plugging the above functions into \eqref{Qeques2}. {\color{blue}{$\|u_h\|_{\infty}=106.95$ with $J=9$.}}
	The numerical results are presented in \cref{hybrid:table:ex1,hybrid:table:ex2} and \cref{hybrid:fig:ex1}.	
\end{example}
\begin{table}[htbp]
	\caption{Performance in \cref{hybrid:ex1,hybrid:ex2,hybrid:ex3,hybrid:ex4}  of the proposed hybrid FDM.}
		\scalebox{0.81}{
	\centering
	\setlength{\tabcolsep}{1mm}{
		\begin{tabular}{c|c|c|c|c|c|c|c|c|c|c|c}
			\hline	
			 \multicolumn{3}{c|}{\cref{hybrid:ex1} with $h=\tfrac{5}{2^J}$} &
			 \multicolumn{3}{c|}{\cref{hybrid:ex2} with $h=2^{2-J}$} &
			 \multicolumn{3}{c|}{\cref{hybrid:ex3} with $h=\tfrac{3}{2^J}$} &
			 \multicolumn{3}{c}{\cref{hybrid:ex4} with $h=2^{2-J}$ } \\
			\cline{1-12}			
			$J$
			 & $\|u_{h}-u\|_{\infty}$
			
			&order &
			
		    $J$   & $\|u_{h}-u\|_{\infty}$
			
			&order   &
			
			$J$  & $\|u_{h}-u_{h/2}\|_{\infty}$
			
			&order   &
			
			$J$  & $\|u_{h}-u_{h/2}\|_{\infty}$
			
			&order  \\
			\hline
   &   &   &   &   &   &   &   &   &4   &1.31113E+06   & \\
5   &3.87139E+01   &   &5   &2.25635E+06   &   &   &   &   &5   &4.75213E+04   &4.79 \\
6   &7.69404E-01   &5.65   &6   &6.11924E+04   &5.20   &6   &1.31810E+02   &   &6   &6.84381E+02   &6.12 \\
7   &1.62588E-02   &5.56   &7   &5.49910E+02   &6.80   &7   &3.04318E+00   &5.44   &7   &5.23606E+00   &7.03\\
8   &1.57108E-04   &6.69   &8   &4.90656E+00   &6.81   &8   &4.78581E-02   &5.99   &8   &9.05642E-02   &5.85\\
9   &1.99369E-06   &6.30   &9   &1.03630E-01   &5.57   &9   &7.89042E-04   &5.92   &9   &1.18424E-03   &6.26\\
	\hline
   &    \textbf{Average Order:} &  \textbf{6.05}   &  &  \textbf{Average Order:} &  \textbf{6.09}   &   & \textbf{Average Order:}  &  \textbf{5.78}   &   & \textbf{Average Order:}  &  \textbf{6.01} \\		
			\hline
	\end{tabular}}
	\label{hybrid:table:ex1}}
\end{table}	
\begin{table}[htbp]
	\caption{ 	 CPU time to compute $u_h$ with $J=9$ in \cref{hybrid:ex1,hybrid:ex2,hybrid:ex3,hybrid:ex4} by the proposed hybrid FDM.}
	\scalebox{0.81}{
		\centering
		\setlength{\tabcolsep}{5mm}{
			\begin{tabular}{c|c|c|c|c}
				\hline	
				\multicolumn{1}{c|}{} &
				\multicolumn{1}{c|}{\cref{hybrid:ex1}} &
				\multicolumn{1}{c|}{\cref{hybrid:ex2}} &
				\multicolumn{1}{c|}{\cref{hybrid:ex3}} &
				\multicolumn{1}{c}{\cref{hybrid:ex4}}  \\
				\cline{1-5}			
			    Stencil generations (regular points)	& 3.970 minutes & 2.756 minutes & 2.261 minutes  & 2.258 minutes \\
				\hline
				Stencil generations	 (irregular points) & 1.701 minutes & 3.058 minutes & 2.064 minutes  & 4.122 minutes \\
				\hline
			    Form linear systems	& 0.104 seconds & 0.107 seconds &  0.096 seconds & 0.118 seconds \\
				\hline
				Solve linear systems	& 1.112 seconds & 1.265 seconds & 1.184 seconds & 1.357 seconds\\
				\hline
				Total 	& 5.691 minutes &   5.837 minutes &  4.347 minutes & 6.405 minutes \\
				\hline
		\end{tabular}}
		\label{hybrid:table:ex2}}
\end{table}	
\begin{figure}[htbp]
	\centering
	\begin{subfigure}[b]{0.24\textwidth}
	 \includegraphics[width=4.75cm,height=4.75cm]{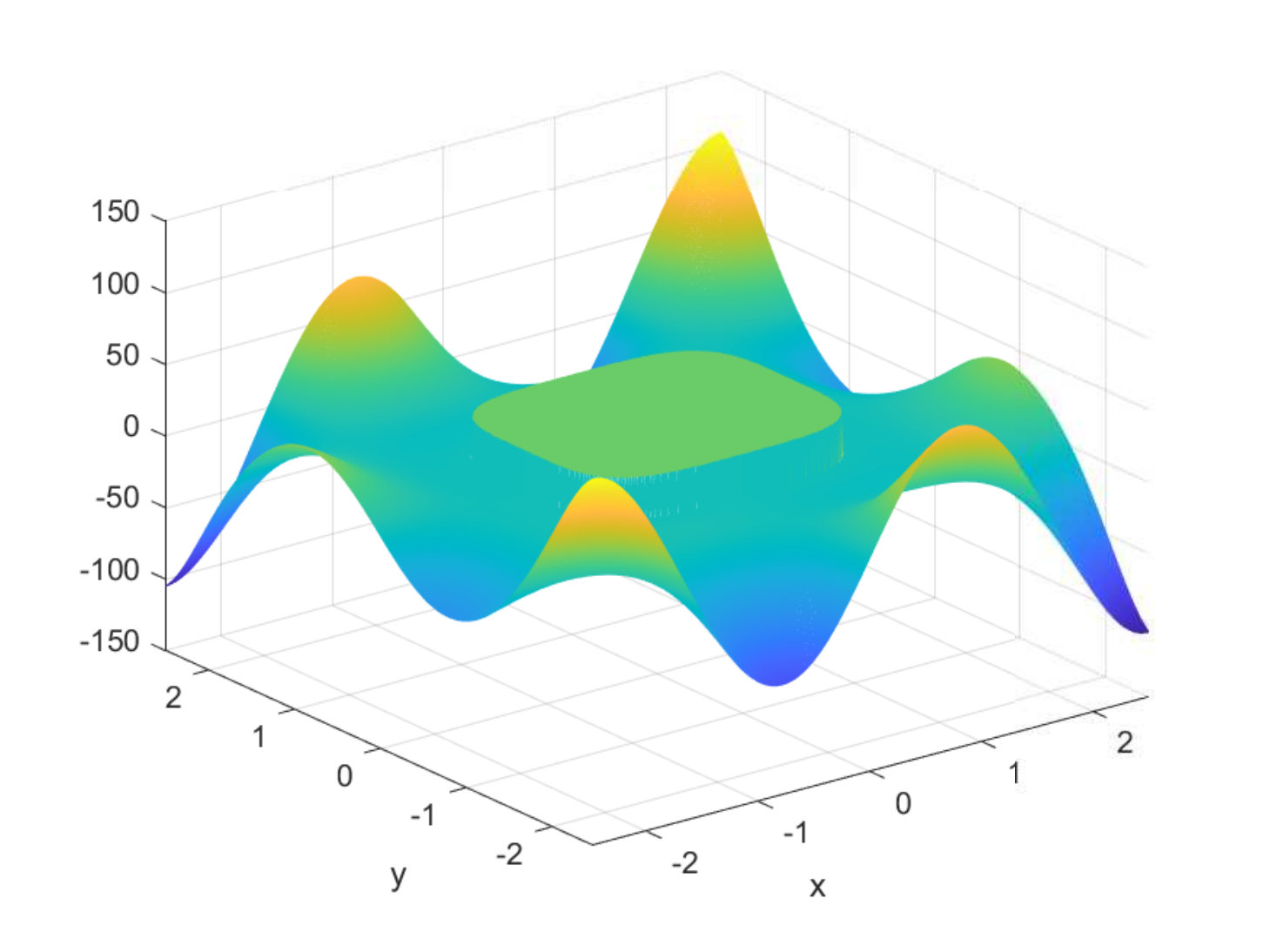}
\end{subfigure}
\begin{subfigure}[b]{0.24\textwidth}
	 \includegraphics[width=4.75cm,height=4.75cm]{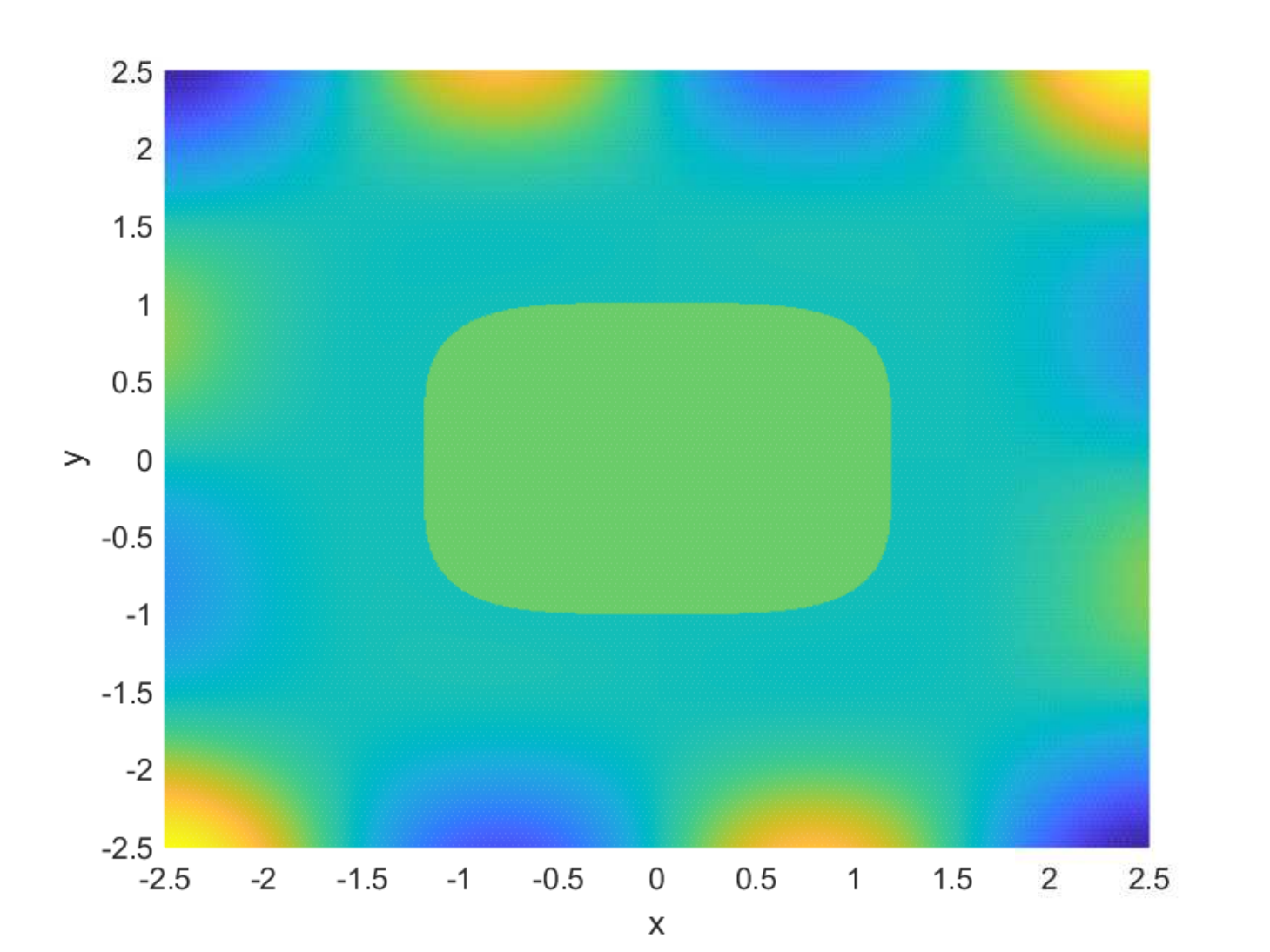}
\end{subfigure}
\begin{subfigure}[b]{0.24\textwidth}
	 \includegraphics[width=4.75cm,height=4.75cm]{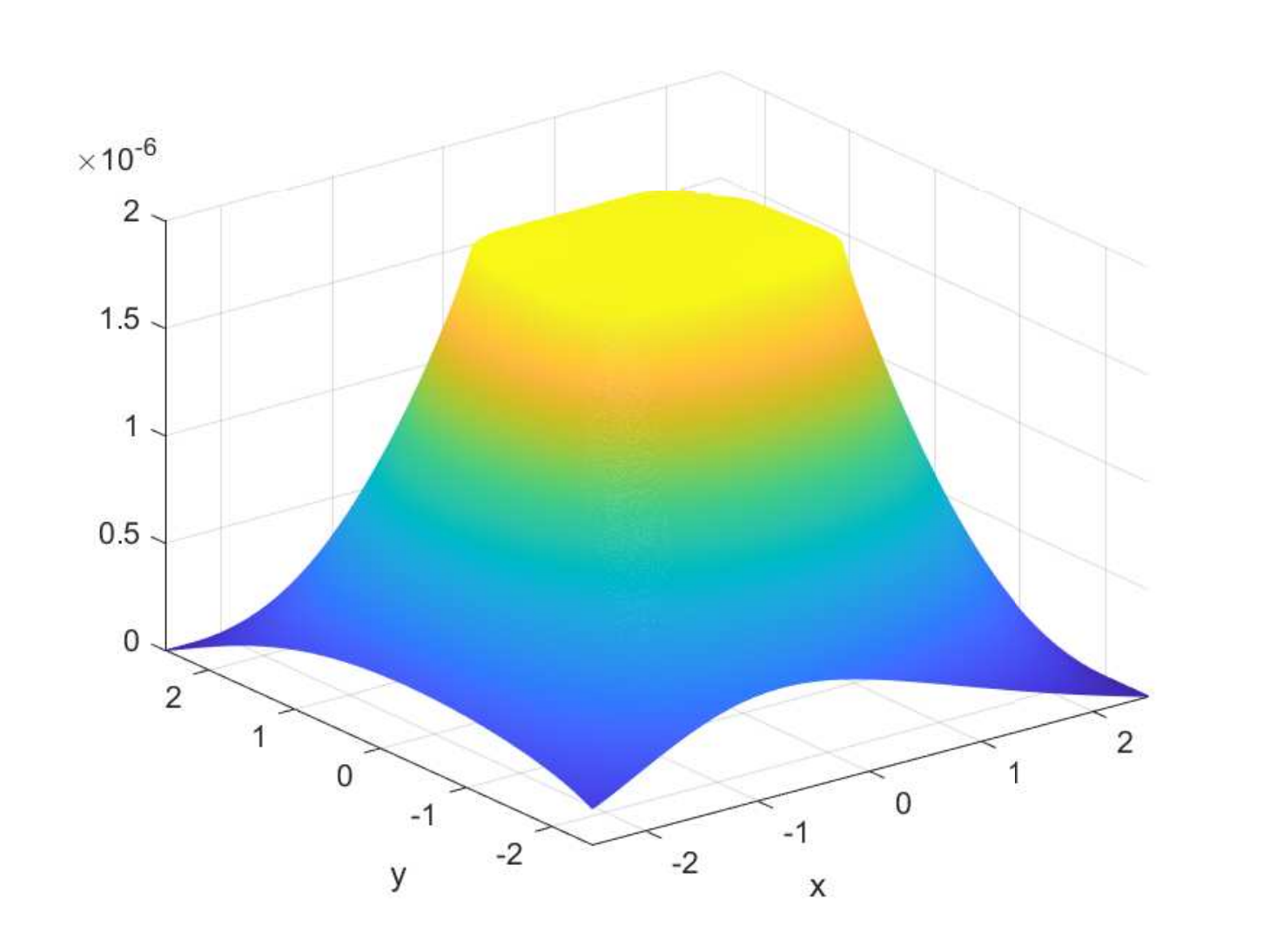}
\end{subfigure}
\begin{subfigure}[b]{0.24\textwidth}
	 \includegraphics[width=4.75cm,height=4.75cm]{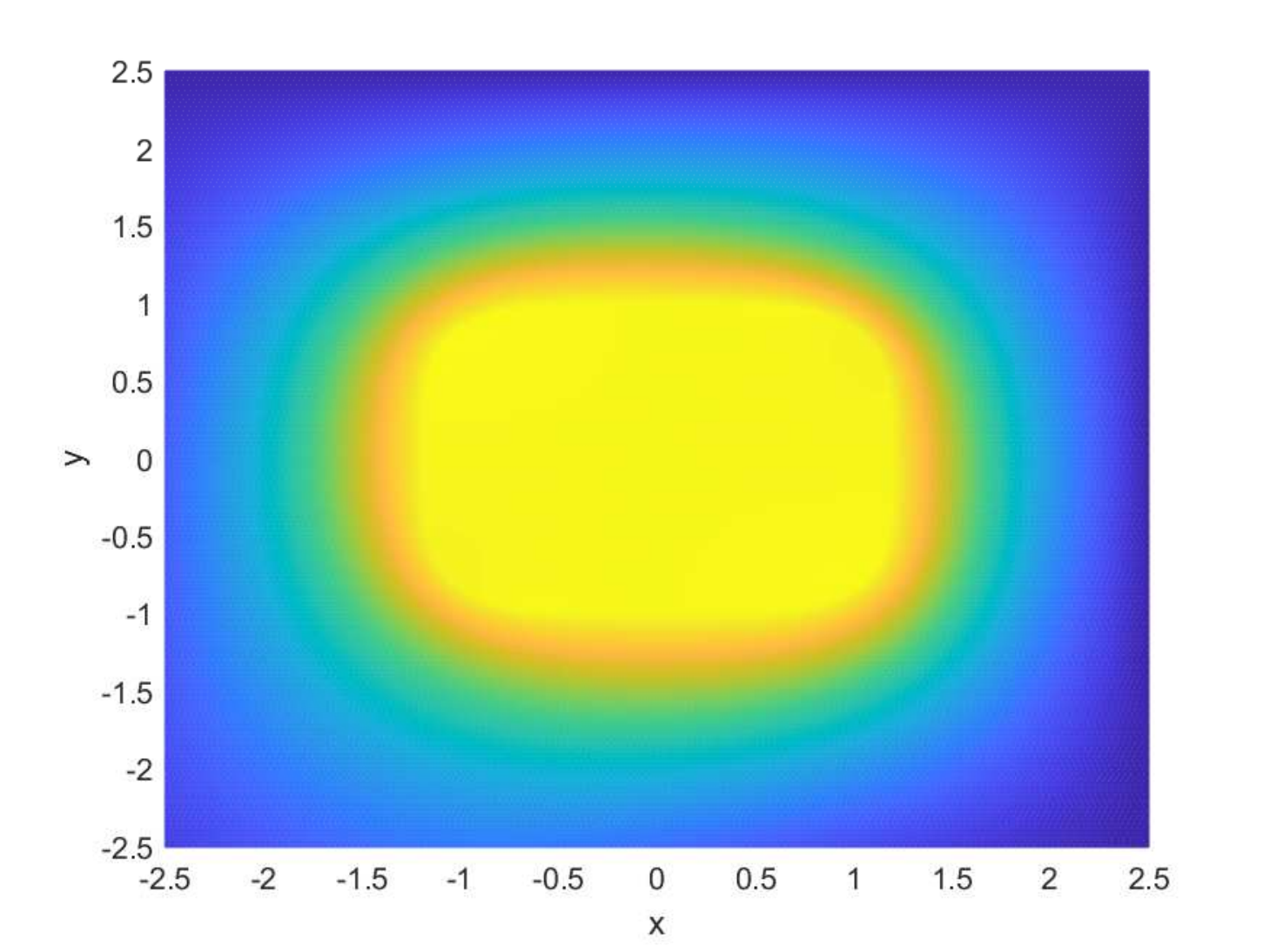}
\end{subfigure}
	\caption
	{Performance  of the  proposed hybrid FDM in case of  \cref{hybrid:ex1}: the numerical solution $(u_h)_{i,j}$  (first and second panels), and  $|(u_h)_{i,j}-u(x_i,y_j)|$ (third and fourth panels), at all grid points $(x_i,y_j)$ on $\overline{\Omega}=[-2.5,2.5]^2$ with $h=\tfrac{5}{2^{9}}$.
Note that the second and fourth panels are just the 2D color map formats corresponding to the first and third panels, respectively.
}
	\label{hybrid:fig:ex1}
\end{figure}	
%
%
%
\begin{example}\label{hybrid:ex2}
	\normalfont
	Let $\Omega=(-2,2)^2$ and
	the functions in \eqref{Qeques2} are given by
	\begin{align*}
		& \Gamma=\{ (x,y) : \ x(\theta)=(\pi/3+0.4\sin(8\theta))\cos(\theta),\ y(\theta)=(\pi/3+0.4\sin(8\theta))\sin(\theta)\}, \\
		& \Omega_{+}=\{(x,y)\in \Omega : x^2(\theta)+y^2(\theta)> (\pi/3+0.4\sin(8\theta))^2\},\qquad a_{+}=1, \qquad f_{+}=\cos(x),\\
		& \Omega_{-}=\{(x,y)\in \Omega : x^2(\theta)+y^2(\theta)< (\pi/3+0.4\sin(8\theta))^2\},\qquad a_{-}=10^{-3}, \qquad f_{-}=(3\pi)^2\sin(3\pi y),\\
	& u_{+}=\cos(x),
\qquad  u_{-}=10^3\sin(3\pi y)+1500,\\
& u =g_1=\cos(-2) \text{ on } \Gamma_1, \quad u =g_2=\cos(2) \text{ on } \Gamma_2,\quad u =g_3=\cos(x) \text{ on } \Gamma_3, \quad u =g_4=\cos(x) \text{ on } \Gamma_4,
	\end{align*}
the two jump functions  $g\ne0,g_{\Gamma}\ne0$ in \eqref{Qeques2} can be obtained by plugging the above functions into \eqref{Qeques2}. Note that $g, g_{\Gamma}$ are two non-constant functions and {\color{blue}{ $\|u_h\|_{\infty}=2499.98$ with $J=9$.}}
	The numerical results are presented in \cref{hybrid:table:ex1,hybrid:table:ex2} and \cref{hybrid:fig:ex2}.	
\end{example}
\begin{figure}[htbp]
	\centering
	\begin{subfigure}[b]{0.24\textwidth}
		 \includegraphics[width=4.75cm,height=4.75cm]{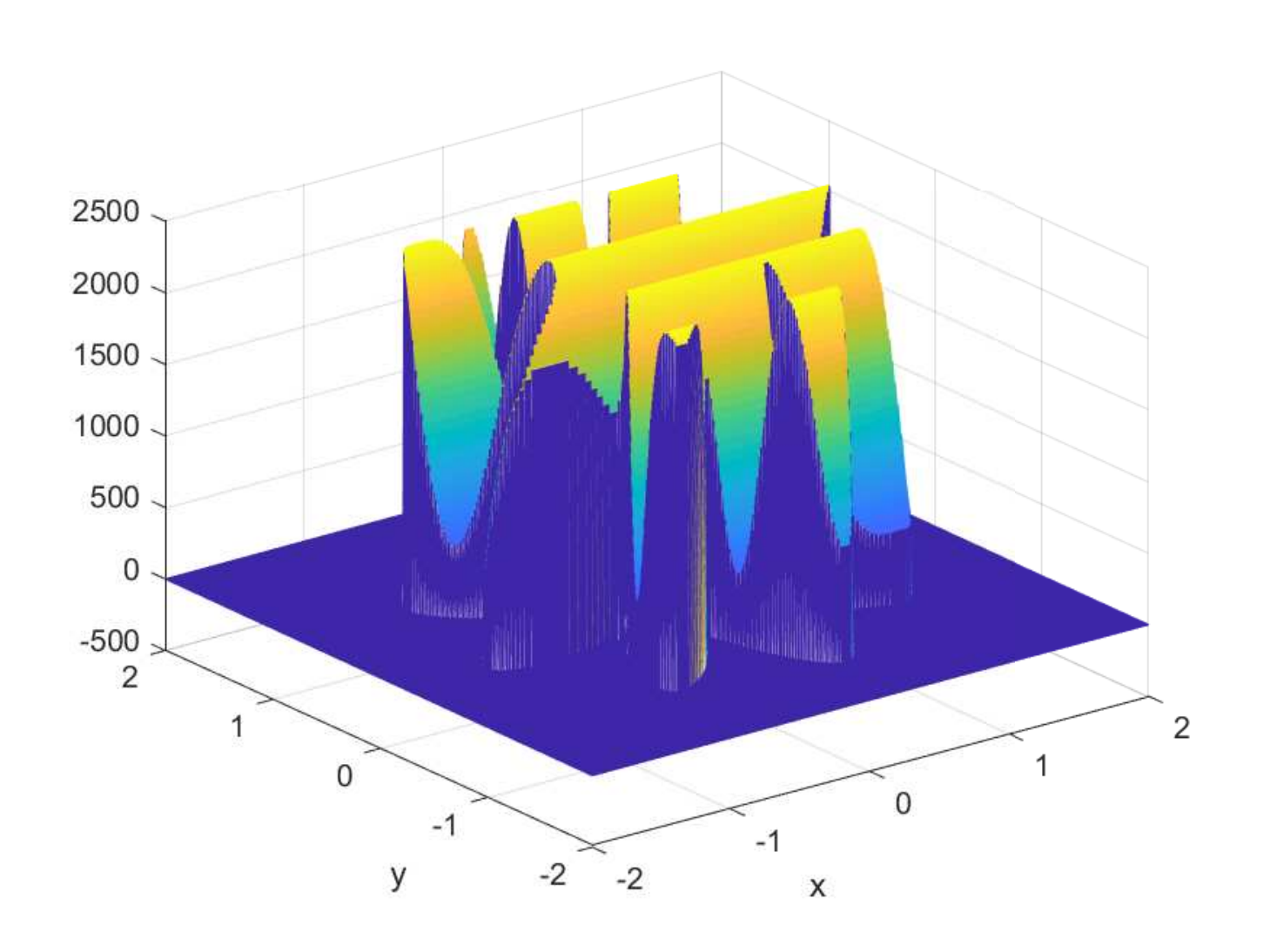}
	\end{subfigure}
	\begin{subfigure}[b]{0.24\textwidth}
		 \includegraphics[width=4.75cm,height=4.75cm]{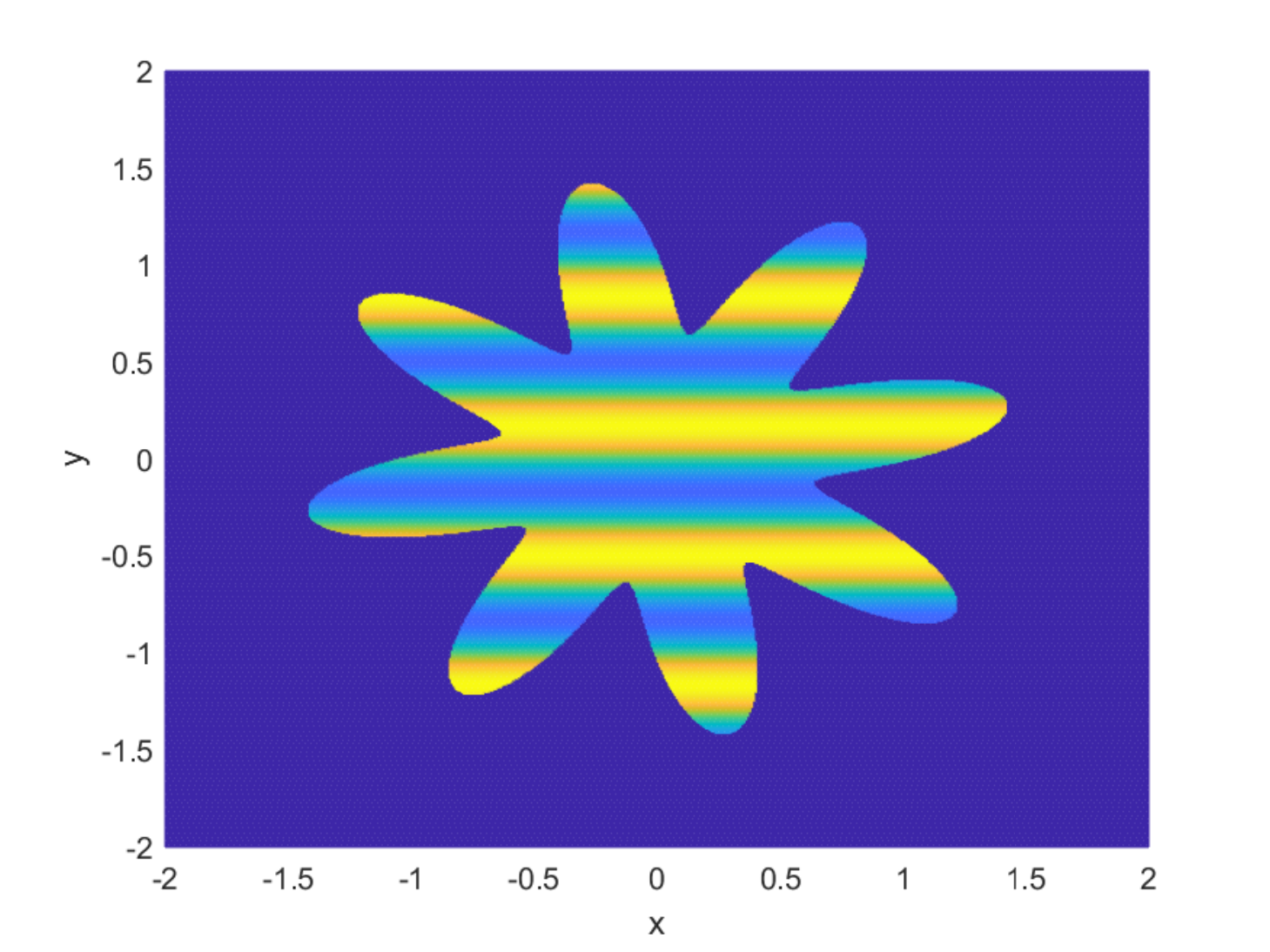}
	\end{subfigure}
	\begin{subfigure}[b]{0.24\textwidth}
		 \includegraphics[width=4.75cm,height=4.75cm]{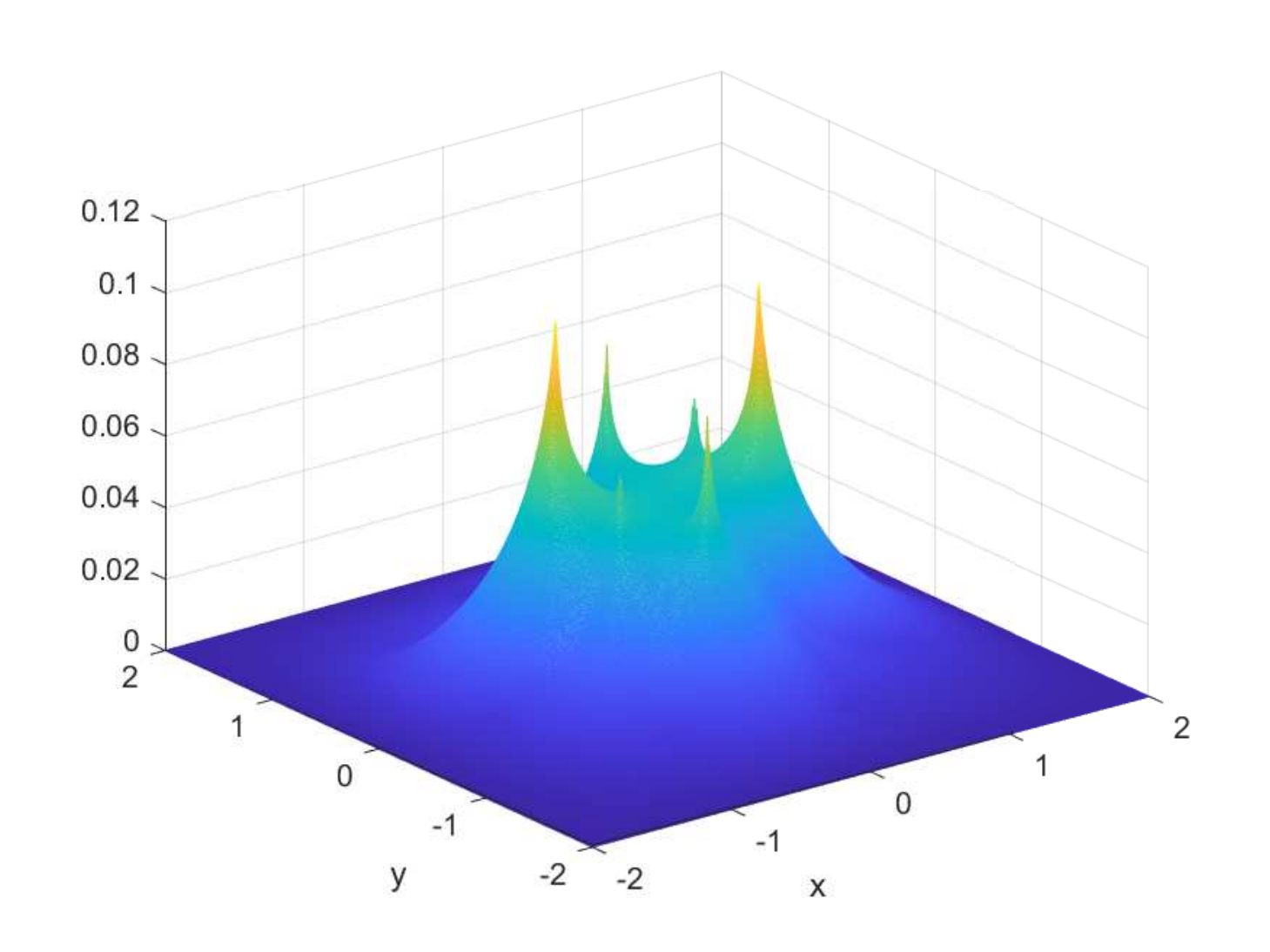}
	\end{subfigure}
	\begin{subfigure}[b]{0.24\textwidth}
		 \includegraphics[width=4.75cm,height=4.75cm]{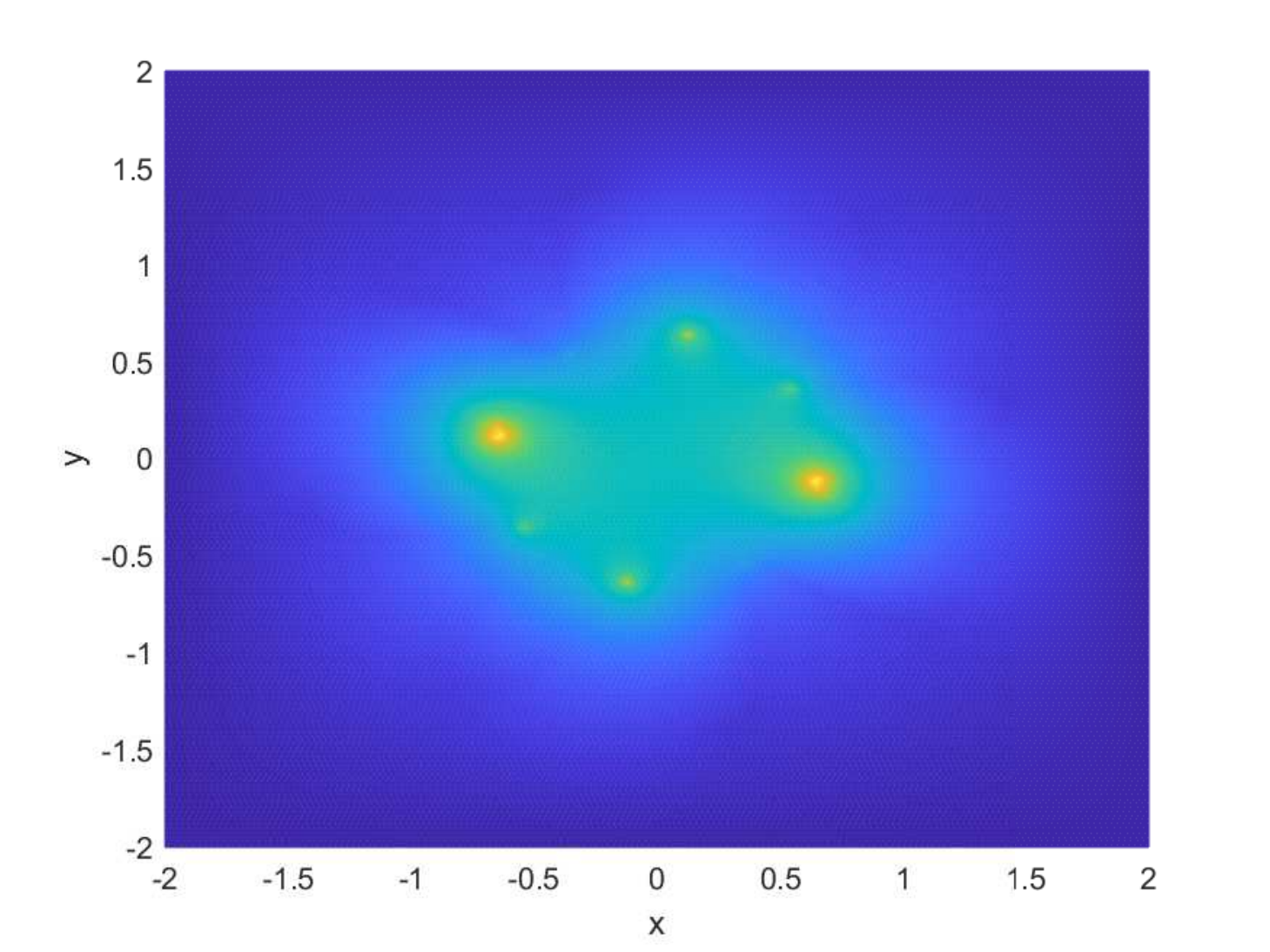}
	\end{subfigure}
	\caption
	{Performance of the proposed hybrid FDM in case of \cref{hybrid:ex2} : the numerical solution $(u_h)_{i,j}$ (first and second panels) and $|(u_h)_{i,j}-u(x_i,y_j)|$ (third and fourth panels) at all grid points $(x_i,y_j)$ on $\overline{\Omega}=[-2,2]^2$ with $h=2^{-7}$. Note that the second and fourth panels are just the 2D color map formats corresponding to the first and third panels, respectively.
}
	\label{hybrid:fig:ex2}
\end{figure}
\subsection{Two numerical examples with unknown $u$}
%
%

\begin{example}\label{hybrid:ex3}
	\normalfont
	Let $\Omega=(-3/2,3/2)^2$ and
	the functions in \eqref{Qeques2} are given by
	\begin{align*}
		& \Gamma=\{ (x,y) : \ x(\theta)=\cos(\theta),\ y(\theta)=1/2\sin(\theta)\}, \\
		& \Omega_{+}=\{(x,y)\in \Omega : x^2(\theta)+4y^2(\theta)> 1\},\qquad a_{+}=2+\sin(x+y),\qquad f_{+}=\cos(\pi x)\cos(\pi y),\\
		& \Omega_{-}=\{(x,y)\in \Omega : x^2(\theta)+4y^2(\theta)< 1\},		 \qquad a_{-}=10^{4}(2+\sin(x+y)),\qquad f_{-}=\sin( \pi(x-y)),\\
	    &g=\frac{|x'(\theta)y''(\theta)-x''(\theta)y'(\theta)|}{((x'(\theta))^2+(y'(\theta))^2)^{3/2}}-1, \qquad g_\Gamma=\frac{|x'(\theta)y''(\theta)-x''(\theta)y'(\theta)|}{((x'(\theta))^2+(y'(\theta))^2)^{3/2}}, \quad \mbox{for} \quad \theta\in [0,2\pi),\\
	& \tfrac{\partial u}{\partial \nv}+(\cos(y)+2)u =\sin(2\pi y) \text{ on } \Gamma_1, \qquad u =0 \text{ on } \Gamma_2,\\
	& \tfrac{\partial u}{\partial \nv} +(\sin(x)+2) u=\cos(\pi x) \text{ on } \Gamma_3, \qquad
u =0 \text{ on } \Gamma_4.
	\end{align*}
	Note that the exact solution $u$ is unknown in this example, and $\frac{|x'(\theta)y''(\theta)-x''(\theta)y'(\theta)|}{((x'(\theta))^2+(y'(\theta))^2)^{3/2}}$ is just the curvature of $\Gamma$ at the point $(x(\theta),y(\theta))\in \Gamma$. Therefore, the data $\gd,\gn$ on $\Gamma$ are functions depending only on the curvature of the interface $\Gamma$. {\color{blue}{$\|u_h\|_{\infty}=1.7822$ with $J=10$.}}
	The numerical results are presented in \cref{hybrid:table:ex1,hybrid:table:ex2} and \cref{hybrid:fig:ex3}.	
\end{example}
\begin{figure}[htbp]
	\centering
	\begin{subfigure}[b]{0.24\textwidth}
		 \includegraphics[width=4.66cm,height=4.66cm]{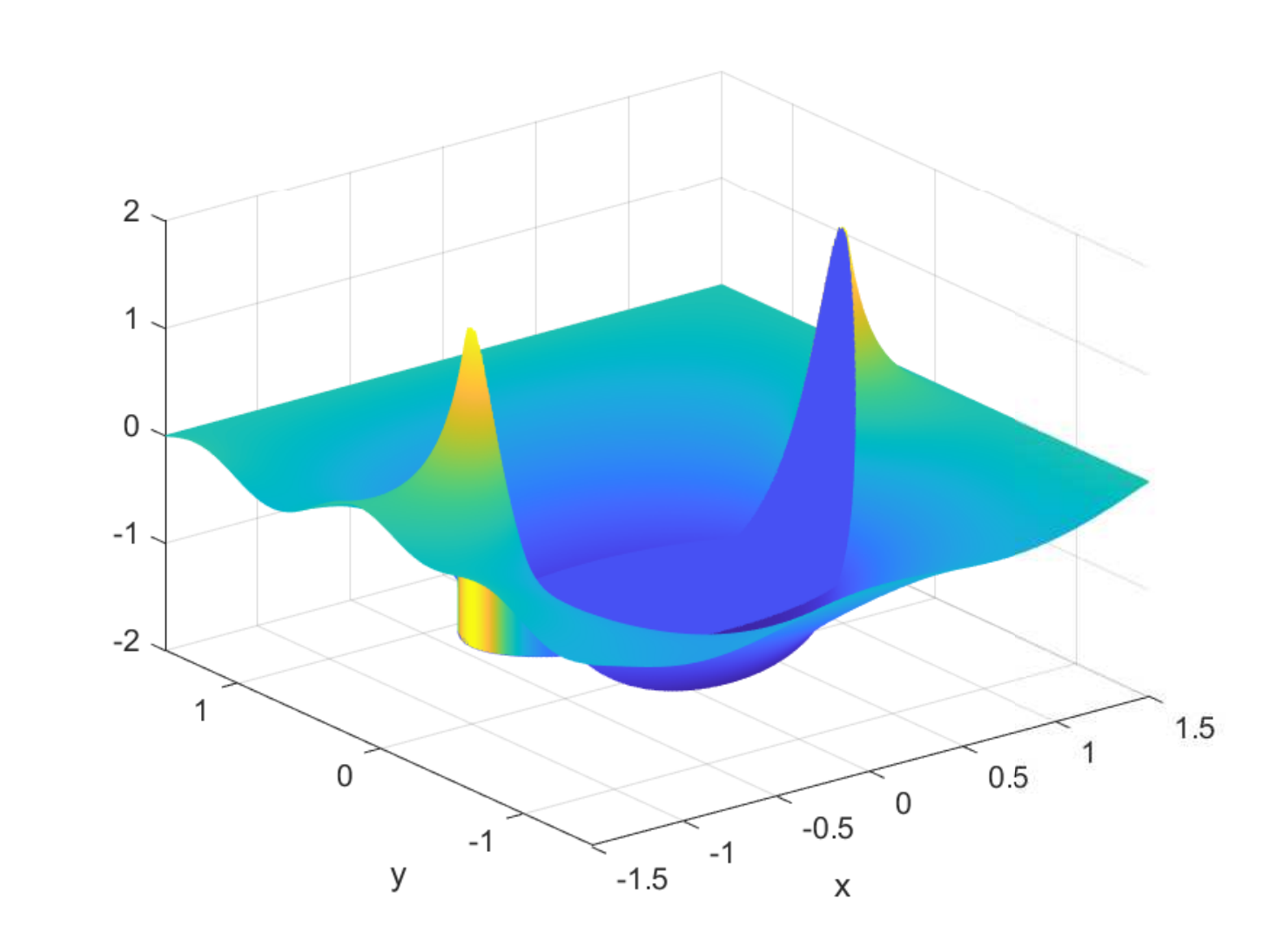}
	\end{subfigure}
	\begin{subfigure}[b]{0.24\textwidth}
		 \includegraphics[width=4.66cm,height=4.66cm]{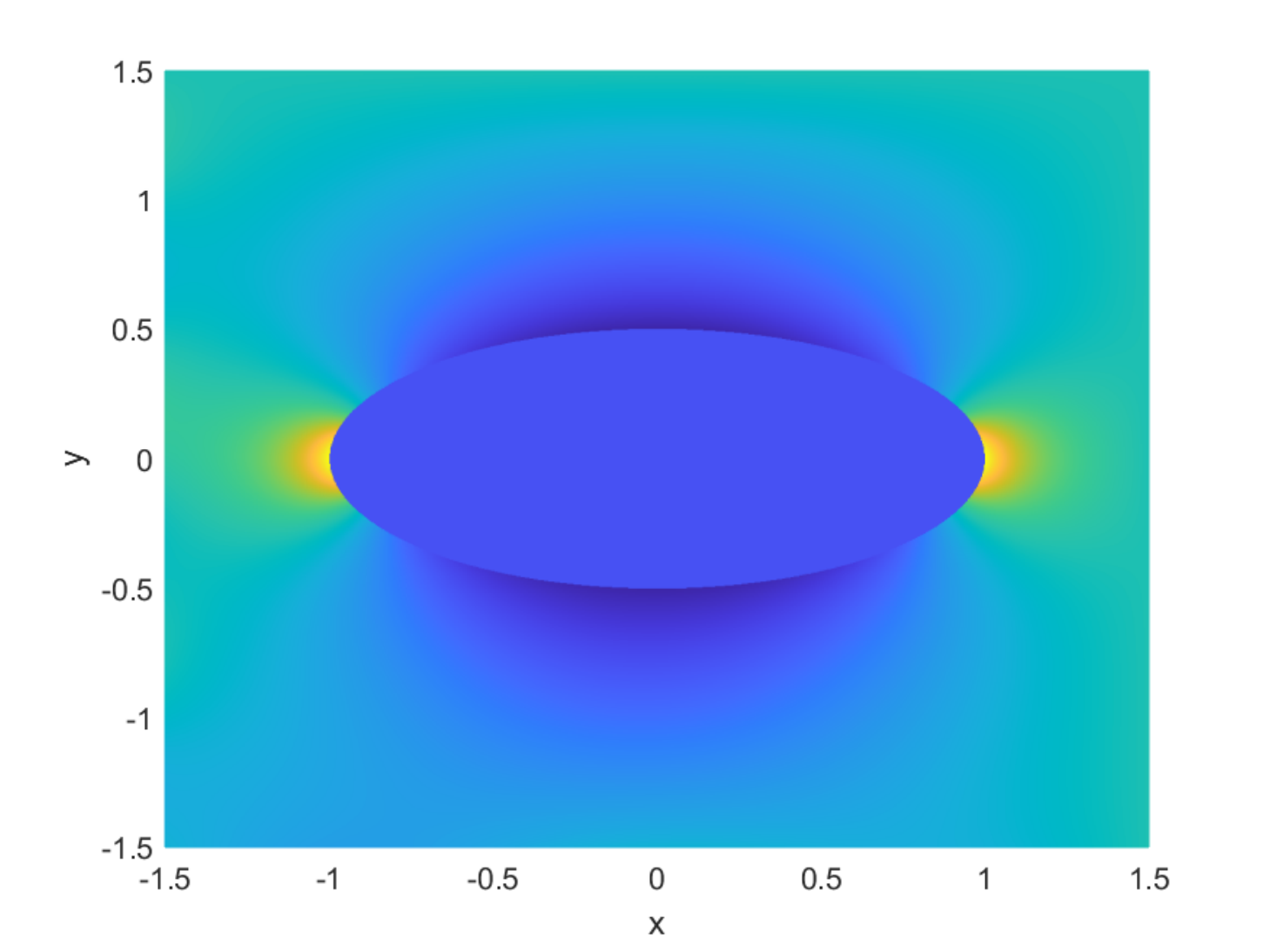}
	\end{subfigure}
	\begin{subfigure}[b]{0.24\textwidth}
		 \includegraphics[width=4.66cm,height=4.66cm]{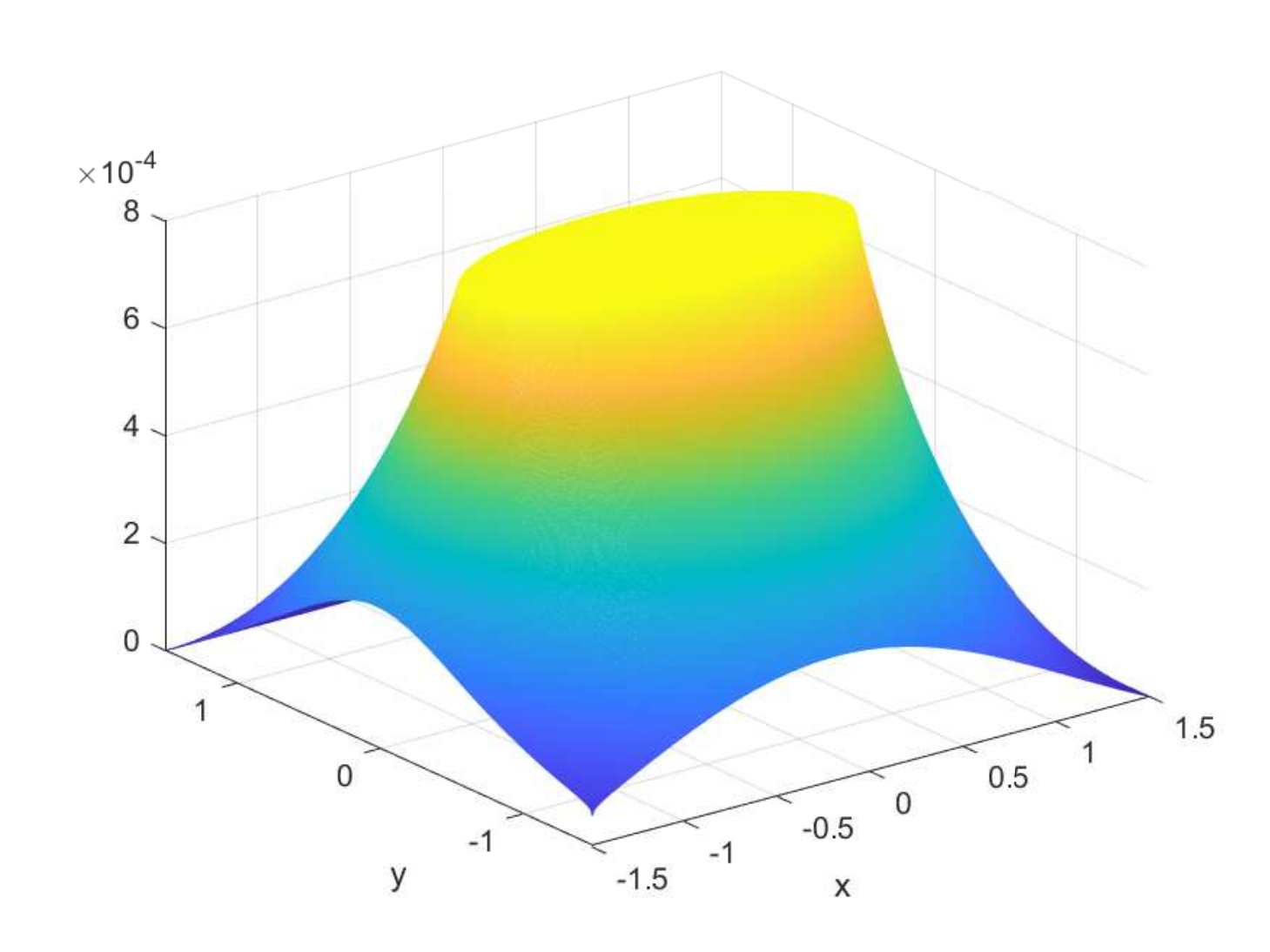}
	\end{subfigure}
	\begin{subfigure}[b]{0.24\textwidth}
		 \includegraphics[width=4.66cm,height=4.66cm]{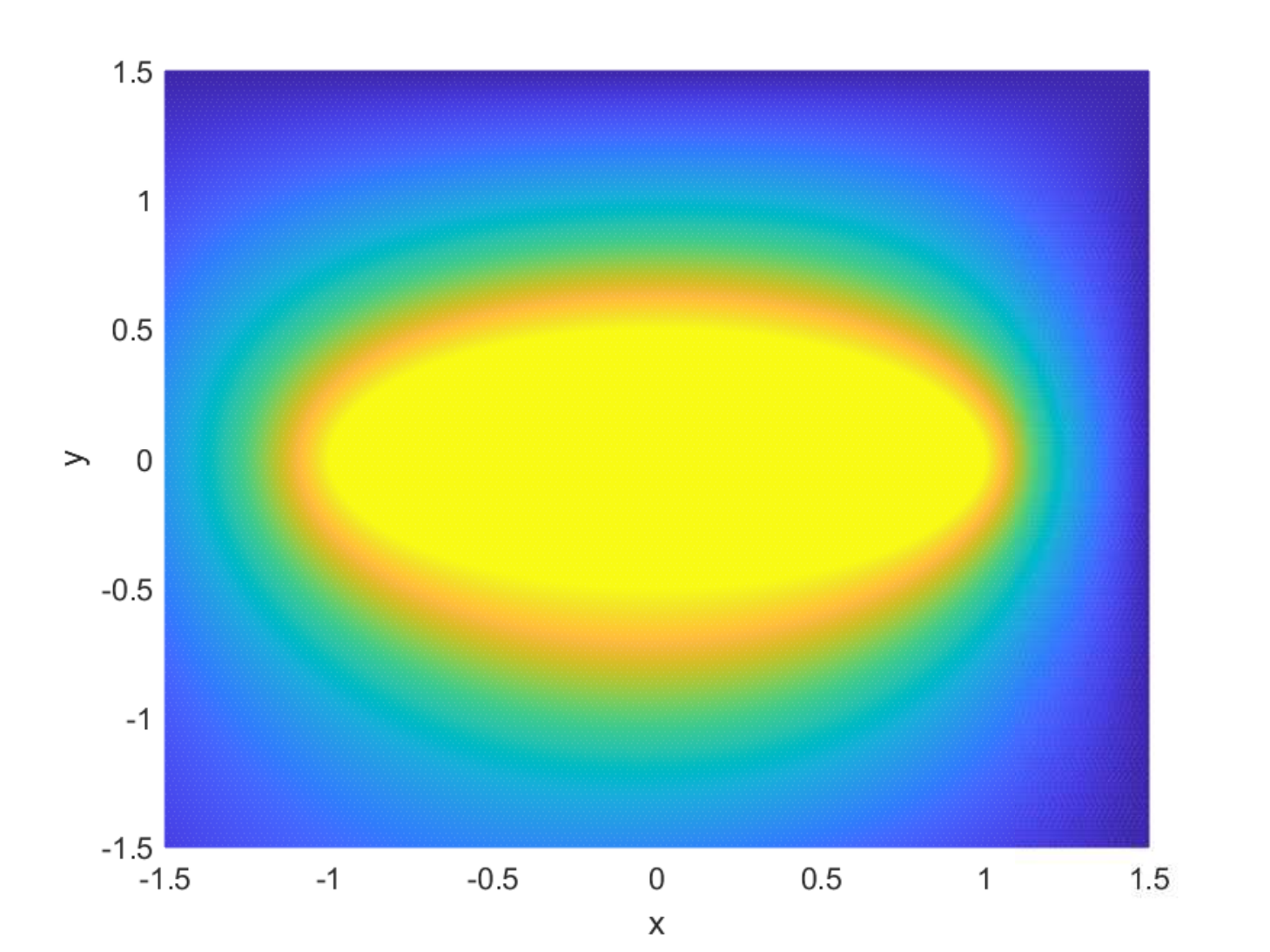}
	\end{subfigure}
	\caption
	{Performance of the proposed hybrid FDM in case of  \cref{hybrid:ex3}:   the numerical solution $(u_h)_{i,j}$ (first and second panels) at all grid points $(x_i,y_j)$ on $\overline{\Omega}=[-3/2,3/2]^2$ with $h=\tfrac{3}{2^{10}}$, and $|(u_h)_{i,j}-(u_{h/2})_{2i,2j}|$ (third and fourth panels) at all grid points $(x_i,y_j)$ on $\overline{\Omega}=[-3/2,3/2]^2$ with $h=\tfrac{3}{2^{9}}$.  Note that the second and fourth panels are just the 2D color map formats corresponding to the first and third panels, respectively.
}
	\label{hybrid:fig:ex3}
\end{figure}	
%
%
\begin{example}\label{hybrid:ex4}
	\normalfont
	Let $\Omega=(-2,2)^2$ and
	the functions in \eqref{Qeques2} are given by
	\begin{align*}
		& \Gamma=\{ (x,y) : \ x(\theta)=(\pi/3+0.4\sin(10\theta))\cos(\theta),\ y(\theta)=(\pi/3+0.4\sin(10\theta))\sin(\theta)\}, \\
		& \Omega_{+}=\{(x,y)\in \Omega : x^2(\theta)+y^2(\theta)> (\pi/3+0.4\sin(10\theta))^2\},\qquad a_{+}=10^{3}(2+\cos(x)\cos(y)),\\
		& \Omega_{-}=\{(x,y)\in \Omega : x^2(\theta)+y^2(\theta)< (\pi/3+0.4\sin(10\theta))^2\},		 \qquad a_{-}=2+\cos(x)\cos(y),\\
		&f_{+}=\sin(\pi x)\sin(\pi y),
		\qquad f_{-}=\cos(\pi x)\cos(\pi y), \\
		& g=-\sin(\theta)-1, \qquad g_{\Gamma}=\cos(\theta), \qquad \mbox{for} \qquad \theta\in [0,2\pi),\qquad  u =0 \text{ on } \partial\Omega.	
	\end{align*}
	Note that the exact solution $u$ is unknown in this example and {\color{blue}{$\|u_h\|_{\infty}=1.9992$ with $J=10$.}}
	The numerical results are presented in \cref{hybrid:table:ex1,hybrid:table:ex2} and \cref{hybrid:fig:ex4}.	
\end{example}
\begin{figure}[htbp]
	\centering
	\begin{subfigure}[b]{0.24\textwidth}
	 \includegraphics[width=4.75cm,height=4.75cm]{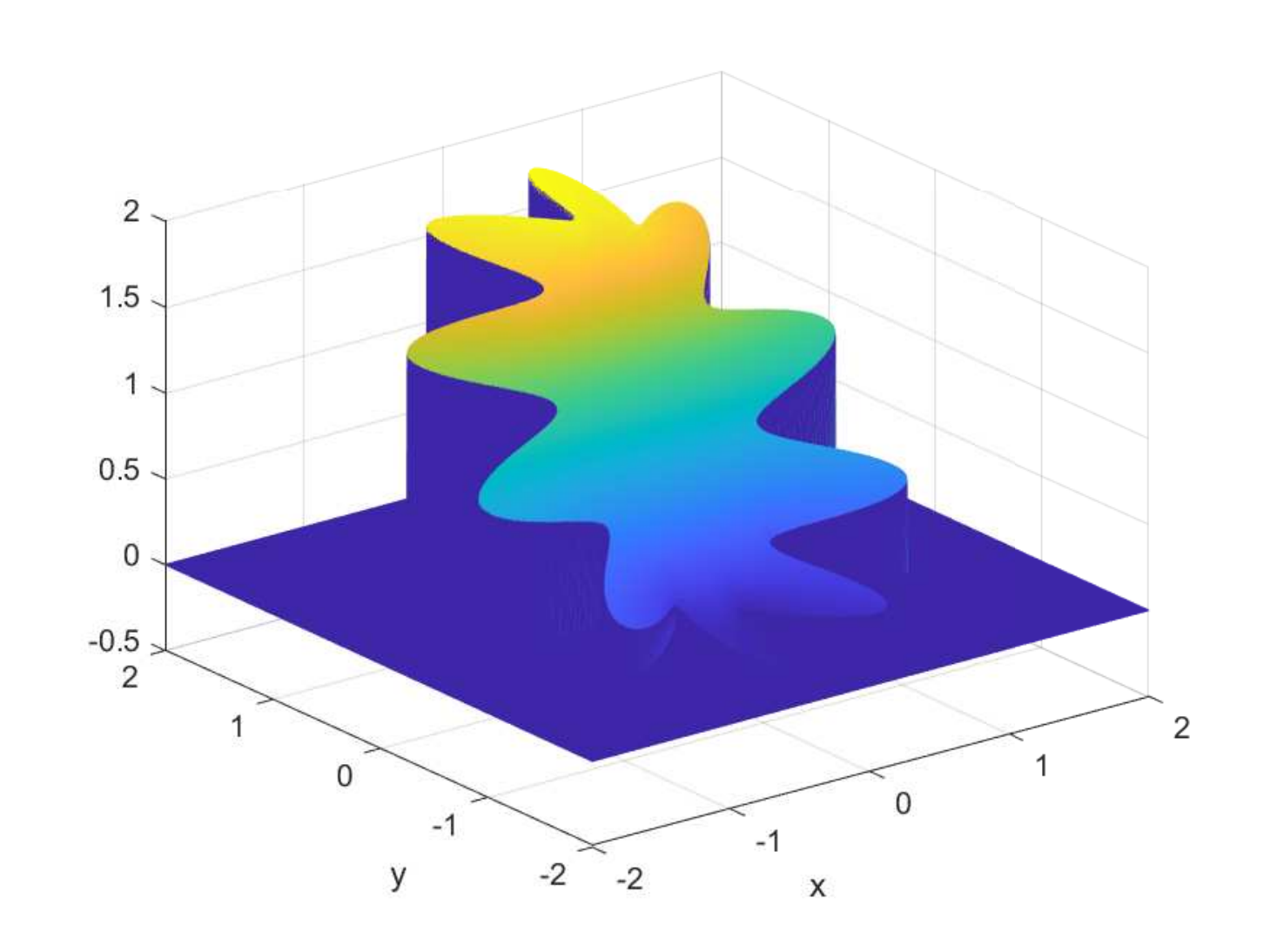}
\end{subfigure}
\begin{subfigure}[b]{0.24\textwidth}
	 \includegraphics[width=4.75cm,height=4.75cm]{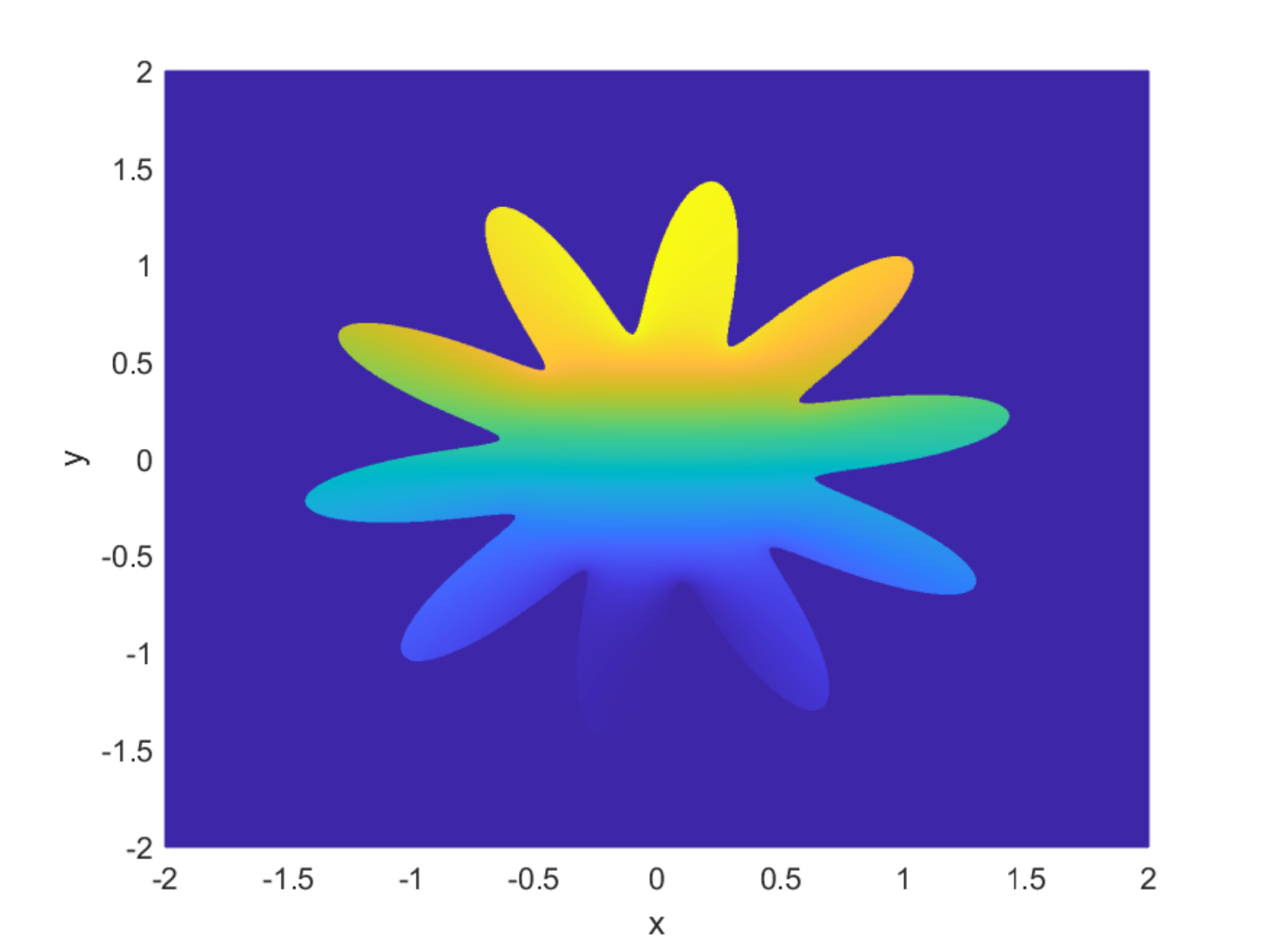}
\end{subfigure}
\begin{subfigure}[b]{0.24\textwidth}
	 \includegraphics[width=4.75cm,height=4.75cm]{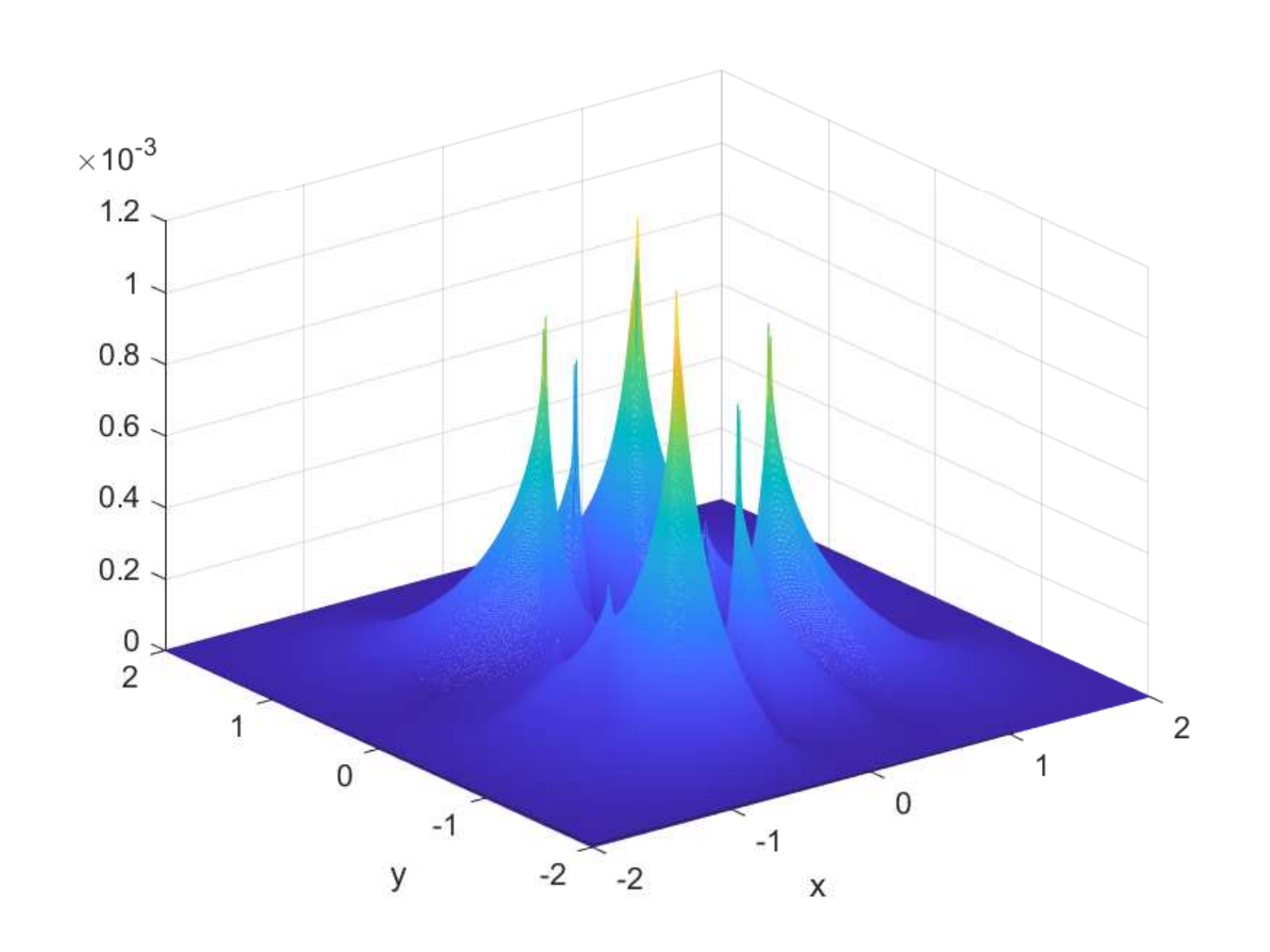}
\end{subfigure}
\begin{subfigure}[b]{0.24\textwidth}
	 \includegraphics[width=4.75cm,height=4.75cm]{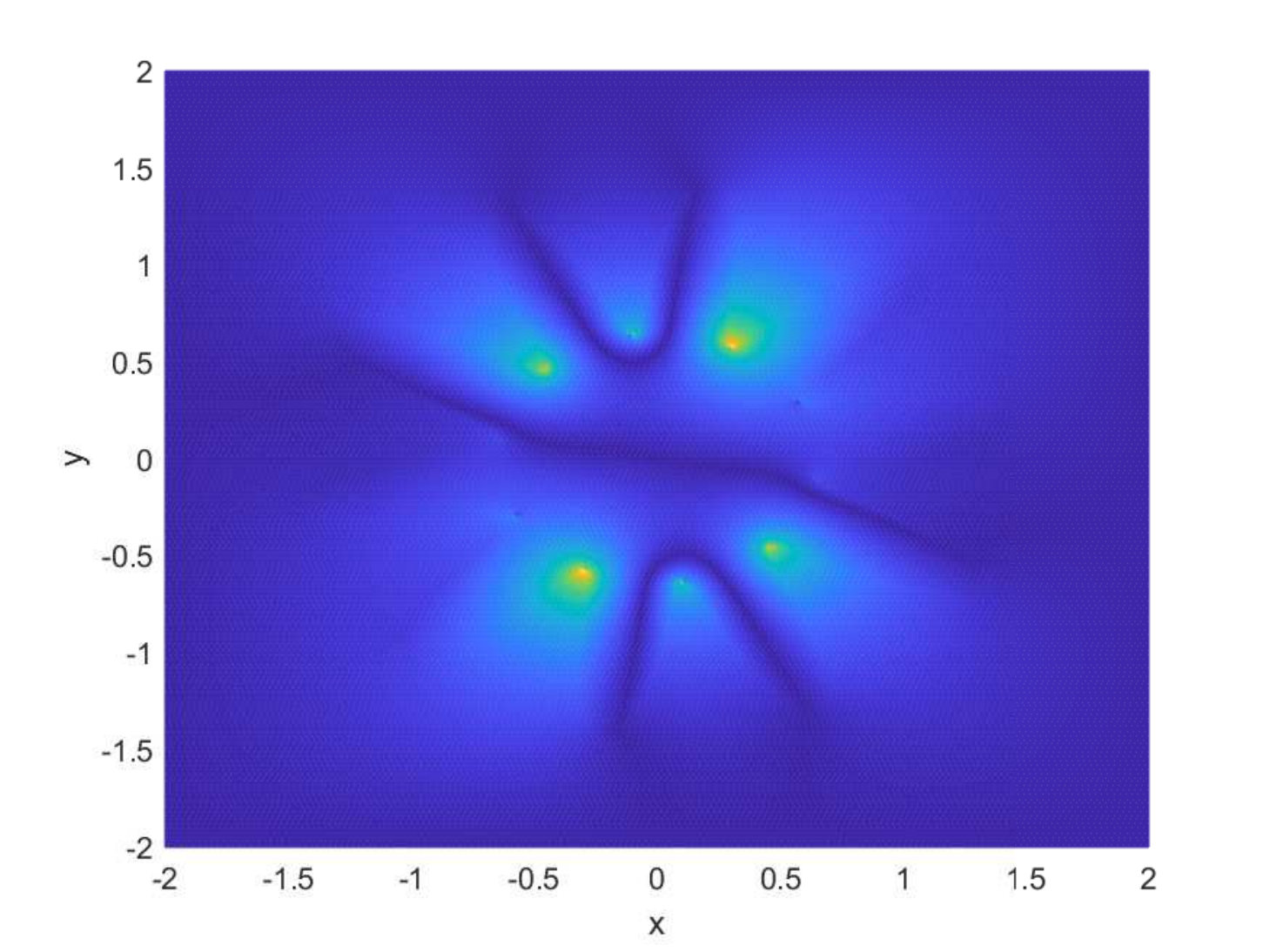}
\end{subfigure}
	\caption
	{Performance of the proposed hybrid FDM in case of  \cref{hybrid:ex4}:   the numerical solution $(u_h)_{i,j}$ (first and second panels) at all grid points $(x_i,y_j)$ on $\overline{\Omega}=[-2,2]^2$ with $h=2^{-8}$, and $|(u_h)_{i,j}-(u_{h/2})_{2i,2j}|$ (third and fourth panels) at all grid points $(x_i,y_j)$ on $\overline{\Omega}=[-2,2]^2$ with $h=2^{-7}$. Note that the second and fourth panels are just the 2D color map formats corresponding to the first and third panels, respectively.
}
	\label{hybrid:fig:ex4}
\end{figure}	
%
%
%
%
%
%
\begin{remark}
 We extended the method proposed in this paper to the Helmholtz interface problem $\Delta u+{\ka}^2u=f$ in $\Omega \setminus \Gamma$ with $\left[u\right]=g$ and  $\left[\nabla  u \cdot \nv \right]=g_{\Gamma}$ on $\Gamma$  in \cite{FHM21Helmholtz}. More precisely, \cite{FHM21Helmholtz} derived a fifth-order 9-point compact FDM for piecewise constant wavenumbers, and a sixth-order 9-point compact FDM with reduced
 pollution effect for constant wavenumbers. We also provided the results of four numerical experiments  that confirm the order of convergence: Examples 3.5 and 3.6 consider constant wavenumbers and Examples 3.7 and 3.8 consider discontinuous piecewise constant wavenumbers.
\end{remark}
		\begin{remark}		
			We consider complex  interface curves  $\Gamma$  in  \cref{hybrid:ex2,hybrid:ex4} (an eight-star interface and a ten-star interface). From \cref{hybrid:8and10Stars}, we observe that when the $J$ is small (i.e., the mesh size $h$ is coarse), the interface curves  $\Gamma$ have large curvatures within some 13-point stencils. So the errors observed in \cref{hybrid:ex2,hybrid:ex4} are large when $J$ is small (see  \cref{hybrid:table:ex1}). Motivated by the pollution minimization strategy in \cite{FHM21Helmholtz},  we plan to propose a new technique to tackle this issue in our future work.
\end{remark}
\begin{figure}[htbp]
	\centering
	\begin{subfigure}[b]{0.4\textwidth}
		\includegraphics[width=6cm,height=6cm]{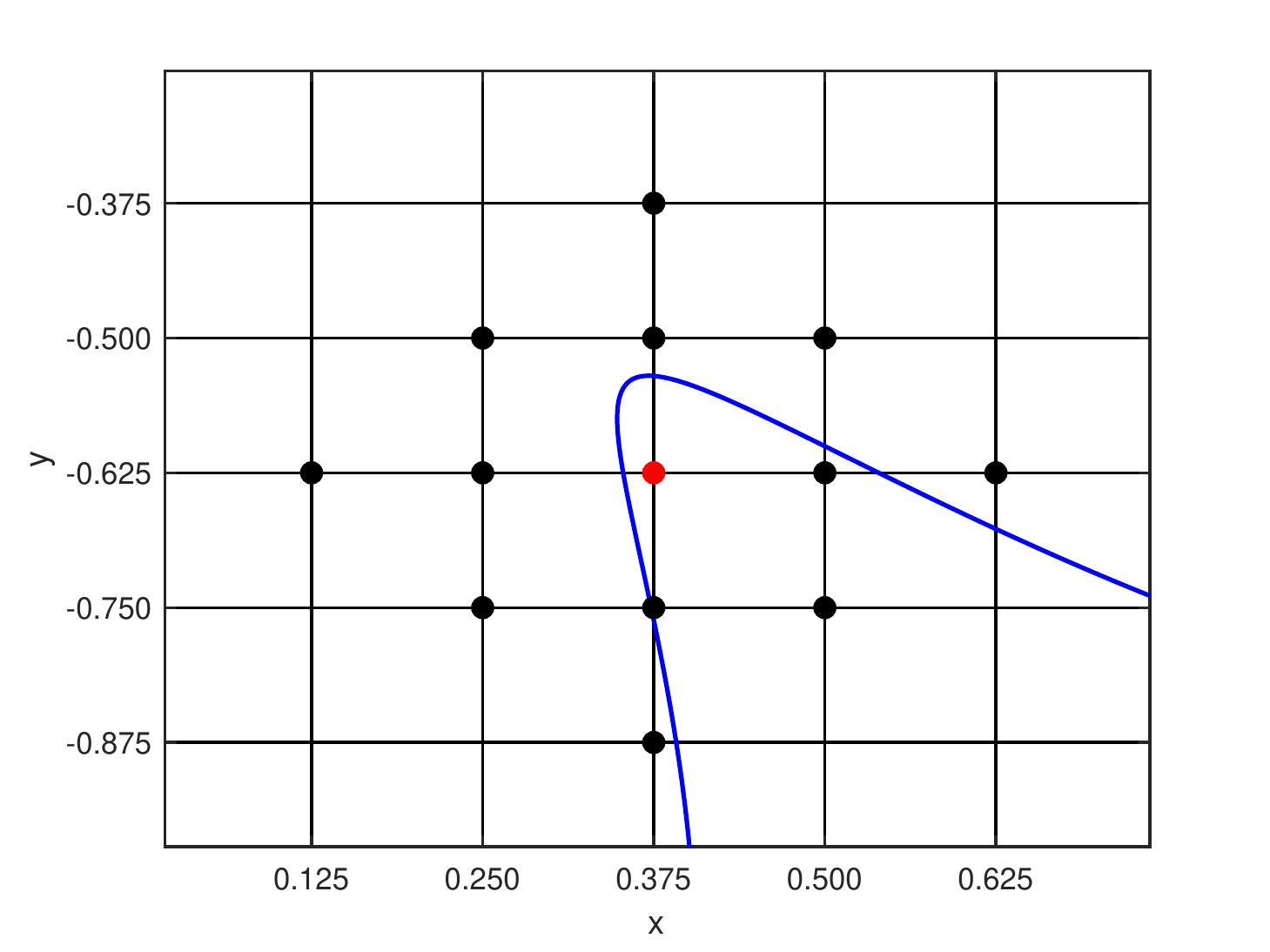}
	\end{subfigure}
	\begin{subfigure}[b]{0.4\textwidth}
		\includegraphics[width=6cm,height=6cm]{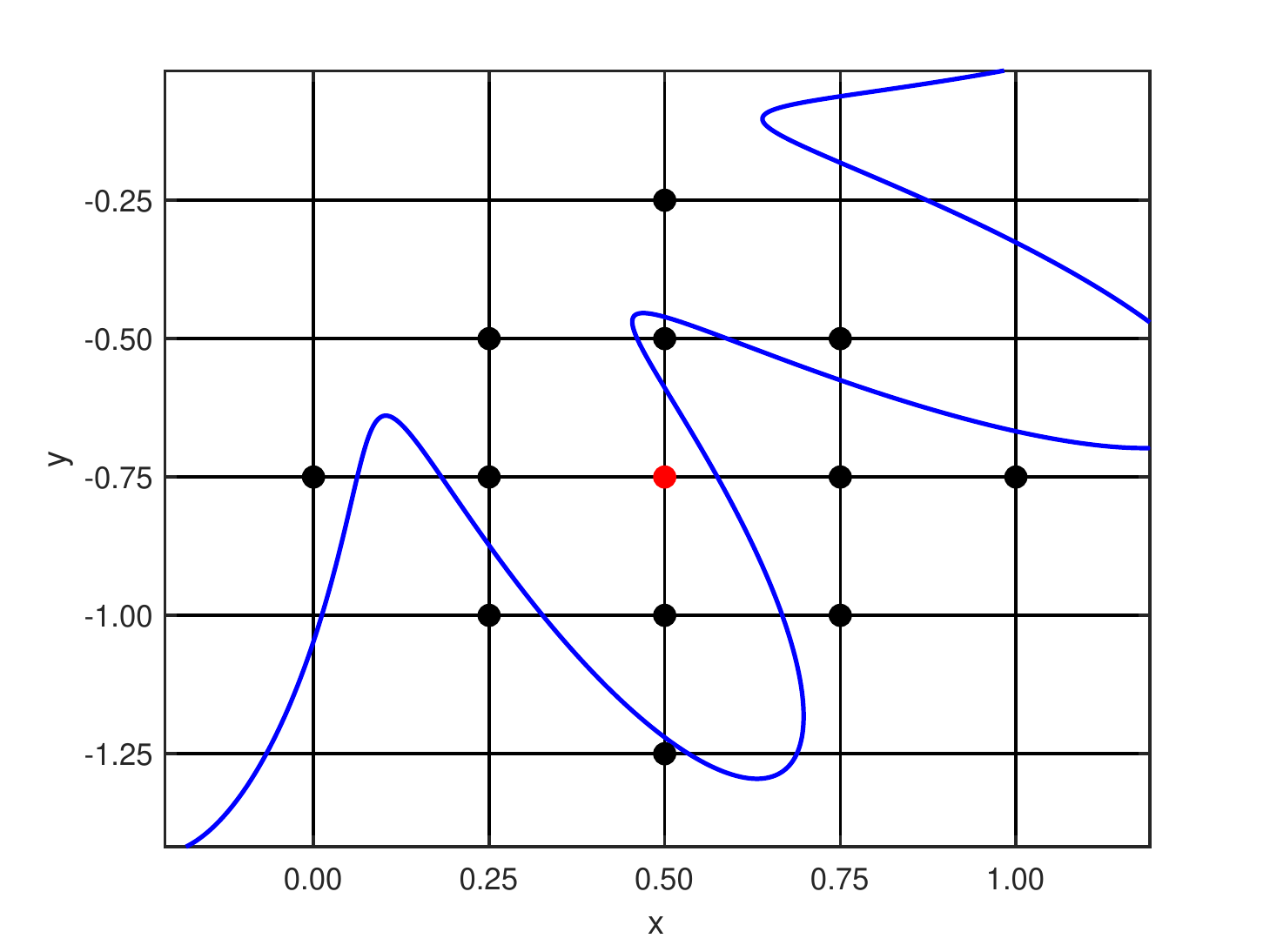}
	\end{subfigure}
		\hspace{-1cm}
	\caption
	{The 13-point scheme at the irregular point $(x_i,y_j)=(0.375,-0.625)$ for \cref{hybrid:ex2} with $J=5$ (left) and the 13-point scheme at the irregular point $(x_i,y_j)=(0.5,-0.75)$ for \cref{hybrid:ex4} with $J=4$ (right).
	}
	\label{hybrid:8and10Stars}
\end{figure}	
\section{Conclusion}\label{sec:hybrid:Conclu}
\noindent Regarding the proposed sixth-order compact FDM at regular points:
\begin{itemize}	
	
	\item	To our best knowledge,  thus far in the literature there were no schemes for solving  the elliptic equation $-\nabla \cdot( a\nabla u)=f$ with a scalar variable coefficient $a>0$, that could achieve fourth and higher order consistency and the M-matrix property for any $h>0$. We prove the existence of  a 9-point compact FDM with the sixth-order consistency, satisfying the M-matrix property for any $h>0$ without any mesh constraints and we explicitly construct such a scheme.
	
	\item
	We also derive the $6$-point/$4$-point compact FDMs with the sixth-order consistency and satisfying the M-matrix property for any $h>0$ at the side/corner boundary points for any mixed boundary conditions  under the proposed necessary condition of the boundary functions $\alpha, \beta$.
	
\item  The matrices ($A_d$ with $d=0,\dots,6$ or 7, see \eqref{A0:Regular}--\eqref{A1:A7:regular}, \eqref{A0:Gamma1}-\eqref{Ad:Gamma1} and \eqref{A0:Corner1}--\eqref{Ad:Corner1}) of linear systems to derive the above FDMs are constant matrices.   High order (partial) derivatives of $a,\alpha, \beta$ only appear on the right hand side of  linear systems.
\end{itemize}
Regarding the proposed fifth-order $13$-point FDM at irregular points:
\begin{itemize}	
	
\item	We propose a 13-point FDM  with the fifth-order consistency,  and our numerical experiments confirm the sixth-order  convergence in the $l_{\infty}$ norm of  the combined hybrid scheme.

\item To obtain the stencil coefficients of the 13-point FDM with the fifth-order consistency,
we explicitly derive a recursive solver to decompose the original linear system  into six smaller ones.  This significantly reduces the computational cost and make the implementation efficient.

\item The matrices ($A_d$ with $d=0,\dots,5$, see \eqref{Matrix:Form:irregular}--\eqref{A0:example}) of  linear systems to obtain the stencil coefficients of the above FDM only depend on  $a_{\pm}^{(0,0)}$, the first derivative of $\Gamma$ and $(v_0,w_0)$ in \eqref{base:pt:gamma}.   High order (partial) derivatives of $a_{\pm}$ and $\Gamma$ only appear on the right hand side of  linear systems.

\end{itemize}
Finally,
our proposed FDMs are independent of the choice of local parametric equations (or the level set functions) of $\Gamma$ near the base point and can handle the jump functions $\gd,\gn$ along $\Gamma$ which only depend on the geometry (such as curvature) of $\Gamma$.
Moreover,
we only use function values to numerically approximate (without losing accuracy) high order (partial) derivatives of the coefficient $a$, the source term $f$, the interface curve $\Gamma$, the two jump functions on $\Gamma$, and the functions on $\partial \Omega$.

\appendix
\section{Proofs of \cref{thm:regular:interior,thm:tranmiss:interface,thm:13point:scheme} and FDMs at boundary grid points}
\label{hybrid:sec:proofs}

\subsection{Proofs of \cref{thm:regular:interior,thm:tranmiss:interface,thm:13point:scheme}}\label{proof:Appendix}
We now prove \cref{thm:regular:interior,thm:tranmiss:interface,thm:13point:scheme}  stated in \cref{sec:sixord}.

\begin{proof}[\textbf{Proof of \cref{thm:regular:interior}}]
Applying the definition of $\mathcal{L}_h$ in
\eqref{stencil:regular} to the exact solution $u$, we have
	\[
	\begin{split}
		& h^{-2}	\mathcal{L}_h u= h^{-2} \sum_{k=-1}^{1} \sum_{\ell=-1}^{1} C_{k,\ell} u(x_i + kh, y_j + \ell h),
	\end{split}
	\]
	where the stencil coefficients $C_{k,\ell}$ are defined in \eqref{Ckell}, i.e., $C_{k,\ell}:=\sum_{p=0}^{M+1}c_{k,\ell,p} h^p$.
Using the established identity \eqref{u:approx} with the particular base point $(x_i^*, y^*_j):=(x_i,y_j)$, we have
\be\label{Lh:u:regular}
\begin{split}
	& h^{-2}	\mathcal{L}_h u   = 	 h^{-2} \sum_{(m,n)\in \ind_{M+1}^{1}} u^{(m,n)} I_{m,n} + \sum_{(m,n) \in \ind_{	 M-1}} f^{(m,n)} J_{m,n}   +\bo(h^{M}),
\end{split}
\ee
where
\be\label{Imn:regular}
\begin{split}
	I_{m,n}=\sum_{k=-1}^{1} \sum_{\ell=-1}^{1} C_{k,\ell}G_{M+1,m,n}  (kh, \ell h),\qquad J_{m,n}= h^{-2}\sum_{k=-1}^{1} \sum_{\ell=-1}^{1}	 C_{k,\ell} H_{M+1,m,n}  (kh, \ell h).
\end{split}
\ee	
%
%
%
%
Then \eqref{Lh:u:regular} and \eqref{stencil:regular}
yield
$h^{-2}	\mathcal{L}_h (u- u_h)=\bo(h^{M})$,
if $I_{m,n}$ in \eqref{Imn:regular} satisfies
%
%
	%
	%
	\be \label{regular:EQ:explicit}
	\sum_{k=-1}^1 \sum_{\ell=-1}^1 C_{k,\ell} G_{M+1,m,n}(kh, \ell h)=\bo(h^{M+2}) \quad \mbox{for all}  \quad (m,n)\in \ind_{M+1}^1.
	\ee
By the symbolic calculation,
\eqref{regular:EQ:explicit} has a nontrivial solution $\{C_{k,\ell}\}_{k,\ell=-1,0,1}$ if and only if $M\le 6$.
We now derive \eqref{Thm:regular:Equation:Form}  and prove items (i) and (ii) in \cref{thm:regular:interior}.
%
	First, \eqref{regular:EQ:explicit} with $m=n=0$, \eqref{GM100}, and \eqref{Ckell} lead to
	\be\label{regular:sum:0}
	\sum\limits_{k=-1}^{1} \sum\limits_{\ell=-1}^{1}
	c_{k,\ell,p}=0, \quad \mbox{for} \quad p=0,\dots, M+1.
	\ee
	So we proved the item (i) in \cref{thm:regular:interior}.
	By \eqref{Ckell} and \eqref{Gmn}, \eqref{regular:EQ:explicit} becomes
\be\label{regular:putGmn}
\begin{split}
	&	\sum_{k=-1}^1 \sum_{\ell=-1}^1 \sum_{p=0}^{M+1} c_{k,\ell,p}h^p \Big(G_{m,n}(kh,\ell h)+\sum_{(p,q)\in \ind_{M+1}^2 \setminus \ind_{m+n}^2 } \frac{a^{u}_{p,q,m,n}}{p!q!} (kh)^{p} (\ell h)^{q}  \Big)\\
	& =\bo(h^{M+2}), \qquad \text{for all }  (m,n)\in \ind_{M+1}^1,
\end{split}
\ee
i.e.,
		\be\label{regular:separate}
		\begin{split}
			&	\sum_{k=-1}^1 \sum_{\ell=-1}^1 \sum_{p=0}^{M+1} c_{k,\ell,p}h^p G_{m,n}(kh,\ell h)\\
			&+	\sum_{k=-1}^1 \sum_{\ell=-1}^1 \sum_{s=0}^{M+1} c_{k,\ell,s}h^s  \sum_{(p,q)\in \ind_{M+1}^2 \setminus \ind_{m+n}^2 } \frac{a^{u}_{p,q,m,n}}{p!q!} k^{p} \ell^{q} h^{p+q} =\bo(h^{M+2}), \quad  \text{for all }  (m,n)\in \ind_{M+1}^1.
		\end{split}
		\ee
Since  $(p,q)\in \ind_{M+1}^2 \setminus \ind_{m+n}^2$ implies $p+q>m+n$, we have that the degree of $h$ of $ k^{p} \ell^{q} h^{p+q}$ must be greater than $m+n$, if $ k^{p} \ell^{q} h^{p+q} \ne0$. By \eqref{Hmn},
the  degree of $h$ of $G_{m,n}(kh,\ell h)$ is $m+n$, if $G_{m,n}(kh,\ell h) \ne0$.
Consider non-zero terms $h^{m+n+d}$ in \eqref{regular:separate} with $0\le d \le M+1-m-n$, we deduce
\be\label{regular:separate:hmnd}
\begin{split}
	& \sum_{k=-1}^1 \sum_{\ell=-1}^1  c_{k,\ell,d}h^d G_{m,n}(kh,\ell h)+	 \sum_{k=-1}^1 \sum_{\ell=-1}^1 \sum_{s=0}^{M+1} c_{k,\ell,s}h^s \sum_{ \substack{(p,q)\in \ind_{M+1}^2 \setminus \ind_{m+n}^2 \\p+q=m+n+d-s}} \frac{a^{u}_{p,q,m,n}}{p!q!}  k^{p} \ell^{q} h^{p+q}\\
	&=\bo(h^{M+2}),\quad \mbox{ for all } 0\le d \le M+1-m-n \mbox{ and }  (m,n)\in \ind_{M+1}^1.
\end{split}
\ee
By \eqref{Sk}--\eqref{Hmn}, we can say that \eqref{regular:separate:hmnd} is equivalent to 	 
\be\label{regular:separate:hmnd:noh:1}
\begin{split}
	& \sum_{k=-1}^1 \sum_{\ell=-1}^1  c_{k,\ell,d}G_{m,n}(k,\ell )+	 \sum_{k=-1}^1 \sum_{\ell=-1}^1 \sum_{s=0}^{M+1} c_{k,\ell,s} \sum_{ \substack{(p,q)\in \ind_{M+1}^2 \setminus \ind_{m+n}^2 \\p+q=m+n+d-s}} \frac{a^{u}_{p,q,m,n}}{p!q!}  k^{p} \ell ^{q} =0,\\
	& \mbox{ for all } 0\le d \le M+1-m-n \mbox{ and }  (m,n)\in \ind_{M+1}^1.
\end{split}
\ee
Note that $(p,q)\in \ind_{M+1}^2 \setminus \ind_{m+n}^2$ results in $m+n+1\le p+q\le M+1$.
By $p+q=m+n+d-s$,  $1\le d-s\le M+1-m-n$, so $s \le d-1$. By the definition of $\ind_{M+1}^1$ in \eqref{indV12},  $0\le d \le M+1-m-n$ and  $(m,n)\in \ind_{M+1}^1$ can be equivalently rewritten as $(m,n)\in \ind_{M+1-d}^1$  and $0 \le d \le M+1$.
So \eqref{regular:separate:hmnd:noh:1} is equivalent to
\be\label{regular:separate:hmnd:noh:2}
\begin{split}
	& \sum_{k=-1}^1 \sum_{\ell=-1}^1  c_{k,\ell,d} G_{m,n}(k,\ell )+\sum_{k=-1}^1 \sum_{\ell=-1}^1 \sum_{s=0}^{d-1} c_{k,\ell,s} 	 A^u_{k,\ell,m,n,s}=0,\\
	&\mbox{with} \quad
	A^u_{k,\ell,m,n,s}:=\sum_{ \substack{(p,q)\in \ind_{M+1}^2 \setminus \ind_{m+n}^2 \\p+q=m+n+d-s}} \frac{a^{u}_{p,q,m,n}} {p!q!} k^{p} \ell ^{q}\qquad\quad
	\mbox{ for all } (m,n)\in \ind_{M+1-d}^1,
\end{split}
\ee		
%
%
for all $d=0,\ldots, M+1$.		
By \eqref{regular:putGmn}--\eqref{regular:separate:hmnd:noh:2},	 we can say that \eqref{regular:EQ:explicit} is equivalent to  \eqref{regular:separate:hmnd:noh:2}.
Note that $\sum_{s=0}^{d-1}$ in \eqref{regular:separate:hmnd:noh:2} is empty for $d=0$.
Now \eqref{Thm:regular:Equation:Form} in \cref{thm:regular:interior} can be seen from
\eqref{regular:separate:hmnd:noh:2} with $M=6$.

	The system of linear equations in \eqref{regular:separate:hmnd:noh:2} can be represented in the following matrix form:
	\be\label{Matrix:Form:2}
	\begin{pmatrix}
		A_0 & \textbf{0} & \textbf{0} & \cdots & \textbf{0} & \textbf{0}  \\
		B_{1,0} & A_1 & \textbf{0} & \cdots & \textbf{0}  & \textbf{0}\\
		B_{2,0} & B_{2,1} & A_2 & \cdots & \textbf{0}  & \textbf{0}\\
		\vdots & \vdots  & \vdots  & \ddots & \vdots  & \vdots \\
		B_{M,0} &  B_{M,1}  & B_{M,2}  & \cdots & A_{M} & \textbf{0}\\
		\textbf{0} & \textbf{0}  & \textbf{0}  & \cdots &  \textbf{0} & A_{M+1}\\
	\end{pmatrix}
	\begin{pmatrix}
		C_0  \\
		C_1  \\
		C_2  \\
		\vdots  \\
		C_M  \\
		C_{M+1}
	\end{pmatrix}=\textbf{0}, \text{ with }
	C_d=
	{\tiny{	
			\begin{pmatrix}
				c_{-1,-1,d} \\
				c_{-1,0,d} \\
				c_{-1,1,d}  \\
				c_{0,-1,d} \\
				c_{0,0,d} \\
				c_{0,1,d}  \\
				c_{1,-1,d} \\
				c_{1,0,d} \\
				c_{1,1,d}
			\end{pmatrix},}}
	\ee
	where $0\le d \le M+1$ and
	\[
	\begin{split}
	& A_d:=[G_{m,n}(k,\ell)]_{(k,\ell)\in \{-1,0,1\}^2, (m,n)\in \ind_{M+1-d}^1}, \text{ with } 0\le d \le M+1,\\
	& B_{d,s}:=[A^u_{k,\ell,m,n,s}]_{(k,\ell)\in \{-1,0,1\}^2, (m,n)\in \ind_{M+1-d}^1}, \text{ with } 1\le d \le M  \text{ and }  0\le s \le d-1.
\end{split}
	\]
The $G_{m,n}(x,y)$ in \eqref{Hmn} and $a^u_{p,q,m,n}$ in \eqref{upq1} imply that that every $A_d$ is a constant matrix, and every matrix $B_{d,s}$ only depends on $\{a^{(i,j)}: (i,j) \in \ind_{M}\}$.    See \eqref{A0:Regular} and \eqref{A1:A7:regular}  for $\{A_d: 0\le d\le 7 \}$ with $M=6$.  Clearly, \eqref{regular:EQ:explicit} is equivalent to
	\eqref{Matrix:Form:2}. For the sake of brevity, let $\textsf{A}\textsf{C}=\textbf{0}$ represent the whole homogeneous linear system in  \eqref{Matrix:Form:2} with $M=6$. Then
 	$\text{rank}(\textsf{A})=48$ and the dimension of the solution $\textsf{C}$ is 24 by the symbolic calculation. Clearly, \eqref{Matrix:Form:2} implies
	\be\label{Matrix:Form:1}
A_0 C_0=\textbf{0},\quad A_d C_d=b_d 	\text{ with }
b_d:=-\sum_{s=0}^{d-1}B_{d,s}C_{s} \text{ for } 1\le d\le M,
\quad \mbox{and}\quad A_{M+1} C_{M+1}=0.
	\ee
 So	
	$b_d$ depends on $\{ C_i : 0\le i \le d-1 \}$ and $\{a^{(i,j)}: (i,j) \in \ind_{M}\}$ for $1\le d\le M$.   By a direct calculation, we observe that  the dimension of the solution of
	\eqref{Matrix:Form:1} is 24 for $M=6$ (see the following   \eqref{ckl0:degree:0}, \eqref{ckl1:degree:1}, \eqref{ckl2:degree:2}, \eqref{ckl4:degree:4}, \eqref{ckl5:degree:5}, \eqref{ckl6:degree:6} and \eqref{ckl7:degree:7}).
Now	we can say that \eqref{Matrix:Form:2}  is equivalent to \eqref{Matrix:Form:1} for $M=6$, i.e., \eqref{regular:EQ:explicit}  is equivalent to \eqref{Matrix:Form:1} for $M=6$.

We now prove item (ii) in \cref{thm:regular:interior} by establishing the following \eqref{Matrix:Form:separate}--\eqref{ckl7:degree:7:interval}.
By $h> 0$,  the nontrivial $\{C_{k,\ell}\}_{k,\ell=-1,0,1}$ in \eqref{stencil:regular} satisfies the sign condition \eqref{sign:condition} for any $h$ if it satisfies \eqref{suff:nece:sign:cond}.
By \eqref{Matrix:Form:1} with $M=6$, we have
\be\label{Matrix:Form:separate}
A_0C_0=\textbf{0},\qquad A_dC_d=b_d,\quad d=1,\dots, 6,\qquad A_7C_7=0.
\ee
All solutions of $A_0C_0=\textbf{0}$ in \eqref{Matrix:Form:separate} with $M=6$ can be represented as
{\small{
\be\label{ckl0:degree:0}
\begin{split}
& c_{0,0,0}=-20c_{1,1,0},\quad
c_{-1,0,0}=c_{1,0,0}=c_{0,-1,0}=c_{0,1,0}=4c_{1,1,0},\quad c_{-1,-1,0}=c_{-1,1,0}=c_{1,-1,0}=c_{1,1,0},
\end{split}
\ee}}
where $c_{1,1,0}$ is free parameter.
Then $\{c_{k,\ell,0}\}_{k,\ell=-1,0,1}$ in \eqref{ckl0:degree:0} satisfies the condition in \eqref{suff:nece:sign:cond}  if and only if $c_{1,1,0}<0$.
All solutions of $A_1C_1=b_1$ in \eqref{Matrix:Form:separate} with $M=6$ can be represented as
{\small{
\be\label{ckl1:degree:1}
\begin{split}
& c_{0,0,1}=-20c_{1,1,1}+r_{1,1},\quad c_{-1,0,1}=4c_{1,1,1}+r_{1,2}, \quad  c_{1,0,1}=4c_{1,1,1}+r_{1,3}, \quad c_{0,-1,1}=4c_{1,1,1}+r_{1,4},\\
&   c_{0,1,1}=4c_{1,1,1}+r_{1,5},\quad  c_{-1,-1,1}=c_{1,1,1}+r_{1,6}, \quad c_{-1,1,1}=c_{1,1,1}+r_{1,7}, \quad c_{1,-1,1}=c_{1,1,1}+r_{1,8},
\end{split}
\ee}}
where $c_{1,1,1}$ is the free parameter, $\{r_{1,p}\}_{p=1,\dots,8}$ only depends on  $\{c_{k,\ell,0}\}_{k,\ell=-1,0,1}$ in \eqref{ckl0:degree:0} and $\{a^{(i,j)}: (i,j) \in \ind_{6}\}$ (a special \eqref{ckl1:degree:1} with explicit expressions is shown in \eqref{particular:0}-\eqref{particular:6}).
Then $\{c_{k,\ell,1}\}_{k,\ell=-1,0,1}$ in \eqref{ckl1:degree:1} satisfies the condition in \eqref{suff:nece:sign:cond}  if and only if
\be\label{ckl1:degree:1:interval}
 c_{1,1,1}\le \min \{\tfrac{r_{1,1}}{20}, \tfrac{-r_{1,2}}{4},\tfrac{-r_{1,3}}{4},\tfrac{-r_{1,4}}{4},\tfrac{-r_{1,5}}{4}, -r_{1,6}, -r_{1,7}, -r_{1,8},0 \}.
\ee
All solutions of $A_dC_d=b_d$ with $d=2,3$ in \eqref{Matrix:Form:separate} with $M=6$ can be represented as
{\small{
\be\label{ckl2:degree:2}
\begin{split}
	& c_{0,0,d}=-20c_{1,1,d}+r_{d,1},\quad c_{-1,0,d}=4c_{1,1,d}+r_{d,2}, \quad  c_{1,0,d}=4c_{1,1,d}+r_{d,3}, \quad c_{0,-1,d}=4c_{1,1,d}+r_{d,4},\\
	&   c_{0,1,d}=4c_{1,1,d}+r_{d,5},\  c_{-1,-1,d}=c_{1,1,d}+r_{d,6}, \  c_{-1,1,d}=c_{1,1,d}+r_{d,7}, \  c_{1,-1,d}=c_{1,1,d}+r_{d,8}, \  d=2,3,
\end{split}
\ee}}
where $c_{1,1,d}$  with $d=2,3$ is the free parameter, $\{r_{d,p}\}_{p=1,\dots,8}$  only depends on  $\{c_{k,\ell,s}\}_{k,\ell=-1,0,1}$ with $0\le s\le d-1$ in \eqref{ckl0:degree:0}, \eqref{ckl1:degree:1}, \eqref{ckl2:degree:2}  and $\{a^{(i,j)}: (i,j) \in \ind_{6}\}$ for $d=2,3$ (a special \eqref{ckl2:degree:2} with explicit expressions is shown in \eqref{particular:0}-\eqref{particular:6}).
Then $\{c_{k,\ell,d}\}_{k,\ell=-1,0,1}$ in \eqref{ckl2:degree:2} satisfies the condition in \eqref{suff:nece:sign:cond}  if and only if
\be\label{ckl2:degree:2:interval}
\begin{split}
 c_{1,1,d}\le \min \{\tfrac{r_{d,1}}{20}, \tfrac{-r_{d,2}}{4},\tfrac{-r_{d,3}}{4},\tfrac{-r_{d,4}}{4},\tfrac{-r_{d,5}}{4}, -r_{d,6}, -r_{d,7}, -r_{d,8},0  \},\quad d=2,3.
\end{split}
\ee
All solutions of $A_4C_4=b_4$ in \eqref{Matrix:Form:separate} with $M=6$ can be represented as
\be\label{ckl4:degree:4}
\begin{split}
	& c_{0,0,4}=-4c_{1,0,4}-4c_{1,1,4}+r_{4,1},\quad c_{-1,0,4}=c_{1,0,4}+r_{4,2}, \quad c_{0,-1,4}=c_{1,0,4}+r_{4,3},\\
	&   c_{0,1,4}=c_{1,0,4}+r_{4,4},\quad c_{-1,-1,4}=c_{1,1,4}+r_{4,5}, \quad c_{-1,1,4}=c_{1,1,4}+r_{4,6}, \quad c_{1,-1,4}=c_{1,1,4}+r_{4,7},
\end{split}
\ee
where $\{ c_{1,0,4}, c_{1,1,4} \}$ are the free parameters, $\{r_{4,p}\}_{p=1,\dots,7}$ only depends on  $\{c_{k,\ell,s}\}_{k,\ell=-1,0,1}$ with $0\le s\le 3$ in \eqref{ckl0:degree:0}, \eqref{ckl1:degree:1}, \eqref{ckl2:degree:2}  and $\{a^{(i,j)}: (i,j) \in \ind_{6}\}$ (a special \eqref{ckl4:degree:4} with explicit expressions is shown in \eqref{particular:0}-\eqref{particular:6}).
Then $\{c_{k,\ell,4}\}_{k,\ell=-1,0,1}$ in \eqref{ckl4:degree:4} satisfies the condition  \eqref{suff:nece:sign:cond}  if and only if
{\small{
\be\label{ckl4:degree:4:interval}
\begin{split}
	& c_{1,0,4}+c_{1,1,4}\le \tfrac{r_{4,1}}{4},\quad c_{1,0,4}\le \min \{ -r_{4,2}, -r_{4,3}, -r_{4,4},0 \},\quad c_{1,1,4}\le \min \{ -r_{4,5}, -r_{4,6}, -r_{4,7},0 \}.
\end{split}
\ee}}
One non-empty interval of \eqref{ckl4:degree:4:interval} is
\be\label{ckl4:degree:4:interval:special}
\begin{split}
	 c_{1,0,4}=c_{1,1,4},  \qquad c_{1,1,4}\le \min \{ \tfrac{r_{4,1}}{8}, -r_{4,2}, -r_{4,3}, -r_{4,4},-r_{4,5}, -r_{4,6}, -r_{4,7},0 \}.
\end{split}
\ee
All solutions of $A_5C_5=b_5$ in \eqref{Matrix:Form:separate} with $M=6$ can be represented as
\be\label{ckl5:degree:5}
\begin{split}
	& c_{0,0,5}=-2c_{0,1,5}-2c_{1,0,5}-4c_{1,1,5}+r_{5,1},\\
	& c_{-1,0,5}=2c_{0,1,5}-2c_{1,-1,5}-c_{1,0,5}+2c_{1,1,5}+r_{5,2}, \quad c_{0,-1,5}=c_{0,1,5}-2c_{1,-1,5}+2c_{1,1,5}+r_{5,3}, \\
	& c_{-1,-1,5}=-c_{0,1,5}+2c_{1,-1,5}+c_{1,0,5}-c_{1,1,5}+r_{5,4}, \quad c_{-1,1,5}=-c_{0,1,5}+c_{1,-1,5}+c_{1,0,5}+r_{5,5},
\end{split}
\ee
where $\{ c_{0,1,5}, c_{1,-1,5}, c_{1,0,5}, c_{1,1,5} \}$ are the free parameters, $\{r_{5,p}\}_{p=1,\dots,5}$ only depends on  $\{c_{k,\ell,s}\}_{k,\ell=-1,0,1}$ with $0\le s\le 4$ in \eqref{ckl0:degree:0}, \eqref{ckl1:degree:1}, \eqref{ckl2:degree:2}, \eqref{ckl4:degree:4}  and $\{a^{(i,j)}: (i,j) \in \ind_{6}\}$ (a special \eqref{ckl5:degree:5} with explicit expressions is shown in \eqref{particular:0}-\eqref{particular:6}).
Then $\{c_{k,\ell,5}\}_{k,\ell=-1,0,1}$ in \eqref{ckl5:degree:5} satisfies the condition in \eqref{suff:nece:sign:cond}  if and only if
\be\label{ckl5:degree:5:interval}
\begin{split}
	& -2c_{0,1,5}-2c_{1,0,5}-4c_{1,1,5}+r_{5,1}\ge 0, \quad 2c_{0,1,5}-2c_{1,-1,5}-c_{1,0,5}+2c_{1,1,5}+r_{5,2}\le 0,\\
	& c_{0,1,5}-2c_{1,-1,5}+2c_{1,1,5}+r_{5,3} \le 0, \quad-c_{0,1,5}+2c_{1,-1,5}+c_{1,0,5}-c_{1,1,5}+r_{5,4} \le 0,\\
	& -c_{0,1,5}+c_{1,-1,5}+c_{1,0,5}+r_{5,5} \le 0, \quad  c_{0,1,5}\le 0, \quad  c_{1,-1,5}\le 0, \quad c_{1,0,5}\le 0, \quad c_{1,1,5}\le 0.
\end{split}
\ee
One non-empty interval of \eqref{ckl5:degree:5:interval} is
\be\label{ckl5:degree:5:interval:special}
\begin{split}
	& c_{0,1,5}=c_{1,-1,5}=c_{1,0,5}=c_{1,1,5}, \qquad c_{1,1,5}\le \min \{\tfrac{r_{5,1}}{8}, -r_{5,2}, -r_{5,3}, -r_{5,4},-r_{5,5}, 0 \}.
\end{split}
\ee
All solutions of $A_6C_6=b_6$ in \eqref{Matrix:Form:separate} with $M=6$ can be represented as
\be\label{ckl6:degree:6}
\begin{split}
	& c_{-1,0,6}=-2c_{-1,1,6}-c_{0,0,6}-2c_{0,1,6}-c_{1,0,6}-2c_{1,1,6}+r_{6,1},\\
	& c_{0,-1,6}=-c_{0,0,6}-c_{0,1,6}-2c_{1,-1,6}-2c_{1,0,6}-2c_{1,1,6}+r_{6,2},\\
	& c_{-1,-1,6}=c_{-1,1,6}+c_{0,0,6}+2c_{0,1,6}+c_{1,-1,6}+2c_{1,0,6}+3c_{1,1,6}+r_{6,3},
\end{split}
\ee
where $\{  c_{-1,1,6}, c_{0,0,6}, c_{0,1,6}, c_{1,-1,6}, c_{1,0,6}, c_{1,1,6} \}$ are the free parameters, $\{r_{6,p}\}_{p=1,2,3}$ only depends on  $\{c_{k,\ell,s}\}_{k,\ell=-1,0,1}$ with $0\le s\le 5$ in \eqref{ckl0:degree:0}, \eqref{ckl1:degree:1}, \eqref{ckl2:degree:2}, \eqref{ckl4:degree:4}, \eqref{ckl5:degree:5} and $\{a^{(i,j)}: (i,j) \in \ind_{6}\}$ (a special \eqref{ckl6:degree:6} with explicit expressions is shown in \eqref{particular:0}-\eqref{particular:6}).
Then $\{c_{k,\ell,6}\}_{k,\ell=-1,0,1}$ in \eqref{ckl6:degree:6} satisfies the condition in \eqref{suff:nece:sign:cond}  if and only if
\be\label{ckl6:degree:6:interval}
\begin{split}
& -2c_{-1,1,6}-c_{0,0,6}-2c_{0,1,6}-c_{1,0,6}-2c_{1,1,6}+r_{6,1}\le 0,\\
& -c_{0,0,6}-c_{0,1,6}-2c_{1,-1,6}-2c_{1,0,6}-2c_{1,1,6}+r_{6,2} \le 0,\\
& c_{-1,1,6}+c_{0,0,6}+2c_{0,1,6}+c_{1,-1,6}+2c_{1,0,6}+3c_{1,1,6}+r_{6,3} \le 0,\\
& c_{-1,1,6}\le 0, \quad c_{0,0,6}\ge 0,\quad c_{0,1,6}\le 0, \quad c_{1,-1,6}	 \le 0, \quad c_{1,0,6}	\le 0, \quad c_{1,1,6}	\le 0.
\end{split}
\ee	
One non-empty interval of \eqref{ckl6:degree:6:interval} is
{\small{
\be\label{ckl6:degree:6:interval:special}
\begin{split}
	 & c_{-1,1,6}=c_{0,1,6}=c_{1,-1,6}=c_{1,0,6}=c_{1,1,6},\quad c_{0,0,6}=-8c_{1,1,6}, \quad c_{1,1,6}\le \min \{ -r_{6,1}, -r_{6,2}, -r_{6,3}, 0\}.
\end{split}
\ee}}
All solutions of $A_7C_7=0$ in \eqref{Matrix:Form:separate} with $M=6$ are
\be\label{ckl7:degree:7}
c_{0,0,7}=-c_{-1,0,7}-c_{1,0,7}-c_{0,-1,7}-c_{0,1,7}-c_{-1,-1,7}-c_{-1,1,7}-c_{1,-1,7}-c_{1,1,7},
\ee
where $\{ c_{-1,0,7}, c_{1,0,7}, c_{0,-1,7}, c_{0,1,7}, c_{-1,-1,7}, c_{-1,1,7}, c_{1,-1,7}, c_{1,1,7} \}$ are free parameters.
Then $\{c_{k,\ell,7}\}_{k,\ell=-1,0,1}$ in \eqref{ckl7:degree:7} satisfies the condition in \eqref{suff:nece:sign:cond}  if and only if
\be\label{ckl7:degree:7:interval}
\begin{split}
&c_{0,0,7}=-c_{-1,0,7}-c_{1,0,7}-c_{0,-1,7}-c_{0,1,7}-c_{-1,-1,7}-c_{-1,1,7}-c_{1,-1,7}-c_{1,1,7} \ge 0,\quad c_{1,1,7}\le0, \\
&  c_{-1,0,7}\le0, \quad c_{1,0,7}\le0, \quad c_{0,-1,7}\le0, \quad c_{0,1,7}\le0, \quad c_{-1,-1,7}\le0, \quad c_{-1,1,7}\le0, \quad c_{1,-1,7}\le0.
\end{split}
\ee	
One non-empty interval of \eqref{ckl7:degree:7:interval} is
$
 c_{k,\ell,7}=0$ for $k,\ell\in \{-1,0,1\}$.
By the symbolic calculation,  all  $r_{d,p}$ $\ne \pm\infty$ in \eqref{ckl1:degree:1:interval}--\eqref{ckl6:degree:6:interval:special}  by $a\ne 0$ in $\Omega$.
By \eqref{Matrix:Form:separate}--\eqref{ckl7:degree:7:interval}, we proved the item (ii) in \cref{thm:regular:interior}.
\end{proof}	
\noindent
\textbf{The explicit expressions for a particular $\{C_{k,\ell}\}$ in \cref{thm:regular:interior}:}
Let $a(x,y)$ be a linear function (i.e., all $a^{(m,n)}$ are 0 for $m+n\ge 2$), $r_1=a^{(1,0)}/a^{(0,0)}$ and $r_2=a^{(0,1)}/a^{(0,0)}$.
Then the explicit expressions of $C_{k,\ell}:=\sum_{p=0}^{7} c_{k,\ell,p}h^p$ in \eqref{stencil:regular} with the sixth-order  consistency and satisfying the sign condition \eqref{sign:condition} and the sum condition \eqref{sum:condition} for any mesh size $h$  are defined in the following \eqref{particular:0}--\eqref{particular:6} (we highlight the free parameters using the color red or blue to increase the visibility):
		\be\label{particular:0}
		\begin{split}
		&
		c_{0,0,0}=20,\   c_{-1,0,0}= c_{1,0,0}= c_{0,-1,0}=c_{0,1,0}=-4,\  c_{-1,-1,0}=c_{-1,1,0}=c_{1,-1,0}=c_{1,1,0}=-1,\\
		& c_{0,0,d}=-20c_{d}+r_{d,1},\quad c_{-1,0,d}=4c_{d}+r_{d,2}, \quad  c_{1,0,d}=4c_{d}+r_{d,3}, \quad c_{0,-1,d}=4c_{d}+r_{d,4}, \\ & c_{0,1,d}=4c_{d}+r_{d,5},\quad
		c_{-1,-1,d}=c_{d}+r_{d,6}, \quad c_{-1,1,d}=c_{d}+r_{d,7}, \quad c_{1,-1,d}=c_{d}+r_{d,8},\\
		& c_{1,1,d}=c_{d}, \quad	 \textcolor{red}{c_{d}}= \min\{r_{d,1}/20, -r_{d,2}/4,-r_{d,3}/4,-r_{d,4}/4,-r_{d,5}/4, -r_{d,6}, -r_{d,7}, -r_{d,8},0  \},
		\end{split}
		\ee
for $d=1,2,3$, where
		\[
		\begin{split}
		&	r_{1, 1} = -10(r_{1}+r_{2}),\quad  r_{1, 2} = 4r_{1}+2r_{2},\quad  r_{1, 3} = 2r_{2},\quad  r_{1, 4} = 2r_{1}+4r_{2},\\
		&  r_{1, 5} = 2r_{1}, \quad r_{1, 6} = r_{1}+r_{2},\quad  r_{1, 7} = r_{1},\quad  r_{1, 8} = r_{2},\\
		&  r_{2, p} = -r_{1, p}\textcolor{blue}{c_{1}}+r_{p+2},   \quad r_{3, p} = -r_{p+2}\textcolor{blue}{c_{1}}-r_{1, p}\textcolor{blue}{c_{2}}+r_{10+p},  \quad \text{for} \quad p=1,\dots,8,
		\end{split}
		\]
 and
{\footnotesize{
		\[
		\begin{split}
			&   r_{3} = 29/5(r_{1}^2+r_{2}^2)+15r_{1}r_{2},\quad r_{4} = -(39r_{1}^2+29r_{2}^2)/20-4r_{1}r_{2},\quad r_{5} = (r_{1}^2-29r_{2}^2)/20-2r_{1}r_{2},\\
			&   r_{6} = -(39r_{2}^2+29r_{1}^2)/20-4r_{1}r_{2},\quad r_{7} = (r_{2}^2-29r_{1}^2)/20-2r_{1}r_{2},\quad r_{8} = -(r_{1}+r_{2})^2/2, \\
			& r_{9} = -r_{1}(2r_{2}+r_{1})/2, \quad r_{10} = -r_{2}(2r_{1}+r_{2})/2,\quad  r_{11}=-127/30 (r_{1}^3+r_{2}^3)-251/15 (r_{1}^2r_{2}+r_{1}r_{2}^2),\\
			& r_{12}=(119r_{2}^3+569r_{1}^2r_{2})/120+(49r_{1}r_{2}^2+19r_{1}^3)/12, \quad r_{13}=(119r_{2}^3+209r_{1}^2r_{2})/120+(29r_{1}r_{2}^2-r_{1}^3)/10, \\
			&  r_{14}=(119r_{1}^3+569r_{1}r_{2}^2)/120+(49r_{1}^2r_{2}+19r_{2}^3)/12, \quad r_{15}=(119r_{1}^3+209r_{1}r_{2}^2)/120+(29r_{1}^2r_{2}-r_{2}^3)/10,\\
			& r_{16}=23/60 (r_{1}^3+r_{2}^3)+83/60(r_{1}^2r_{2}+r_{1}r_{2}^2),\  r_{17}=(23r_{1}^3+53r_{1}r_{2}^2)/60+r_{1}^2r_{2},\  r_{18}=(23r_{2}^3+53r_{1}^2r_{2})/60+r_{1}r_{2}^2.
		\end{split}
		\]}}
		\be
		\begin{split}
			& c_{0,0,4}=-8c_{4}+r_{4,1},\quad c_{-1,0,4}=c_{4}+r_{4,2}, \quad c_{0,-1,4}=c_{4}+r_{4,3}, \quad c_{0,1,4}=c_{4}+r_{4,4},\\
			& c_{-1,-1,4}=c_{4}+r_{4,5}, \quad c_{-1,1,4}=c_{4}+r_{4,6},\quad c_{1,-1,4}=c_{4}+r_{4,7},\quad  c_{1,0,4}=c_{1,1,4}=c_{4},\\
			& \textcolor{red}{c_{4}}=\min \{ r_{4,1}/8, -r_{4,2}, -r_{4,3}, -r_{4,4},-r_{4,5}, -r_{4,6}, -r_{4,7},0 \},
		\end{split}
		\ee
where
{\tiny{
		\[
		\begin{split}
			r_{4, 1} =	& (139r_{1}^3+293r_{1}^2r_{2}+154r_{1}r_{2}^2+8r_{2}^3)/30 \textcolor{blue}{c_{1}}-(6r_{1}^2+7r_{1}r_{2}) \textcolor{blue}{c_{2}}+2(5r_{1}+r_{2})\textcolor{blue}{c_{3}}+(413r_{1}^4-23r_{2}^4)/120+23/2r_{1}^2r_{2}^2\\
			&+(34r_{1}^3r_{2}+13r_{1}r_{2}^3)/3, \\
			r_{4, 2} =& (-(101r_{1}^3+71r_{1}r_{2}^2)/60-3r_{1}^2r_{2})\textcolor{blue}{c_{1}}+2r_{1}(r_{1}+r_{2})\textcolor{blue}{c_{2}}-4r_{1}\textcolor{blue}{c_{3}}-(133r_{1}^4+433r_{1}^3r_{2}+403r_{1}^2r_{2}^2+103r_{1}r_{2}^3)/120,\\
			r_{4, 3} =& (-131r_{1}^3-281r_{1}^2r_{2}-221r_{1}r_{2}^2-71r_{2}^3)/120\textcolor{blue}{c_{1}}+(3r_{1}+r_{2})(r_{1}+r_{2})/2\textcolor{blue}{c_{2}}-2(r_{1}+r_{2})\textcolor{blue}{c_{3}}-1/5r_{2}^4 -(109r_{1}^4\\
			 &+403r_{1}^2r_{2}^2+223r_{1}r_{2}^3+313r_{1}^3r_{2})/120, \\
			r_{4, 4} = & (131(r_{2}^3-r_{1}^3)+139(r_{1}r_{2}^2-r_{1}^2r_{2}))/120\textcolor{blue}{c_{1}}+3/2(r_{1}^2-r_{2}^2)\textcolor{blue}{c_{2}}+2(r_{2}-r_{1})\textcolor{blue}{c_{3}}+109/120(r_{2}^4-r_{1}^4)+7/4(r_{1}r_{2}^3 -r_{1}^3r_{2}), \\
			r_{4, 5} = & -23/60(r_{1}^2+60/23r_{1}r_{2}+r_{2}^2)(r_{1}+r_{2})\textcolor{blue}{c_{1}}+(r_{1}+r_{2})^2/2\textcolor{blue}{c_{2}}-(r_{1}+r_{2})\textcolor{blue}{c_{3}}-31/120(r_{1}^2+90/31r_{1}r_{2}+r_{2}^2)(r_{1}+r_{2})^2, \\
			r_{4, 6} = & (-(53r_{1}r_{2}^2+23r_{1}^3)/60-r_{1}^2r_{2})\textcolor{blue}{c_{1}}+(r_{1}+2r_{2})r_{1}/2\textcolor{blue}{c_{2}}-r_{1}\textcolor{blue}{c_{3}}-31/120r_{1}^4-13/10r_{1}^3r_{2}-83/60r_{1}^2r_{2}^2-4/5r_{1}r_{2}^3, \\
			r_{4, 7} = & (-(53r_{1}^2r_{2}+23r_{2}^3)/60-r_{1}r_{2}^2)\textcolor{blue}{c_{1}}+(r_{2}+2r_{1})r_{2}/2\textcolor{blue}{c_{2}}-r_{2}\textcolor{blue}{c_{3}}-31/120r_{2}^4-13/10r_{1}r_{2}^3-83/60r_{1}^2r_{2}^2-4/5r_{1}^3r_{2}.
		\end{split}
		\]}}
		\be
		\begin{split}
			& c_{0,0,5}=-8c_{5}+r_{5,1},\quad c_{-1,0,5}=c_{5}+r_{5,2}, \quad c_{0,-1,5}=c_{5}+r_{5,3}, \quad c_{-1,-1,5}=c_{5}+r_{5,4},\\
			&    c_{-1,1,5}=c_{5}+r_{5,5}, \quad c_{0,1,5}=c_{1,-1,5}=c_{1,0,5}=c_{1,1,5}=c_{5}, \\
			& \textcolor{red}{c_{5}}= \min \{r_{5,1}/8, -r_{5,2}, -r_{5,3}, -r_{5,4},-r_{5,5}, 0 \},
		\end{split}
		\ee
where
{\tiny{
		\[
		\begin{split}
			r_{5, 1} = & (-207r_{1}^4-743r_{1}^3r_{2}-823r_{1}^2r_{2}^2-473r_{1}r_{2}^3-76r_{2}^4)/120\textcolor{blue}{c_{1}}+(5/2r_{1}^3+11/2r_{1}^2r_{2}+4r_{1}r_{2}^2+r_{2}^3)\textcolor{blue}{c_{2}}-(3r_{1}^2+5r_{1}r_{2}+r_{2}^2)\textcolor{blue}{c_{3}}\\
			& +3(r_{1}+r_{2})\textcolor{blue}{c_{4}}-73/48r_{1}^5-749/120r_{1}^4r_{2}-521/48r_{1}^3r_{2}^2-437/48r_{1}^2r_{2}^3-43/12r_{1}r_{2}^4-49/80r_{2}^5, \\
			r_{5, 2} = & (217r_{1}^3r_{2}-85r_{1}^4+526r_{1}^2r_{2}^2+487r_{1}r_{2}^3+161r_{2}^4)/120\textcolor{blue}{c_{1}}+(r_{1}^3/2-3/2r_{2}^3-7/2r_{1}r_{2}^2-5/2r_{1}^2r_{2})\textcolor{blue}{c_{2}}+(2r_{2}^2-r_{1}^2+4r_{1}r_{2})\textcolor{blue}{c_{3}}\\
			& -3r_{2}\textcolor{blue}{c_{4}}+31/30r_{2}^5+241/60r_{1}r_{2}^4+869/120r_{1}^2r_{2}^3+1133/240r_{1}^3r_{2}^2+23/24r_{1}^4r_{2}-101/240r_{1}^5,\\
			r_{5, 3} = & (76r_{2}^4+397r_{1}r_{2}^3+526r_{1}^2r_{2}^2+307r_{1}^3r_{2})/120\textcolor{blue}{c_{1}}-3r_{2}(r_{1}^2+r_{1}r_{2}+1/3r_{2}^2)\textcolor{blue}{c_{2}}+r_{2}(4r_{1}+r_{2})\textcolor{blue}{c_{3}}-3r_{2}\textcolor{blue}{c_{4}}+49/80r_{2}^5\\
			 &+713/240r_{1}r_{2}^4+397/60r_{1}^2r_{2}^3+1283/240r_{1}^3r_{2}^2+481/240r_{1}^4r_{2},\\
			r_{5, 4} = & (292r_{1}^4+119r_{1}^3r_{2}-319r_{1}^2r_{2}^2-511r_{1}r_{2}^3-161r_{2}^4)/240\textcolor{blue}{c_{1}}+(r_{1}^2r_{2}/4-3/2r_{1}^3+3/2r_{1}r_{2}^2+3/4r_{2}^3)\textcolor{blue}{c_{2}}
			 +(2r_{1}^2-2r_{1}r_{2}\\&-r_{2}^2)\textcolor{blue}{c_{3}}+3/2(r_{2}-r_{1})\textcolor{blue}{c_{4}}+233/240r_{1}^5+671/480r_{1}^4r_{2}-77/480r_{1}^3r_{2}^2-469/160r_{1}^2r_{2}^3-311/160r_{1}r_{2}^4-31/60r_{2}^5,\\
			r_{5, 5} = & (292r_{1}^4+319r_{1}^3r_{2}-139r_{1}^2r_{2}^2-311r_{1}r_{2}^3-161r_{2}^4)/240\textcolor{blue}{c_{1}}+(r_{1}r_{2}^2-3/2r_{1}^3-r_{1}^2r_{2}/4+3/4r_{2}^3)\textcolor{blue}{c_{2}}+(r_{2}+2r_{1})(r_{1}\\
			 &-r_{2})\textcolor{blue}{c_{3}}+3/2(r_{2}-r_{1})\textcolor{blue}{c_{4}}+233/240r_{1}^5+301/160r_{1}^4r_{2}+91/96r_{1}^3r_{2}^2-175/96r_{1}^2r_{2}^3-701/480r_{1}r_{2}^4-31/60r_{2}^5.
		\end{split}
		\]}}
		\be\label{particular:6}
		\begin{split}
			& c_{0,0,6}=-8c_{6}, \quad c_{-1,0,6}=c_{6}+r_{6,1},\quad
			c_{0,-1,6}=c_{6}+r_{6,2},\quad
			c_{-1,-1,6}=c_{6}+r_{6,3}, \\
			& c_{-1,1,6}=c_{0,1,6}=c_{1,-1,6}=c_{1,0,6}=c_{1,1,6}=c_{6},\quad \textcolor{red}{c_{6}}= \min \{ -r_{6,1}, -r_{6,2}, -r_{6,3}, 0\},\  \text{ all } c_{k,\ell,7}=0,
		\end{split}
		\ee
where
{\tiny{
		\[
		\begin{split}
			r_{6, 1} = & (85r_{2}^5+166r_{2}^4r_{1}+642r_{2}^3r_{1}^2+616r_{2}^2r_{1}^3+377r_{2}r_{1}^4)/240\textcolor{blue}{c_{1}}-(r_{2}^4+3r_{1}r_{2}^3+7r_{1}^2r_{2}^2+7r_{1}^3r_{2})/4\textcolor{blue}{c_{2}}+r_{2}/2(5r_{1}^2+r_{1}r_{2}+r_{2}^2)\textcolor{blue}{c_{3}}\\
			 &-3/2r_{1}r_{2}\textcolor{blue}{c_{4}}+3r_{2}\textcolor{blue}{c_{5}}+101/480r_{2}^6+199/240r_{1}r_{2}^5+19/8r_{1}^2r_{2}^4+403/96r_{1}^3r_{2}^3+1519/480r_{1}^4r_{2}^2+189/160r_{1}^5r_{2}, \\
			r_{6, 2} = & (377r_{1}^5+616r_{2}r_{1}^4+642r_{2}^2r_{1}^3+166r_{2}^3r_{1}^2+85r_{2}^4r_{1})/240\textcolor{blue}{c_{1}}-(7r_{1}^4+7r_{1}^3r_{2}+3r_{1}^2r_{2}^2+r_{1}r_{2}^3)/4\textcolor{blue}{c_{2}}+r_{1}/2(5r_{1}^2+r_{1}r_{2}+r_{2}^2)\textcolor{blue}{c_{3}}\\
			& -3/2r_{1}^2\textcolor{blue}{c_{4}}+3r_{1}\textcolor{blue}{c_{5}}+189/160r_{1}^6+1519/480r_{1}^5r_{2}+403/96r_{1}^4r_{2}^2+19/8r_{1}^3r_{2}^3+199/240r_{1}^2r_{2}^4+101/480r_{1}r_{2}^5, \\
			r_{6, 3} = & (-377r_{1}^5-993r_{2}r_{1}^4-1258r_{2}^2r_{1}^3-808r_{2}^3r_{1}^2-251r_{2}^4r_{1}-85r_{2}^5)/240\textcolor{blue}{c_{1}}+(7r_{1}^4+14r_{1}^3r_{2}+10r_{1}^2r_{2}^2+4r_{1}r_{2}^3+r_{2}^4)/4\textcolor{blue}{c_{2}}\\
			 &-(5/2r_{1}^3+3r_{1}^2r_{2}+r_{1}r_{2}^2+r_{2}^3/2)\textcolor{blue}{c_{3}}+3/2r_{1}(r_{1}+r_{2})\textcolor{blue}{c_{4}}-3(r_{1}+r_{2})\textcolor{blue}{c_{5}}-189/160r_{1}^6-1043/240r_{1}^5r_{2}-589/80r_{1}^4r_{2}^2\\
			 &-631/96r_{1}^3r_{2}^3-769/240r_{1}^2r_{2}^4-499/480r_{1}r_{2}^5-101/480r_{2}^6.
		\end{split}
		\]}}
If $a$ is a positive constant, then all the above parameters $r$'s vanish and all the stencil coefficients $C_{k,\ell}$ are constants given in \eqref{particular:0}.

%
%
%
%
%
%
\begin{proof}[\textbf{Proof of \cref{thm:tranmiss:interface}}]
Replacing $x$ and $y$ by  $x-x_i^*$ and $y-y_j^*$ respectively in \eqref{u:approx:ir:key:2}, we have
\[
\begin{split}
	u_{\pm}(x,y)=\sum_{(m,n)\in \ind_{M}^{1}}
	u_\pm^{(m,n)} G^{\pm}_{M,m,n}(x-x_i^*,y-y_j^*) +\sum_{(m,n)\in \ind_{M-2}}
	f_\pm ^{(m,n)} H^{\pm}_{M,m,n}(x-x_i^*,y-y_j^*)+\bo(h^{M+1}),
\end{split}
\]
for $x\in (x_i^*-2h,x_i^*+2h)$ and  $y\in (y_j^*-2h,y_j^*+2h)$. Since $(x(t),y(t))=(r(t),s(t))$ is the parametric equation of $\Gamma$ and $(x_i^*,y_j^*)=(r(t_k^*),s(t_k^*))$, we have
\be\label{transs:flux}
\begin{split}
	&a_{\pm}(r(t), s(t)) \nabla u_\pm (r(t), s(t))\cdot(s'(t),-r'(t))\\
	&=
	\sum_{p=0}^{M-1}
	\Big(
	\sum_{(m,n)\in \ind_{M}^1}
	u_\pm^{(m,n)} \tilde{g}^{\pm}_{m,n,p}
	+\sum_{(m,n)\in \ind_{M-2}}
	f_\pm ^{(m,n)} \tilde{h}^{\pm}_{m,n,p}
	\Big) (t-t_k^*)^p+\bo((t-t_k^*)^{M}),
\end{split}
\ee
as $t \to t_k^*$,
where
\be\label{gmnhmn2}
\begin{split}
	&	\tilde{g}^{\pm}_{m,n,p}:=\frac{d^p( \widetilde{G}^{\pm}_{M,m,n}(r(t)-x_i^*,s(t)-y_j^*) \cdot(s'(t),-r'(t)))}{p!dt^p}\Big|_{t=t_k^*},\\
	 & \tilde{h}^{\pm}_{m,n,p}:=\frac{d^p( \widetilde{H}^{\pm}_{M,m,n}(r(t)-x_i^*,s(t)-y_j^*) \cdot(s'(t),-r'(t)))}{p!dt^p}\Big|_{t=t_k^*},
\end{split}
\ee
{\small{
		\begin{align*}
			& \widetilde{G}^{\pm}_{M,m,n}(x,y) = \nabla G^{\pm}_{M,m,n}(x,y)\sum_{(m,n)\in \ind_{M-1}} a_{\pm}^{(m,n)} \tfrac{x^m y^{n}}{m!n!},\ \ \widetilde{H}^{\pm}_{M,m,n}(x,y) = \nabla H^{\pm}_{M,m,n}(x,y)\sum_{(m,n)\in \ind_{M-1}} a_{\pm}^{(m,n)} \tfrac{x^m y^{n} }{m!n!}.
\end{align*}}}
By \eqref{gp:gGammap},
\be\label{transs:jump2}
g_\Gamma(t)\sqrt{(r'(t))^2+(s'(t))^2}
=\sum_{p=0}^{M-1} g_\Gamma^{(p)} (t-t_k^*)^p+\bo((t-t_k^*)^{M}),\quad \text{as} \quad t \to t_k^*.
\ee
So
$\left[a\nabla  u \cdot \nv \right]=\gn$ on $\Gamma$ with \eqref{transs:flux}--\eqref{transs:jump2} implies
\be \label{interface:flux:0}
\sum_{(m,n)\in \ind_{M}^1}
(u_+^{(m,n)}\tilde{g}^+_{m,n,p}-u_-^{(m,n)} \tilde{g}^-_{m,n,p})=g_{\Gamma}^{(p)}+
\sum_{(m,n)\in \ind_{M-2}} (f_-^{(m,n)}\tilde{h}^{-}_{m,n,p}-f_+^{(m,n)}
\tilde{h}^{+}_{m,n,p}),
\ee
where $p=0,\ldots,M-1$. By the definition of $\tilde{g}^{\pm}_{m,n,p}$ in \eqref{gmnhmn2}, \eqref{Gmn}, and $(r(t_k^*)-x_i^*,s(t_k^*)-y_j^*)=(0,0)$,
we have  $\tilde{g}^{\pm}_{m,n,p-1}=0$ for $m+n>p$. So \eqref{interface:flux:0} with $p$ being replaced by $p-1$ yields
\be \label{interface:flux:1}
\begin{split}
	&u_-^{(0,p)}\tilde{g}^{-}_{0,p,p-1}
	+u_-^{(1,p-1)} \tilde{g}^{-}_{1,p-1,p-1}
	=u_+^{(0,p)}\tilde{g}^{+}_{0,p,p-1}
	+u_+^{(1,p-1)} \tilde{g}^{+}_{1,p-1,p-1}+\sum_{(m,n)\in \ind_{p-1}^1}
	u_{+}^{(m,n)}\tilde{g}^{+}_{m,n,p-1}\\
	&\qquad \qquad-\sum_{(m,n)\in \ind_{p-1}^1}
u_{-}^{(m,n)}\tilde{g}^{-}_{m,n,p-1}-\sum_{(m,n)\in \ind_{M-2}} ( f_-^{(m,n)}\tilde{h}^{-}_{m,n,p-1}-f_+^{(m,n)}
	\tilde{h}^{+}_{m,n,p-1}) -g_{\Gamma}^{(p-1)},
\end{split}
\ee
where $p=1,\ldots,M$. By \eqref{gmnhmn2}, \eqref{Gmn} and $(r(t_k^*)-x_i^*,s(t_k^*)-y_j^*)=(0,0)$, we also have
%
\be\label{gmnhmntide}
\tilde{g}^{\pm}_{m,n,p-1}=a_{\pm}^{(0,0)}\frac{d^{p-1}( \nabla G_{m,n}(r(t)-x_i^*,s(t)-y_j^*) \cdot(s'(t),-r'(t)))}{ (p-1)! dt^{p-1}}\Big|_{t=t_k^*},
\ee
for $(m,n)\in\{(0,p),(1,p-1)\}$.
Similarly, $\left[u\right]=\gd$ on $\Gamma$ implies $u_{-}^{(0,0)}
=u_{+}^{(0,0)}-g^{(0)}$  and
\be \label{interface:u:1}
\begin{split}
	 &u_{-}^{(0,p)}g^{-}_{0,p,p}+u_{-}^{(1,p-1)}g^{-}_{1,p-1,p}
	 =u_{+}^{(0,p)}g^{+}_{0,p,p}+u_{+}^{(1,p-1)}g^{+}_{1,p-1,p} +\sum_{(m,n)\in \ind_{p-1}^1} u_{+}^{(m,n)}g^{+}_{m,n,p}\\
	&\qquad \qquad \qquad -\sum_{(m,n)\in \ind_{p-1}^1} u_{-}^{(m,n)}g^{-}_{m,n,p}-
	\sum_{(m,n)\in \ind_{M-2}} (f_-^{(m,n)}h^{-}_{m,n,p}-f_+^{(m,n)}
	h^{+}_{m,n,p})-g^{(p)},
\end{split}
\ee
for $p=1,\ldots,M$, where
{\small{
\be\label{gmnhmn0}
g^{\pm}_{m,n,p}:= \frac{d^p(G^{\pm}_{M,m,n}(r(t)-x_i^*,s(t)-y_j^*))}{p!dt^p}\Big|_{t=t_k^*},\ \ h^{\pm}_{m,n,p}:=\frac{d^p(H^{\pm}_{M,m,n}(r(t)-x_i^*,s(t)-y_j^*))}{p!dt^p}\Big|_{t=t_k^*},
\ee}}
%
\be\label{gmnhmn1}
g^{\pm}_{m,n,p}= \frac{d^p(G_{m,n}(r(t)-x_i^*,s(t)-y_j^*))}{p!dt^p}\Big|_{t=t_k^*}, \quad (m,n)\in\{(0,p),(1,p-1)\}.
\ee
By $(r'(t_k^*))^2+(s'(t_k^*))^2>0$, $a_{\pm}^{(0,0)}> 0$, \eqref{gmnhmntide} and \eqref{gmnhmn1},   \cite[(A.15)-(A.28)]{FHM21a}  implies
\be\label{GThan0}
g^{-}_{0,p,p}\tilde{g}^{-}_{1,p-1,p-1}- g^{-}_{1,p-1,p}
\tilde{g}^{-}_{0,p,p-1}=a_{-}^{(0,0)}p((r'(t_k^*))^2+(s'(t_k^*))^2)^p/(p!)^2>0, \quad  p=1,\ldots,M.
\ee
%

Now, \eqref{tranmiss:cond}  can be obtained by solving  \eqref{interface:flux:1} and \eqref{interface:u:1} recursively in the order $p=1,\dots,M$,
and  $u_{-}^{(0,0)}
=u_{+}^{(0,0)}-g^{(0)}$ (\eqref{GThan0} implies the uniqueness and existence of \eqref{tranmiss:cond}). Since all $\{u_{\pm}^{(m,n)}:(m,n)\in \ind_{M}^1\setminus \ind_{p}^1\}$ vanish in \eqref{interface:flux:1} and \eqref{interface:u:1}, we have $T^{u_{+}}_{m',n',m,n}=0$ for $m+n>m'+n'$ in
\eqref{tranmiss:cond}.

The uniqueness of \eqref{tranmiss:cond} and $u_{-}^{(0,0)}
=u_{+}^{(0,0)}-g^{(0)}$ lead to $T^{u_+}_{0,0,0,0}=1$. \eqref{GM100} implies $G^{\pm}_{M,0,0}(x,y)=1$ for all $M\in \NN$. Then, by \eqref{gmnhmn2} and \eqref{gmnhmn0}, 	 $\tilde{g}^{\pm}_{0,0,p-1}=g^{\pm}_{0,0,p}=0$ in \eqref{interface:flux:1} and \eqref{interface:u:1} for $1\le p$. So $T^{u_+}_{m',n',0,0}=0$ for  $(m',n')\ne (0,0)$ in \eqref{tranmiss:cond}. Thus we proved \eqref{T0000}.

By \eqref{Hmn}, \eqref{gmnhmntide}, \eqref{gmnhmn1} and $(r(t_k^*)-x_i^*,s(t_k^*)-y_j^*)=(0,0)$, we have that $\tilde{g}^{\pm}_{0,p,p-1},\tilde{g}^{\pm}_{1,p-1,p-1}$ in \eqref{interface:flux:1}  only depend on $a_{\pm}^{(0,0)}$, $(r'(t_k^*), s'(t_k^*))$, and $g^{\pm}_{0,p,p},g^{\pm}_{1,p-1,p}$ in \eqref{interface:u:1} only depend on $(r'(t_k^*), s'(t_k^*))$. So each $T^{u_+}_{m',n',m,n}$ with $m+n=m'+n'$ in \eqref{tranmiss:cond} only depends on $a_{\pm}^{(0,0)}$ and $(r'(t_k^*), s'(t_k^*))$.

\end{proof}
\begin{proof}[\textbf{Proof of \cref{thm:13point:scheme}}]
 In the following statement, \eqref{Lhu:irregular}--\eqref{Imn:irregular} are used to derive \eqref{13point:scheme}--\eqref{right:irregular}, \eqref{U00:coeff}--\eqref{irregular:simplify2} are used to  derive \eqref{irregular:simplify3} and \eqref{Matrix:Form:irregular}.\\	
	For $(x_i^*,y_j^*)\in \Gamma$ and $(x_i,y_j)=(x_i^*+v_0h,y_j^*+w_0h)$, we define
	\[
	h^{-1}\mathcal{L}_h u:=h^{-1} \Big(\sum_{k=-1}^1 \sum_{\ell=-1}^1
	C_{k,\ell}  u(x_i+kh,y_j+\ell h)+ \sum_{k=-2,2}
	C_{k,0}  u(x_i+kh,y_j)+  \sum_{\ell=-2,2}
	C_{0,\ell}  u(x_i,y_j+\ell h)\Big),
	\]
	where
	$C_{k,\ell}:=\sum_{p=0}^{M} c_{k,\ell,p}h^p$ with $c_{k,\ell,p} \in \R$.
	According to \eqref{u:approx:ir:key:2},
	\be\label{Lhu:irregular}
	\begin{split}
		&h^{-1}\mathcal{L}_h u
		=h^{-1}\sum_{(k,\ell)\in \tilde{d}_{i,j}^+}
		C_{k,\ell}    u(x_i^*+v_1h,y_j^*+w_1 h)  +h^{-1} \sum_{(k,\ell)\in \tilde{d}_{i,j}^-}
		C_{k,\ell}  u(x_i^*+v_1h,y_j^*+w_1 h)\\
		&=h^{-1} \sum_{(m,n)\in \ind_{M}^1}
		\Big(u_+^{(m,n)}  I^{+}_{m,n}+	 u_-^{(m,n)}  I^{-}_{m,n}\Big)+ \sum_{(m,n)\in \ind_{M-2}}
		\Big(	f_+ ^{(m,n)} J^{+,0}_{m,n}+
		f_- ^{(m,n)} J^{-,0}_{m,n}\Big)+\bo(h^{M}),\\
	\end{split}
	\ee
	with
	\be	\label{IJmn:irregular}
	\begin{split}
		& I^\pm_{m,n}:=\sum_{(k,\ell)\in \tilde{d}_{i,j}^\pm}
		C_{k,\ell} G^{\pm}_{M,m,n}(v_1h,w_1 h), \quad J^{\pm,0}_{m,n}:=h^{-1}\sum_{(k,\ell)\in \tilde{d}_{i,j}^\pm} C_{k,\ell}  H^{\pm}_{M,m,n}(v_1h, w_1 h),\\
		& v_1=v_0+k, \quad w_1=w_0+\ell,
		\quad \tilde{d}_{i,j}^\pm =d_{i,j}^\pm \cup e_{i,j}^\pm.
	\end{split}
	\ee
	\eqref{tranmiss:cond} implies
	\begin{align*}
		h^{-1} \sum_{(m',n')\in \ind_{M}^1}
		u_-^{(m',n')}   I^-_{m',n'}
		=&h^{-1} \sum_{(m,n)\in \ind_{M}^1} u_+^{(m,n)}
		J^{u_+,T}_{m,n}+
		\sum_{(m,n)\in \ind_{M-2}}
		f_+^{(m,n)} J^{+,T}_{m,n}\\
		&+\sum_{(m,n)\in \ind_{M-2}} f_-^{(m,n)} J^{-,T}_{m,n}+\sum_{p=0}^{M} g^{(p)}   J^{g}_p+\sum_{p=0}^{M-1} g^{(p)}_{\Gamma}   J^{g_\Gamma}_p,
	\end{align*}
	with
	\be
	\label{IJmn:2:irregular}
	\begin{split}
		&J^{u_+,T}_{m,n}:=
		\sum_{  \substack{ (m',n')\in \ind_{M}^{1} \\  m'+n' \ge m+n}} I^-_{m',n'} T^{u_+}_{m',n',m,n}, \quad J^{\pm,T}_{m,n}:=h^{-1}
		\sum_{(m',n')\in \ind_{M}^1} I^-_{m',n'} T^\pm_{m',n',m,n},\\
		&J^{g}_{p}:=h^{-1}
		\sum_{(m',n')\in \ind_{M}^1} I^-_{m',n'} T^{g}_{m',n',p},\quad
		J^{g_\Gamma}_{p}:=h^{-1}
		\sum_{(m',n')\in \ind_{M}^1} I^-_{m',n'} T^{g_\Gamma}_{m',n',p}.
	\end{split}
	\ee
	We define that
	\be \label{Lhuh:irregular}
	\begin{split}
		h^{-1}	\mathcal{L}_h u_h & :=h^{-1}	 \Big( \sum_{k=-1}^{1}\sum_{\ell =-1}^{1} C_{k,\ell}(u_{h})_{i+k,j+\ell}  +  \sum_{k=-2,2} C_{k,0}(u_{h})_{i+k,j}+\sum_{\ell=-2,2} C_{0,\ell }(u_{h})_{i,j+\ell} \Big)\\
		&=\sum_{(m,n)\in \ind_{M-2}} \left(f_-^{(m,n)}J^-_{m,n}
		+f_+^{(m,n)}J^+_{m,n}\right)+ \sum_{p=0}^{M} g^{(p)} J^{g}_p+\sum_{p=0}^{M-1} g^{(p)}_{\Gamma} J^{g_\Gamma}_p,
	\end{split}
	\ee
	where
	$J^{\pm}_{m,n}:= J_{m,n}^{\pm,0}+J^{\pm,T}_{m,n}$ with
 $J_{m,n}^{\pm,0}$ and $J^{\pm,T}_{m,n}$  in \eqref{IJmn:irregular} and \eqref{IJmn:2:irregular} respectively.
	We conclude from \eqref{Lhu:irregular}--\eqref{Lhuh:irregular} that
	$
	h^{-1} \mathcal{L}_h (u-u_h)=\bo(h^M),
	$
	if
	\be \label{Imn:irregular}
	I^+_{m,n}+J^{u_+,T}_{m,n}=\bo(h^{M+1}), \quad \mbox{ for all }\; (m,n)\in \ind_{M}^1,
	\ee
	where $I^{\pm}_{m,n}$ and $J^{u_+,T}_{m,n}$ are defined in \eqref{IJmn:irregular} and \eqref{IJmn:2:irregular} respectively.
By the symbolic calculation,
	\eqref{Imn:irregular} has a nontrivial solution $\{C_{k,\ell}\}$ if and only if $M\le 5$.
	 \eqref{IJmn:irregular}--\eqref{Imn:irregular} with $M=5$ yield \eqref{13point:scheme}--\eqref{right:irregular} in \cref{thm:13point:scheme}.	

	In order to derive \eqref{irregular:simplify3} and \eqref{Matrix:Form:irregular} to solve \eqref{Imn:irregular} efficiently, let us consider the following \eqref{U00:coeff}--\eqref{irregular:simplify2}:
	\eqref{IJmn:irregular}, \eqref{IJmn:2:irregular} and $m=n=0$ in \eqref{Imn:irregular} result in
	\be \label{U00:coeff}
	\sum_{(k,\ell)\in \tilde{d}_{i,j}^+}
	C_{k,\ell} G^{+}_{M,0,0}(v_1h,w_1 h)+	 \sum_{ \substack{ (m',n')\in \ind_{M}^{1} \\  m'+n' \ge 0}} I^-_{m',n'} T^{u_+}_{m',n',0,0}=\bo(h^{M+1}).
	\ee
	By \eqref{T0000},
	\eqref{U00:coeff} becomes
	\[\sum_{(k,\ell)\in \tilde{d}_{i,j}^+}
	C_{k,\ell} G^{+}_{M,0,0}(v_1h,w_1 h)+	 I^-_{0,0}=\bo(h^{M+1}).\]
	Using $I^-_{m,n}$ in \eqref{IJmn:irregular}, we obtain
	\be \label{U00:coeff:2}
	\sum_{(k,\ell)\in \tilde{d}_{i,j}^+}
	C_{k,\ell} G^{+}_{M,0,0}(v_1h,w_1 h)+	 \sum_{(k,\ell)\in \tilde{d}_{i,j}^-}
	C_{k,\ell} G^{-}_{M,0,0}(v_1h,w_1 h)=\bo(h^{M+1}).
	\ee		
By \eqref{GM100},
	$G^{\pm}_{M,0,0}(x,y)=1$.
	$C_{k,\ell}:=\sum_{p=0}^{M} c_{k,\ell,p}h^p$ and \eqref{U00:coeff:2} lead to
	\be\label{ckl:sum:irregular}
	\sum_{(k,\ell)\in \tilde{d}_{i,j}^+}
	c_{k,\ell,p} +	 \sum_{(k,\ell)\in \tilde{d}_{i,j}^-}
	c_{k,\ell,p} =0,	\quad \mbox{for} \quad p=0,\dots, M,
	\ee
	i.e.,
	\be\label{ckl:sum:2:irregular}	
	\sum_{k=-2,2} c_{k,0,p}+\sum_{\ell=-2,2} c_{0,\ell,p}+\sum_{k=-1}^{1} \sum_{\ell=-1}^{1}
	c_{k,\ell,p}=0, \quad \mbox{for} \quad p=0,\dots, M.
	\ee	
	By  \eqref{IJmn:irregular} and \eqref{IJmn:2:irregular}, \eqref{Imn:irregular} becomes
	\[
	\begin{split}
		&\sum_{(k,\ell)\in \tilde{d}_{i,j}^+}
		C_{k,\ell} G^{+}_{M,m,n}(v_1h,w_1h)+	 \sum_{ \substack{ (m',n')\in \ind_{M}^1 \\  m'+n' \ge m+n}}  T^{u_+}_{m',n',m,n} \sum_{(k,\ell)\in \tilde{d}_{i,j}^-}
		C_{k,\ell} G^{-}_{M,m',n'}(v_1h,w_1h) =\bo(h^{M+1}),
	\end{split}
	\]
	for all $(m,n)\in \ind_{M}^1$.
By $C_{k,\ell}:=\sum_{p=0}^{M} c_{k,\ell,p}h^p$, \eqref{Imn:irregular} is equivalent to
	{\small{
	\be \label{irregular:plug}
	\begin{split}
		&\sum_{(k,\ell)\in \tilde{d}_{i,j}^+}
		\sum_{p=0}^{M} c_{k,\ell,p}h^p G^{+}_{M,m,n}(v_1h,w_1h)	 + \sum_{(k,\ell)\in \tilde{d}_{i,j}^-} \sum_{p=0}^{M} c_{k,\ell,p}h^p\sum_{ \substack{ (m',n')\in \ind_{M}^1 \\  m'+n' \ge m+n}}
		T^{u_+}_{m',n',m,n} G^{-}_{M,m',n'}(v_1h,w_1h)=\bo(h^{M+1}),
	\end{split}
	\ee	}}
	for all $(m,n)\in \ind_{M}^1$.	
	By \eqref{Gmn},	
	\eqref{irregular:plug} is equivalent to
	\be \label{irregular:plug:all}
	\begin{split}
		&\sum_{(k,\ell)\in \tilde{d}_{i,j}^+}
		\sum_{p=0}^{M} c_{k,\ell,p}h^p G_{m,n}(v_1h,w_1 h)+\sum_{(k,\ell)\in \tilde{d}_{i,j}^+}
		\sum_{p=0}^{M} c_{k,\ell,p}h^p G^{2,+}_{m,n}(v_1h,w_1h)\\
		&+	\sum_{(k,\ell)\in \tilde{d}_{i,j}^-} \sum_{p=0}^{M} c_{k,\ell,p}h^p \sum_{ \substack{ (m',n')\in \ind_{M}^1 \\  m'+n' \ge m+n}}
		T^{u_+}_{m',n',m,n} G_{m',n'}(v_1h,w_1 h) \\
		&+	\sum_{(k,\ell)\in \tilde{d}_{i,j}^-} \sum_{p=0}^{M} c_{k,\ell,p}h^p \sum_{ \substack{ (m',n')\in \ind_{M}^1 \\  m'+n' \ge m+n}}
		T^{u_+}_{m',n',m,n} G^{2,-}_{m',n'}(v_1h,w_1h)=\bo(h^{M+1}),
	\end{split}
	\ee	
	for all $(m,n)\in \ind_{M}^1$,
	where
	$G^{2,\pm}_{m',n'}(x,y)=\sum_{(p,q)\in \ind_{M}^2 \setminus \ind_{m'+n'}^2 }a^{u_{\pm}}_{p,q,m',n'}  \frac{x^{p} y^{q}}{p!q!}$.
	For the sake of brevity, we define that $\textsf{degree}(f)$ is the degree of $f$ of $h$.
	By $\textsf{degree}(G_{m,n}(v_1h,w_1h))= m+n$, and $\textsf{degree}(G^{2,\pm}_{m,n}(v_1h,w_1h))> m+n$,
 non-zero terms of $h^{m+n+d}$ in \eqref{irregular:plug:all} with $0\le d \le M-m-n$ and $(m,n)\in \ind_{M}^1$ lead to
	\be\label{irregular:simplify1}
	\begin{split}
		& \sum_{(k,\ell)\in \tilde{d}_{i,j}^+}c_{k,\ell,d}h^d G_{m,n}(v_1h,w_1h)+  \sum_{(k,\ell)\in \tilde{d}_{i,j}^+} 	 \sum_{s=0}^{M} c_{k,\ell,s}
		h^s  \sum_{ \substack{ (p,q)\in \ind_{M}^2 \setminus \ind_{m+n}^2 \\ p+q=m+n+d-s}} \frac{a^{u_{+}}_{p,q,m,n}}{p!q!}  v_1^{p}w_1^{q} h^{p+q}  \\
		&+	\sum_{(k,\ell)\in \tilde{d}_{i,j}^-} 	 \sum_{s=0}^{M} c_{k,\ell,s}h^s \sum_{\substack{(m',n')\in \ind_{M}^1 \\ m'+n'\ge m+n \\ m'+n'=m+n+d-s} }
		T^{u_+}_{m',n',m,n} G_{m',n'}(v_1h,w_1h) \\
		&+	\sum_{(k,\ell)\in \tilde{d}_{i,j}^-} 	 \sum_{s=0}^{M} c_{k,\ell,s}	h^s \sum_{\substack{(m',n')\in \ind_{M}^1 \\ m'+n'\ge m+n } }
		T^{u_+}_{m',n',m,n}  \sum_{ \substack{ (p,q)\in \ind_{M}^2 \setminus \ind_{m'+n'}^2 \\ p+q=m+n+d-s}} \frac{a^{u_{-}}_{p,q,m',n'}}{p!q!} v_1^{p}w_1^{q} h^{p+q} =\bo(h^{M+1}).
	\end{split}
	\ee	
	Note that $(p,q)\in \ind_{M}^2 \setminus \ind_{m+n}^2$ implies $p+q>m+n$.
	In the first row of \eqref{irregular:simplify1},  $m+n+d-s =p+q> m+n\Rightarrow$ $m+n+d-s > m+n\Rightarrow$  $s\le d-1$.
	In the second row of \eqref{irregular:simplify1},  $m+n+d-s \ge m+n\Rightarrow$   $s\le d$.
	In the third row of \eqref{irregular:simplify1},  $m+n+d-s =p+q>m'+n' \ge m+n\Rightarrow$ $m+n+d-s> m+n\Rightarrow$ $s\le d-1$.
	So \eqref{irregular:simplify1} is equivalent to
	\be\label{irregular:simplify2}
	\begin{split}
		& \sum_{(k,\ell)\in \tilde{d}_{i,j}^+}c_{k,\ell,d} G_{m,n}(v_1,w_1)+  \sum_{(k,\ell)\in \tilde{d}_{i,j}^+} 	 \sum_{s=0}^{d-1} c_{k,\ell,s}
		\sum_{ \substack{ (p,q)\in \ind_{M}^2 \setminus \ind_{m+n}^2 \\ p+q=m+n+d-s}} \frac{a^{u_{+}}_{p,q,m,n}}{p!q!}  v_1^{p}w_1^{q}  \\
		&+	\sum_{(k,\ell)\in \tilde{d}_{i,j}^-} 	 \sum_{s=0}^{d} c_{k,\ell,s}	 \sum_{\substack{(m',n')\in \ind_{M}^1 \\ m'+n'=m+n+d-s} }
		T^{u_+}_{m',n',m,n} G_{m',n'}(v_1,w_1) \\
		&+	\sum_{(k,\ell)\in \tilde{d}_{i,j}^-} 	 \sum_{s=0}^{d-1} c_{k,\ell,s}	 \sum_{\substack{(m',n')\in \ind_{M}^1 \\ m'+n'\ge m+n } }
		T^{u_+}_{m',n',m,n}  \sum_{ \substack{ (p,q)\in \ind_{M}^2 \setminus \ind_{m'+n'}^2 \\ p+q=m+n+d-s}} \frac{a^{u_{-}}_{p,q,m',n'}}{p!q!}  v_1^{p}w_1^{q} =0,
	\end{split}
	\ee	
	where $0\le d \le M-m-n$ and $(m,n)\in \ind_{M}^1$. Note that the summation $\sum_{s=0}^{d-1}$ in \eqref{irregular:simplify2} is empty for $d=0$. By a direct calculation of \eqref{irregular:simplify2} with $M=5$, we can obtain \eqref{irregular:simplify3} and \eqref{Matrix:Form:irregular}.
\end{proof}
%
%
%
%
\subsection{6-point and 4-point compact stencils at boundary points}\label{Appendix:Boundary}
In this subsection, we  discuss how to find the 6-point FDM with the sixth-order consistency centered at $(x_i,y_j) \in \Gamma_1$ in \cref{thm:Robin:Gamma1}, where $(x_i,y_j)$ is not the corner point (see \cref{fig:model:problem,Ckl:6:points:Gamma1}).
Then we  discuss how to find the 4-point FDM with the sixth-order consistency centered at the corner point $(x_i,y_j)=(l_1,l_3)$ in \cref{thm:Corner:1} (see \cref{fig:model:problem,4:points:scheme:Corner1,Ckl:4:points:Corner1}).
In this subsection, we choose $(x_i^*,y_j^*)=(x_i,y_j)$, i.e., $v_0=w_0=0$ in \eqref{base:pt}
and use the following notations:
\be\label{g1ng3n}
\begin{split}	
	&\alpha^{(n)}:=\frac{d^{n} \alpha}{ dy^n }(y_j^*),
	\quad {g_1}^{(n)}:=\frac{d^{n} g_1}{ dy^n }(y_j^*), \quad \beta^{(m)}:=\frac{d^{m} \beta}{ dx^m }(x_i^*),
	\quad {g_3}^{(m)}:=\frac{d^{m} g_3}{ dx^m }(x_i^*),
\end{split}
\ee
which are their $n$th or $m$th  derivatives at the base point $(x_i^*,y_j^*)$.
For the sake of presentation,  we establish the  following auxiliary identities \eqref{u1n:Gamma1}--\eqref{Em:tilde:beta} for the proofs of \cref{thm:Robin:Gamma1,thm:Corner:1}.
Since $-u_x+\alpha u=g_1$ on $\Gamma_1$, choose $(x_i^*,y_j^*)=(x_i,y_j)\in\Gamma_1$, we have $u^{(1,0)}={\alpha}^{(0)}u^{(0,0)} - g_1^{(0)}$. Then
\be\label{u1n:Gamma1}
u^{(1,n)} = \sum_{i=0}^n  {n\choose i}  {\alpha}^{(n-i)}u^{(0,i)} - g_1^{(n)},\qquad \text{ for all } n = 0,\dots, M-1.
\ee
By \eqref{u:approx} with $M$ being replaced by $M-1$ and \eqref{u1n:Gamma1}, choose $(x_i^*,y_j^*)=(x_i,y_j)\in\Gamma_1$, we have (see \cref{Robin:1:umn}):
\begin{align*}
	u(x+x_i^*,y+y_j^*)
	& = \sum_{n=0}^{M}
	u^{(0,n)} G_{M,0,n}(x,y)+\sum_{n=0}^{M-1}
	\bigg( \sum_{i=0}^n  {n\choose i}  {\alpha}^{(n-i)}u^{(0,i)} -g_1^{(n)}  \bigg) G_{M,1,n}(x,y) \\
	& \qquad +\sum_{(m,n)\in \ind_{M-2}}
	f^{(m,n)} H_{M,m,n}(x,y)+\bo(h^{M+1})\\
	&=\sum_{n=0}^{M}
	u^{(0,n)} G_{M,0,n}(x,y) +\sum_{i=0}^{M-1}u^{(0,i)}\bigg(
	\sum_{n=i}^{M-1}  {n\choose i}  {\alpha}^{(n-i)} G_{M,1,n}(x,y) \bigg) \\
	& \qquad -\sum_{n=0}^{M-1}
	g_{1}^{(n)}  G_{M,1,n}(x,y) +\sum_{(m,n)\in \ind_{M-2}}
	f^{(m,n)} H_{M,m,n}(x,y) +\bo(h^{M+1}),
\end{align*}
for  $x,y\in (-2h,2h)$, i.e.,
\be\label{u0n:En:alpha}
\begin{split}
	u(x+x_i^*,y+y_j^*)	&= \sum_{n=0}^{M}u^{(0,n)}E_n(x,y)-\sum_{n=0}^{M-1}
	g_{1}^{(n)}  G_{M,1,n}(x,y)\\
	& \quad   +\sum_{(m,n)\in \ind_{M-2}}
	f^{(m,n)} H_{M,m,n}(x,y) +\bo(h^{M+1}), 	
\end{split}
\ee
where
\be\label{En:alpha}
E_n(x,y)=G_{M,0,n}(x,y)+(1-\delta_{n,M})
\sum_{i=n}^{M-1}  {i\choose n}  {\alpha}^{(i-n)} G_{M,1,i}(x,y),
\ee
and $\delta_{n,M}=1$ if $n=M$, and  $\delta_{n,M}=0$ if $n\ne M$.
\begin{figure}[h]
	\centering
	\hspace{12mm}	
	\begin{subfigure}[b]{0.3\textwidth}
		\begin{tikzpicture}[scale = 0.45]
			\node (A) at (0,7) {$u^{(0,0)}$};
			\node (A) at (0,6) {$u^{(0,1)}$};
			\node (A) at (0,5) {$u^{(0,2)}$};
			\node (A) at (0,4) {$u^{(0,3)}$};
			\node (A) at (0,3) {$u^{(0,4)}$};
			\node (A) at (0,2) {$u^{(0,5)}$};
			\node (A) at (0,1) {$u^{(0,6)}$};
			\node (A) at (2,7) {$u^{(1,0)}$};
			\node (A) at (2,6) {$u^{(1,1)}$};
			\node (A) at (2,5) {$u^{(1,2)}$};
			\node (A) at (2,4) {$u^{(1,3)}$};
			\node (A) at (2,3) {$u^{(1,4)}$};
			\node (A) at (2,2) {$u^{(1,5)}$};	
			\node (A) at (4,7) {$u^{(2,0)}$};
			\node (A) at (4,6) {$u^{(2,1)}$};
			\node (A) at (4,5) {$u^{(2,2)}$};
			\node (A) at (4,4) {$u^{(2,3)}$};
			\node (A) at (4,3) {$u^{(2,4)}$};
			\node (A) at (6,7) {$u^{(3,0)}$};
			\node (A) at (6,6) {$u^{(3,1)}$};
			\node (A) at (6,5) {$u^{(3,2)}$};
			\node (A) at (6,4) {$u^{(3,3)}$};
			\node (A) at (8,7) {$u^{(4,0)}$};
			\node (A) at (8,6) {$u^{(4,1)}$};
			\node (A) at (8,5) {$u^{(4,2)}$};
			\node (A) at (10,7) {$u^{(5,0)}$};
			\node (A) at (10,6) {$u^{(5,1)}$};
			\node (A) at (12,7) {$u^{(6,0)}$};
			\draw[ultra thick, blue][->] (7,4) to[left] (18,4);     	
		\end{tikzpicture}
	\end{subfigure}
	\begin{subfigure}[b]{0.3\textwidth}
		\hspace{3cm}
		\begin{tikzpicture}[scale = 0.45]
			\node (A) at (0,7) {$u^{(0,0)}$};
			\node (A) at (0,6) {$u^{(0,1)}$};
			\node (A) at (0,5) {$u^{(0,2)}$};
			\node (A) at (0,4) {$u^{(0,3)}$};
			\node (A) at (0,3) {$u^{(0,4)}$};
			\node (A) at (0,2) {$u^{(0,5)}$};
			\node (A) at (0,1) {$u^{(0,6)}$};
			\node (A) at (2,7) {$u^{(1,0)}$};
			\node (A) at (2,6) {$u^{(1,1)}$};
			\node (A) at (2,5) {$u^{(1,2)}$};
			\node (A) at (2,4) {$u^{(1,3)}$};
			\node (A) at (2,3) {$u^{(1,4)}$};
			\node (A) at (2,2) {$u^{(1,5)}$};
			\draw[ultra thick, blue][->] (3,4) to[left] (6,4);  	      	
		\end{tikzpicture}
	\end{subfigure}
	\begin{subfigure}[b]{0.3\textwidth}
		\hspace{0.6cm}
		\begin{tikzpicture}[scale = 0.45]
			\node (A) at (0,7) {$u^{(0,0)}$};
			\node (A) at (0,6) {$u^{(0,1)}$};
			\node (A) at (0,5) {$u^{(0,2)}$};
			\node (A) at (0,4) {$u^{(0,3)}$};
			\node (A) at (0,3) {$u^{(0,4)}$};
			\node (A) at (0,2) {$u^{(0,5)}$};
			\node (A) at (0,1) {$u^{(0,6)}$};      	
		\end{tikzpicture}
	\end{subfigure}
	\caption
	{The illustration for  \eqref{u0n:En:alpha} with $M=6$.}
	\label{Robin:1:umn}
\end{figure}
Choose $(x_i^*,y_j^*)=(x_i,y_j)\in\Gamma_3$, $-u_y+\beta u=g_3$ on $\Gamma_3$ implies $u^{(0,1)}={\beta}^{(0)}u^{(0,0)} - g_3^{(0)}$, and
\be\label{um1:Gamma3}
u^{(m,1)} = \sum_{i=0}^m  {m\choose i}  {\beta}^{(m-i)}u^{(i,0)} - g_3^{(m)},\qquad \text{ for all } m = 0,\dots, M-1.
\ee
Similarly to \eqref{u:original}--\eqref{Hmn}, we  have (see \cref{Corner:umn}):
\be
\label{u:approx:tilde}
u(x+x_i^*,y+y_j^*)  = \sum_{(n,m)\in \ind_{M+1}^{1}} u^{(m,n)} \tilde{G}_{M+1,m,n}(x,y) + \sum_{(m,n)\in \ind_{M-1}} f^{(m,n)} \tilde{H}_{M+1,m,n}(x,y) + \mathcal{O}(h^{M+2}),
\ee
for $x,y\in (-2h,2h)$ and $(x_i^*,y_j^*)\in \overline{\Omega}$ with
\be\label{Gmn:tilde}
\tilde{G}_{M+1,m,n}(x,y):=\sum_{(q,p)\in \ind_{M+1} }\tilde{a}^{u}_{p,q,m,n} \frac{x^{p} y^{q}}{p!q!}, \qquad 	 \tilde{H}_{M+1,m,n}(x,y):=
\sum_{(q,p)\in \ind_{M+1}^{2}  }\tilde{a}^{f}_{p,q,m,n}  \frac{x^{p} y^{q}}{p!q!},	
\ee
where
  $\tilde{a}^u_{p, q, m, n}$ and $\tilde{a}^f_{p, q, m, n}$  are uniquely determined by $\{a^{(i,j)}: (i,j) \in \ind_{M}\}$, and can be obtained similarly as \eqref{uderivx2}--\eqref{recursive:au}.
Choose $(x_i^*,y_j^*)=(x_i,y_j)\in\Gamma_3$,  \eqref{u:approx:tilde} with $M$ being replaced by $M-1$ and \eqref{um1:Gamma3} imply (see \cref{Corner:umn}):
\be\label{um0:Em:beta}
\begin{split}
	u(x+x_i^*,y+y_j^*)	 &=\sum_{m=0}^{M}u^{(m,0)} \tilde{E}_m(x,y)  -\sum_{m=0}^{M-1}
	g_{3}^{(m)}  \tilde{G}_{M,m,1}(x,y)\\
	& \quad  +\sum_{(m,n)\in \ind_{M-2}}
	f^{(m,n)} \tilde{H}_{M,m,n}(x,y) +\bo(h^{M+1}),
\end{split}
\ee
where
\be\label{Em:tilde:beta}
\tilde{E}_m(x,y)=\tilde{G}_{M,m,0}(x,y)+(1-\delta_{m,M})
\sum_{i=m}^{M-1}  {i\choose m}  {\beta}^{(i-m)} \tilde{G}_{M,i,1}(x,y).
\ee
\begin{figure}[htbp]
	\begin{subfigure}[b]{0.3\textwidth}
		\hspace{-3cm}
		\begin{tikzpicture}[scale = 2.5]
			\draw[help lines,step = 0.5]
			(0,-1/2) grid (1.5,1/2);
			\node at (0,0.5)[circle,fill,inner sep=2pt,color=black]{};
			\node at (0,0)[circle,fill,inner sep=2pt,color=red]{};
			\node at (0,-0.5)[circle,fill,inner sep=2pt,color=black]{};
			\node at (0.5,0.5)[circle,fill,inner sep=2pt,color=black]{};
			\node at (0.5,0)[circle,fill,inner sep=2pt,color=black]{};
			\node at (0.5,-0.5)[circle,fill,inner sep=2pt,color=black]{};
			\node (A) at (0.22,0.6) {{{$u_{0,j+1}$}}};
			\node (A) at (0.15,0.1) {{{$u_{0,j}$}}};
			\node (A) at (0.22,-0.4) {{{$u_{0,j-1}$}}};
			\node (A) at (0.72,0.6) {{{$u_{1,j+1}$}}};
			\node (A) at (0.65,0.1) {{{$u_{1,j}$}}};
			\node (A) at (0.72,-0.4) {{{$u_{1,j-1}$}}};
			\node (A) at (-0.2,0.1) {$\Gamma_1$};
			\node (A) at (-0.5,-0.1) {	$ \tfrac{\partial u}{\partial \nv}+{\color{blue}{\alpha}} u=g_1$};
		\end{tikzpicture}
	\end{subfigure}	
	\begin{subfigure}[b]{0.3\textwidth}
		\begin{tikzpicture}[scale = 2.5]
			\draw[help lines,step = 0.5]
			(0,-1/2) grid (1.5,1/2);
			\node at (0,0.5)[circle,fill,inner sep=2pt,color=black]{};
			\node at (0,0)[circle,fill,inner sep=2pt,color=red]{};
			\node at (0,-0.5)[circle,fill,inner sep=2pt,color=black]{};
			\node at (0.5,0.5)[circle,fill,inner sep=2pt,color=black]{};
			\node at (0.5,0)[circle,fill,inner sep=2pt,color=black]{};
			\node at (0.5,-0.5)[circle,fill,inner sep=2pt,color=black]{};
			\node (A) at (0.15,0.6) {{{$C_{0,1}$}}};
			\node (A) at (0.15,0.1) {{{$C_{0,0}$}}};
			\node (A) at (0.20,-0.4) {{{$C_{0,-1}$}}};
			\node (A) at (0.65,0.6) {{{$C_{1,1}$}}};
			\node (A) at (0.65,0.1) {{{$C_{1,0}$}}};
			\node (A) at (0.7,-0.4) {{{$C_{1,-1}$}}};
			\node (A) at (-0.2,0.1) {$\Gamma_1$};
			\node (A) at (-0.5,-0.1) {	$ \tfrac{\partial u}{\partial \nv}+{\color{blue}{\alpha}} u=g_1$};
		\end{tikzpicture}
	\end{subfigure}
	\caption
	{The illustration for the 6-point scheme in \eqref{stencil:regular:Robin:1} of \cref{thm:Robin:Gamma1}.}
	\label{Ckl:6:points:Gamma1}
\end{figure}
Now, we  discuss the  6-point FDM with the sixth-order consistency centered at the point $(x_i,y_j)\in \Gamma_1$ in the following theorem (see \cref{Ckl:6:points:Gamma1}).
\begin{theorem}\label{thm:Robin:Gamma1}
	Let  $(x_0,y_j)\in \Gamma_1$ and $(x_i^*,y_j^*)=(x_i,y_j)=(x_0,y_j)$. Then
	the following 6-point scheme  centered at $(x_0,y_j)$ (see \cref{Ckl:6:points:Gamma1}):
	\be	 \label{stencil:regular:Robin:1}
	\begin{split}
	h^{-1}	\mathcal{L}_h u_h  :=
	h^{-1} \sum_{k=0}^{1}\sum_{\ell =-1}^{1} C_{k,\ell}(u_{h})_{k,j+\ell}  =&	 h^{-1}\sum_{(m,n)\in \ind_{4}} f^{(m,n)}  \sum\limits_{k=0}^1 \sum\limits_{\ell=-1}^1 C_{k,\ell} H_{6,m,n}(kh, \ell h)\\
	& -h^{-1}\sum_{n=0}^{5}g_{1}^{(n)}  \sum\limits_{k=0}^1 \sum\limits_{\ell=-1}^1 C_{k,\ell} G_{6,1,n}(kh, \ell h),
	\end{split}
	\ee
	achieves the sixth-order consistency  for $\frac{\partial u}{\partial \nv}+\alpha u=g_1$ at the point $(x_0,y_j)$,
	where $\{	C_{k,\ell} : C_{k,\ell}:=\sum_{p=0}^{6} c_{k,\ell,p}h^p,\ c_{k,\ell,p}\in \R\}_{ k\in\{0,1\},\ \ell\in\{-1,0,1\} }$ is any nontrivial solution of the linear system induced by
	\be\label{Ckl:Gamma1:EQ:M=6}
	\begin{split}
		&\sum_{k=0}^1 \sum_{\ell=-1}^1 C_{k,\ell} \bigg( G_{6,0,n}(kh, \ell h) +(1-\delta_{n,6})
		\sum_{i=n}^5  {i\choose n}  {\alpha}^{(i-n)} G_{6,1,i}(kh, \ell h)  \bigg)\\
		&=\bo(h^{7}),\quad \mbox{for all} \quad n=0,\dots,6,  \ \text{with} \quad \delta_{6,6}=1 \quad \mbox{and}\quad
		\delta_{n,6}=0 \quad \mbox{for } \quad n\ne 6.
	\end{split}
	\ee
	By the symbolic calculation,
	the linear system in \eqref{Ckl:Gamma1:EQ:M=6} always has nontrivial solutions.
	 Furthermore,
	\begin{itemize}
		\item[(i)] One necessary condition for $\{C_{k,\ell}\}$ in \eqref{stencil:regular:Robin:1} to satisfy \eqref{suff:nece:sum:cond} is  $\alpha(y_j)\ge 0$;
		\item[(ii)] There must exist a nontrivial solution of \eqref{Ckl:Gamma1:EQ:M=6} such that
		$\{C_{k,\ell}\}$ in \eqref{stencil:regular:Robin:1}  satisfies the sign condition \eqref{sign:condition} and the sum condition \eqref{sum:condition} for any mesh size $h$ if $\alpha(y_j)\ge 0$.
	\end{itemize}
\end{theorem}
\begin{proof}[\textbf{Proof of \cref{thm:Robin:Gamma1}}]
	 In the following statement, \eqref{Lh:u:Gamma1}--\eqref{Gamma1:EQ:explicit} are used to derive \eqref{stencil:regular:Robin:1} and \eqref{Ckl:Gamma1:EQ:M=6}, \eqref{Matrix:Form:1:Gamma1}--\eqref{Gamma1:sd:345} are used to derive the necessary and sufficient condition for $\{C_{k,\ell}\}$ in \eqref{stencil:regular:Robin:1} to satisfy \eqref{suff:nece:sign:cond} and \eqref{suff:nece:sum:cond}, and prove items (i) and (ii).
	Let
	\[
	\begin{split}
		& \frac{1}{h}	\mathcal{L}_h u  := \frac{1}{h} \sum_{k=0}^{1} \sum_{\ell=-1}^{1} C_{k,\ell} u(x_0 + kh, y_j + \ell h),\quad C_{k,\ell}:=\sum_{p=0}^{M} c_{k,\ell,p}h^p, \quad  c_{k,\ell,p}\in \R,
		\quad k=0,1,\ \ell=-1,0,1.
	\end{split}
	\]
	\eqref{u0n:En:alpha} and \eqref{En:alpha} with $x_i^*=x_i=x_0$ and $y_j^*=y_j$ result in
	\be\label{Lh:u:Gamma1}
	\begin{split}
		& h^{-1}	\mathcal{L}_h u   = 	 h^{-1} \sum_{n=0}^M u^{(0,n)} I_{n} + \sum_{(m,n) \in \ind_{	 M-2}} f^{(m,n)} J_{m,n}  +\sum_{n=0}^{M-1}g_{1}^{(n)}J_{g_{1},n} +\bo(h^{M}),
	\end{split}
	\ee
	where
	\be\label{In:Gamma1}
	\begin{split}
		&	I_{n}:=\sum_{k=0}^{1} \sum_{\ell=-1}^{1} C_{k,\ell}E_{n}  (kh, \ell h),\qquad J_{m,n}:= h^{-1} \sum_{k=0}^{1} \sum_{\ell=-1}^{1}	 C_{k,\ell} H_{M,m,n}  (kh, \ell h), \\  & J_{g_1,n}:=  -h^{-1} \sum\limits_{k=0}^1 \sum\limits_{\ell=-1}^1 C_{k,\ell} G_{M,1,n}(kh, \ell h),
	\end{split}
	\ee	
	and $E_{n}  (kh, \ell h)$ is defined in \eqref{En:alpha}.
	We define
	\be\label{Lh:uh:Gamma1}
	h^{-1}	\mathcal{L}_h u_h:=  h^{-1} \sum_{k=0}^{1} \sum_{\ell=-1}^{1} C_{k,\ell} (u_h)_{k,j+\ell}  = \sum_{(m,n) \in \ind_{M-2}} f^{(m,n)} J_{m,n} +\sum_{n=0}^{M-1}g_{1}^{(n)}J_{g_{1},n}.
	\ee
	We deduce from \eqref{Lh:u:Gamma1} and \eqref{Lh:uh:Gamma1} that
	$
	h^{-1}	\mathcal{L}_h (u- u_h)=\bo(h^{M}),
    $
	if
	\be\label{Gamma1:EQ}
	\sum_{k=0}^{1} \sum_{\ell=-1}^{1} C_{k,\ell}E_{n}  (kh, \ell h)=\bo(h^{M+1}), \quad  \text{for all} \quad n=0,\dots, M.
	\ee
	where $E_{n}  (kh, \ell h)$ is defined in \eqref{En:alpha}.
By the symbolic calculation,
	\eqref{Gamma1:EQ} has a nontrivial solution $\{C_{k,\ell}\}$ if and only if $M\le 6$.
		Plugging $E_{n}(kh, \ell h)$ in \eqref{En:alpha} and $I_{n}$ in \eqref{In:Gamma1} into \eqref{Gamma1:EQ}, we have
	\be \label{Gamma1:EQ:explicit}
	\begin{split}
		& \sum_{k=0}^1 \sum_{\ell=-1}^1 C_{k,\ell}\bigg(  G_{M,0,n}(kh,\ell h )+(1-\delta_{n,M})
		\sum_{i=n}^{M-1}  {i\choose n}  {\alpha}^{(i-n)} G_{M,1,i}(kh,\ell h) \bigg) =\bo(h^{M+1}),
	\end{split}
	\ee
	for all $n=0,\dots, M$.
	We can obtain \eqref{stencil:regular:Robin:1}--\eqref{Ckl:Gamma1:EQ:M=6} in \cref{thm:Robin:Gamma1} by \eqref{In:Gamma1}--\eqref{Gamma1:EQ:explicit} with $M=6$.

		Next we use the following \eqref{Matrix:Form:1:Gamma1}--\eqref{Gamma1:ckl6:degree:6} to prove items (i) and (ii) in \cref{thm:Robin:Gamma1}.
	Using similar steps as
	 \eqref{regular:putGmn}--\eqref{Matrix:Form:1} or \eqref{Imn:irregular}--\eqref{irregular:simplify2},
	the system of linear equations in \eqref{Gamma1:EQ:explicit} can be represented in the following matrix form:
	\be\label{Matrix:Form:1:Gamma1}
A_0C_0=\textbf{0},\quad
A_d C_d=b_d,\quad d=1,\dots,M,
	\text{ with }
	C_d:=(c_{0,-1,d},
		c_{0,0,d},
		c_{0,1,d},
		c_{1,-1,d},
		c_{1,0,d},
		c_{1,1,d})^{\textsf{T}},
	\ee
	where all $A_d$ with $0\le d\le M$ are constant matrices,
	$b_d$  depends on $\{ C_i : 0\le i \le d-1 \}$, $\{a^{(i,j)}: (i,j) \in \ind_{M-1}\}$ and $\{\alpha^{(i)}:  0\le i \le M-1 \}$ for $1\le d\le M$.
	For example, similar to \eqref{A0:Regular}--\eqref{A1:A7:regular}, $\{A_d: 0\le d\le 6 \}$ in \eqref{Matrix:Form:1:Gamma1} with $M=6$ is
	{\small{
\be\label{A0:Gamma1}
A_0=	\begin{pmatrix}
1& 1& 1& 1& 1& 1\\	
-1& 0& 1& -1& 0& 1\\
1/2& 0& 1/2& 0& -1/2& 0\\
-1/6& 0& 1/6& 1/3& 0& -1/3\\
1/24& 0& 1/24& -1/6& 1/24& -1/6\\
-1/120& 0& 1/120& 1/30& 0& -1/30\\
1/720& 0& 1/720& 0& -1/720& 0\\
	\end{pmatrix}, \quad \text{the size of $A_0$ is } 7\times 6,\quad  A_1=A_0(1:6,:),
\ee}}
	 \be\label{Ad:Gamma1}
	\begin{split}
	A_2=A_0(1:5,:),	\quad A_3=A_0(1:4,:),\quad A_4=A_0(1:3,:),\quad  A_5=A_0(1:2,:), \quad A_6=A_0(1,:),
	\end{split}
	\ee
where the submatrix $A_0(1:n,:)$ consists of the first $n$ rows of $A_0$.
	All solutions of \eqref{Matrix:Form:1:Gamma1} with $M=6$ can be represented as
	\be\label{Gamma1:ckl0:degree:0}
	\begin{split}
		& c_{0,0,0}=-10c_{1,1,0},\qquad
		 c_{0,-1,0}=c_{0,1,0}=2c_{1,1,0},\qquad  c_{1,0,0}=4c_{1,1,0}, \qquad c_{1,-1,0}=c_{1,1,0},\\
		& 	 c_{1,1,0}  \text{ is free,} \text{ and } \{c_{k,\ell,0}\} \text{ satisfies }  \sum_{k=0}^{1} \sum_{\ell=-1}^{1}
		c_{k,\ell,0}=0,
	\end{split}
	\ee
	{\small{
	\be\label{Gamma1:ckl1:degree:1}
	\begin{split}
		& c_{0,0,1}=-10c_{1,1,1}+r_{1,1},\quad
		c_{0,-1,1}=2c_{1,1,1}+r_{1,2},\quad c_{0,1,1}=2c_{1,1,1}+r_{1,3},\quad  c_{1,0,1}=4c_{1,1,1}+r_{1,4}, \\
		& c_{1,-1,1}=c_{1,1,1}+r_{1,5}, \quad	 c_{1,1,1} \text{ is free,} \text{ and } \{c_{k,\ell,1}\} \text{ satisfies }  \sum_{k=0}^{1} \sum_{\ell=-1}^{1}
		c_{k,\ell,1}=6\alpha^{(0)},
	\end{split}
	\ee}}
	{\small{
	\be\label{Gamma1:ckl2:degree:2}
	\begin{split}
		& c_{0,0,2}=-10c_{1,1,2}+r_{2,1},\quad
		c_{0,-1,2}=2c_{1,1,2}+r_{2,2},\quad c_{0,1,2}=2c_{1,1,2}+r_{2,3},\quad  c_{1,0,2}=4c_{1,1,2}+r_{2,4}, \\
		&   c_{1,-1,2}=c_{1,1,2}+r_{2,5}, \quad	  c_{1,1,2} \text{ is free,} \text{ and } \{c_{k,\ell,2}\} \text{ satisfies }  \sum_{k=0}^{1} \sum_{\ell=-1}^{1}
		 c_{k,\ell,2}=-6\alpha^{(0)}c_{1,1,1}+r_{1,6},
	\end{split}
	\ee}}
	\be\label{Gamma1:ckl3:degree:3}
	\begin{split}
		& c_{0,0,3}=-2c_{1,0,3}-2c_{1,1,3}+r_{3,1},\quad
		 c_{0,-1,3}=\frac{1}{2}c_{1,0,3}+r_{3,2},\quad c_{0,1,3}=\frac{1}{2}c_{1,0,3}+r_{3,3}, \\
		& 	c_{1,-1,3}=c_{1,1,3}+r_{3,4},\quad \{ c_{1,0,3}, c_{1,1,3} \} \text{ are free,} \text{ and }   \sum_{k=0}^{1} \sum_{\ell=-1}^{1}
		 c_{k,\ell,3}=-6\alpha^{(0)}c_{1,1,2}+r_{2,6},
	\end{split}
	\ee
	\be\label{Gamma1:ckl4:degree:4}
	\begin{split}
		& c_{0,0,4}=-c_{1,-1,4}-2c_{1,0,4}-c_{1,1,4}+r_{4,1},\quad
		 c_{0,-1,4}=-\frac{1}{2}c_{1,-1,4}+\frac{1}{2}c_{1,0,4}+\frac{1}{2}c_{1,1,4}+r_{4,2}, \\ & c_{0,1,4}=\frac{1}{2}c_{1,-1,4}+\frac{1}{2}c_{1,0,4}-\frac{1}{2}c_{1,1,4}+r_{4,3}, \quad	  \{ c_{1,-1,4}, c_{1,0,4}, c_{1,1,4} \} \text{ are free}\\
		& \text{and } \{c_{k,\ell,4}\} \text{ satisfies }  \sum_{k=0}^{1} \sum_{\ell=-1}^{1}
		c_{k,\ell,4}=-\alpha^{(0)}( c_{1,0,3}+2c_{1,1,3})+r_{3,5},
	\end{split}
	\ee
	\be\label{Gamma1:ckl5:degree:5}
	\begin{split}
		& c_{0,0,5}=-2c_{0,1,5}-c_{1,0,5}-2c_{1,1,5}+r_{5,1},\quad
		 c_{0,-1,5}=c_{0,1,5}-c_{1,-1,5}+c_{1,1,5}+r_{5,2},  \\
		& \{ c_{0,1,5}, c_{1,-1,5}, c_{1,0,5}, c_{1,1,5} \}  \text{ are free, } \text{and }  \sum_{k=0}^{1} \sum_{\ell=-1}^{1}
		c_{k,\ell,5}=-\alpha^{(0)}( c_{1,-1,4}+c_{1,0,4}+c_{1,1,4})+r_{4,4},
	\end{split}
	\ee
	{\small{
	\be\label{Gamma1:ckl6:degree:6}
	\begin{split}
		& c_{0,-1,6}=-c_{0,0,6}-c_{0,1,6}-c_{1,-1,6}-c_{1,0,6}-c_{1,1,6}+r_{6,1},\  \{ c_{0,0,6}, c_{0,1,6}, c_{1,-1,6}, c_{1,0,6}, c_{1,1,6} \} \text{ are }\\
		& \text{free, and } \{c_{k,\ell,6}\} \text{ satisfies }  \sum_{k=0}^{1} \sum_{\ell=-1}^{1}
		c_{k,\ell,6}=-\alpha^{(0)}( c_{1,-1,5}+c_{1,0,5}+c_{1,1,5})+r_{5,3},
	\end{split}
	\ee}}
	where $\{r_{d,p}\}$ is determined by   $\{c_{k,\ell,s}\}$ with $0\le s\le d-1$, $\{a^{(i,j)}: (i,j) \in \ind_{5}\}$ and  $\{\alpha^{(i)}:  0\le i \le 5 \}$ for $d=1,\dots,6$.
	By $\sum_{k=0}^{1} \sum_{\ell=-1}^{1}
	c_{k,\ell,0}=0$ in \eqref{Gamma1:ckl0:degree:0} and $\sum_{k=0}^{1} \sum_{\ell=-1}^{1}
	c_{k,\ell,1}=6\alpha^{(0)}$ in \eqref{Gamma1:ckl1:degree:1},  $\{ c_{k,\ell,1} \}$ satisfies \eqref{suff:nece:sum:cond} if $\alpha^{(0)}\ge 0$. So, we proved the item (i) in \cref{thm:Robin:Gamma1}.

	For $\alpha^{(0)}=0$, we can check that $\{ c_{k,\ell,p} \}$ in \eqref{Gamma1:ckl0:degree:0}--\eqref{Gamma1:ckl6:degree:6} satisfies $\sum_{k=0}^{1} \sum_{\ell=-1}^{1}
	c_{k,\ell,p}=0$ for $p=0,\dots,6$.
	Let
	\be\label{Gamma1:Set:Free}
	\begin{split}
		& c_{1,1,0}=-1, \qquad c_{1,0,3}=c_{1,1,3}, \qquad  c_{1,-1,4}=c_{1,0,4}=c_{1,1,4},  \\
		& c_{0,1,5}=c_{1,-1,5}=c_{1,0,5}=c_{1,1,5}, \qquad c_{0,1,6}=c_{1,-1,6}=c_{1,0,6}=0, \qquad c_{0,0,6}=-2c_{1,1,6}.
	\end{split}
	\ee
	Then the non-empty intervals of  $\{c_{1,1,d}\}_{d=1,\dots,6}$ such that $\{c_{k,\ell,p}\}$ in \eqref{Gamma1:ckl0:degree:0}--\eqref{Gamma1:ckl6:degree:6} satisfies \eqref{suff:nece:sign:cond} and \eqref{suff:nece:sum:cond} for $\alpha^{(0)}\ge  0$ are
	{\small{
	\be\label{Gamma1:Interval}
	\begin{split}
		&  c_{1,1,d} \le  t_d:=\min \{\tfrac{r_{d,1}}{10}, \tfrac{-r_{d,2}}{2},\tfrac{-r_{d,3}}{2},\tfrac{-r_{d,4}}{4}, -r_{d,5}, s_d, 0 \} \quad \text{with} \quad d=1,2,\\
		&   c_{1,1,3} \le  t_3:=\min  \{ \tfrac{r_{3,1}}{4}, -2r_{3,2}, -2r_{3,3}, -r_{3,4}, s_3, 0  \},\quad  c_{1,1,4} \le  t_4:=  \min \{\tfrac{r_{4,1}}{4}, -2r_{4,2}, -2r_{4,3}, s_4, 0 \},\\
		& c_{1,1,5} \le  t_5:= \min \{\tfrac{r_{5,1}}{5}, -r_{5,2}, s_5, 0 \},\qquad  c_{1,1,6} \le  t_6:=\min\{-r_{6,1}, 0 \},
	\end{split}
	\ee}}
	where
	\be\label{Gamma1:sd:12}
	s_d=
	\begin{cases}
		\frac{r_{d,6}}{6\alpha^{(0)}}, &\mbox{ if }  \alpha^{(0)}>0,\\
		0, & \mbox{ if }  \alpha^{(0)}=0,
	\end{cases}\qquad d=1,2,
	\ee	
	\be\label{Gamma1:sd:345}
	s_3=
	\begin{cases}
		\frac{r_{3,5}}{3\alpha^{(0)}}, &\mbox{ if }  \alpha^{(0)}>0,\\
		0, & \mbox{ if }  \alpha^{(0)}=0,
	\end{cases}\quad
	s_4=
	\begin{cases}
		\frac{r_{4,4}}{3\alpha^{(0)}}, &\mbox{ if }  \alpha^{(0)}>0,\\
		0, & \mbox{ if }  \alpha^{(0)}=0,
	\end{cases}\quad
	s_5=
	\begin{cases}
		\frac{r_{5,3}}{3\alpha^{(0)}}, &\mbox{ if }  \alpha^{(0)}>0,\\
		0, & \mbox{ if }  \alpha^{(0)}=0.
	\end{cases}
	\ee	
By the symbolic calculation,  all  $r_{d,p}$ $\ne \pm\infty$ in \eqref{Gamma1:Interval}--\eqref{Gamma1:sd:345}  by $a\ne 0$ in $\overline{\Omega}$.
	Thus, we proved the item (ii) in \cref{thm:Robin:Gamma1}.
\end{proof}
\noindent
\textbf{Stencil coefficients $\{C_{k,\ell}\}$ in \cref{thm:Robin:Gamma1}  forming an M-matrix  for numerical tests:}
To verify the 6-point scheme \eqref{stencil:regular:Robin:1} of  \cref{thm:Robin:Gamma1} with numerical experiments in \cref{sec:numerical}, we use the unique $\{C_{k,\ell}\}$  by solving $A_0C_0=\textbf{0}$ and $A_dC_d=b_d$ in \eqref{Matrix:Form:1:Gamma1} with $M=6$, \eqref{Gamma1:Set:Free} and choosing  $c_{1,1,d}$ to be the maximum value such that
\[
\begin{cases}
c_{k,\ell,d}\ge 0, &\quad \mbox{if} \quad (k,\ell)=(0,0),\\
c_{k,\ell,d}\le  0, &\quad \mbox{if} \quad (k,\ell)\ne(0,0),
\end{cases}\quad \text{and} \quad
\sum_{k}\sum_{\ell} c_{k,\ell,d} \ge 0, \quad \text{for} \quad d=1,\dots,6.
\]
By the proof of  \cref{thm:Robin:Gamma1}, if $\alpha(y_j)\ge 0$ for $l_3< y_j <l_4$, then the above unique $\{C_{k,\ell}\}$ must exist and satisfy the sign condition \eqref{sign:condition}  and the sum condition \eqref{sum:condition}  for any $h$.
Similarly, we can obtain  6-point schemes with the sixth-order consistency at $(x_i,y_j)\in \cup_{p=2}^4\Gamma_p$ (see \cref{fig:model:problem}).

Next, we  discuss the  4-point FDM with the sixth-order consistency centered at the corner point $(x_i,y_j)=(l_1,l_3)$ in the following theorem (see \cref{4:points:scheme:Corner1,Ckl:4:points:Corner1}).
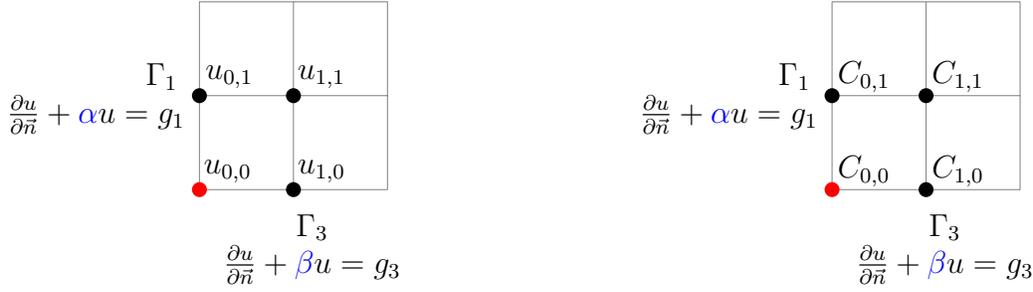
\begin{figure}[htbp]
	\begin{subfigure}[b]{0.3\textwidth}
		\hspace{-3cm}
		\begin{tikzpicture}[scale = 2.5]
			\draw[help lines,step = 0.5]
			(0,0) grid (1,1);
			\node at (0,0.5)[circle,fill,inner sep=2pt,color=black]{};
			\node at (0,0)[circle,fill,inner sep=2pt,color=red]{};
			\node at (0.5,0.5)[circle,fill,inner sep=2pt,color=black]{};
			\node at (0.5,0)[circle,fill,inner sep=2pt,color=black]{};
			\node (A) at (0.16,0.6) {$u_{0,1}$};
			\node (A) at (0.15,0.1) {$u_{0,0}$};
			\node (A) at (0.65,0.6) {$u_{1,1}$};
			\node (A) at (0.65,0.1) {$u_{1,0}$};
			\node (A) at (-0.2,0.6) {$\Gamma_1$};
			\node (A) at (-0.55,0.4) {$\tfrac{\partial u}{\partial \nv}+{\color{blue}{\alpha}} u=g_1$};
			\node (A) at (0.6,-0.2) {$\Gamma_3$};
			\node (A) at (0.6,-0.4) {$\tfrac{\partial u}{\partial \nv}+{\color{blue}{\beta}} u=g_3$};
		\end{tikzpicture}
	\end{subfigure}
	\begin{subfigure}[b]{0.3\textwidth}
		\begin{tikzpicture}[scale = 2.5]
			\draw[help lines,step = 0.5]
			(0,0) grid (1,1);
			\node at (0,0.5)[circle,fill,inner sep=2pt,color=black]{};
			\node at (0,0)[circle,fill,inner sep=2pt,color=red]{};
			\node at (0.5,0.5)[circle,fill,inner sep=2pt,color=black]{};
			\node at (0.5,0)[circle,fill,inner sep=2pt,color=black]{};
			\node (A) at (0.17,0.6) {$C_{0,1}$};
			\node (A) at (0.17,0.1) {$C_{0,0}$};
			\node (A) at (0.67,0.6) {$C_{1,1}$};
			\node (A) at (0.67,0.1) {$C_{1,0}$};
			\node (A) at (-0.2,0.6) {$\Gamma_1$};
			\node (A) at (-0.55,0.4) {$\tfrac{\partial u}{\partial \nv}+{\color{blue}{\alpha}} u=g_1$};
			\node (A) at (0.6,-0.2) {$\Gamma_3$};
			\node (A) at (0.6,-0.4) {$\tfrac{\partial u}{\partial \nv}+{\color{blue}{\beta}} u=g_3$};
		\end{tikzpicture}
	\end{subfigure}
	\caption
	{ The illustration for the 4-point scheme in \eqref{stencil:corner:1} in \cref{thm:Corner:1}. }
	\label{4:points:scheme:Corner1}
\end{figure}
\begin{figure}[htbp]
	\begin{subfigure}[b]{0.3\textwidth}
		\hspace{-1.5cm}
		\begin{tikzpicture}[scale = 2.5]
			\draw[help lines,step = 0.5]
			(0,0) grid (1,1);
			\node at (0,0.5)[circle,fill,inner sep=2pt,color=black]{};
			\node at (0,0)[circle,fill,inner sep=2pt,color=red]{};
			\node at (0.5,0.5)[circle,fill,inner sep=2pt,color=black]{};
			\node at (0.5,0)[circle,fill,inner sep=2pt,color=black]{};
			\node (A) at (0.16,0.6) {$C_{0,1}$};
			\node (A) at (0.15,0.1) {$C_{0,0}$};
			\node (A) at (0.65,0.6) {$C_{1,1}$};
			\node (A) at (0.65,0.1) {$C_{1,0}$};
			\node (A) at (-0.15,0.5) {$\Gamma_1$};
			\node (A) at (-0.55,0.6) {};
			\node (A) at (0.5,-0.2) {$\Gamma_3$};
			\node (A) at (0.8,-0.4) {};
			\node (A) at (1.5,0.4) {\color{red}\huge{$=$}};
		\end{tikzpicture}
	\end{subfigure}
	\begin{subfigure}[b]{0.3\textwidth}
		\hspace{-1cm}
		\begin{tikzpicture}[scale = 2.5]
			\draw[help lines,step = 0.5]
			(0,0) grid (1,1);
			\node at (0,0.5)[circle,fill,inner sep=2pt,color=black]{};
			\node at (0,0)[circle,fill,inner sep=2pt,color=red]{};
			\node at (0.5,0.5)[circle,fill,inner sep=2pt,color=black]{};
			\node at (0.5,0)[circle,fill,inner sep=2pt,color=black]{};
			\node (A) at (0.15,0.6) {$\hat{C}_{0,1}$};
			\node (A) at (0.15,0.1) {$\hat{C}_{0,0}$};
			\node (A) at (0.65,0.6) {$\hat{C}_{1,1}$};
			\node (A) at (0.65,0.1) {$\hat{C}_{1,0}$};
			\node (A) at (-0.15,0.6) {$\Gamma_1$};
			\node (A) at (-0.55,0.4) {$\tfrac{\partial u}{\partial \nv}+{\color{blue}{\alpha}} u=g_1$};
			\node (A) at (0.8,-0.1) {};
			\node (A) at (0.8,-0.4) {};
			\node (A) at (1.5,0.4) {\color{red}\huge{$+$}};
		\end{tikzpicture}
	\end{subfigure}
	\begin{subfigure}[b]{0.3\textwidth}
		\hspace{1cm}
		\begin{tikzpicture}[scale = 2.5]
			\draw[help lines,step = 0.5]
			(0,0) grid (1,1);
			\node at (0,0.5)[circle,fill,inner sep=2pt,color=black]{};
			\node at (0,0)[circle,fill,inner sep=2pt,color=red]{};
			\node at (0.5,0.5)[circle,fill,inner sep=2pt,color=black]{};
			\node at (0.5,0)[circle,fill,inner sep=2pt,color=black]{};
			\node (A) at (0.15,0.6) {$\tilde{C}_{0,1}$};
			\node (A) at (0.15,0.1) {$\tilde{C}_{0,0}$};
			\node (A) at (0.65,0.6) {$\tilde{C}_{1,1}$};
			\node (A) at (0.65,0.1) {$\tilde{C}_{1,0}$};
			\node (A) at (0.55,-0.15) {$\Gamma_3$};
			\node (A) at (0.55,-0.3) {$\tfrac{\partial u}{\partial \nv}+{\color{blue}{\beta}} u=g_3$};
		\end{tikzpicture}
	\end{subfigure}
	\caption
	{ The illustration for \eqref{stencil:corner:1} and \eqref{C:hat:tilde:corner:1}  in \cref{thm:Corner:1}. }
	\label{Ckl:4:points:Corner1}
\end{figure}
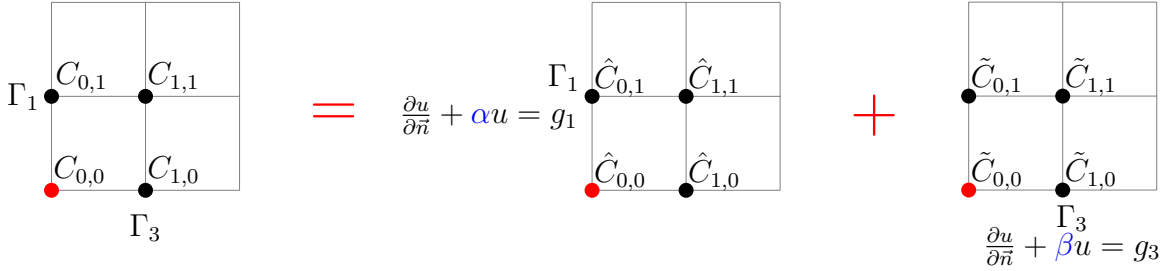
\begin{theorem}  \label{thm:Corner:1}
	Let  $(x_i^*,y_j^*)=(x_i,y_j)=(x_0,y_0)=\overline{\Gamma_1}\cap \overline{\Gamma_3}$.
	Then
	the following 4-point scheme  centered at $(x_0,y_0)$ (see \cref{4:points:scheme:Corner1,Ckl:4:points:Corner1}):
	\be \label{stencil:corner:1}
	\begin{aligned}
		h^{-1}	\mathcal{L}_h u_h  :=
		h^{-1}	\sum_{k=0}^{1}\sum_{\ell =0}^{1} C_{k,\ell}(u_{h})_{k,\ell}
		= \sum_{(m,n)\in \ind_{4}} f^{(m,n)}J_{m,n} + \sum_{n=0}^{5}g_{1}^{(n)}J_{g_{1},n} + \sum_{m=0}^{5}g_{3}^{(m)}J_{g_{3},m},
	\end{aligned}	
	\ee
	achieves the sixth-order consistency  for $\frac{\partial u}{\partial \nv}+\alpha u=g_1$ and $\frac{\partial u}{\partial \nv}+\beta u=g_3$ at the point $(x_0,y_0)$, where
	\be\label{C:hat:tilde:corner:1}
	\begin{split}
		& C_{k,\ell}:=\hat{C}_{k,\ell} + \tilde{C}_{k,\ell}, \quad c_{k,\ell,p}:=\hat{c}_{k,\ell,p} + \tilde{c}_{k,\ell,p}, \quad \hat{C}_{k,\ell}:=\sum_{p=0}^6  \hat{c}_{k,\ell,p}h^p, \quad \tilde{C}_{k,\ell}:=\sum_{p=0}^6  \tilde{c}_{k,\ell,p}h^p,
	\end{split}
	\ee
	$\{\hat{c}_{k,\ell,p} : \hat{c}_{k,\ell,p}\in \R, k,\ell=0,1 \}_{p=0,\dots,6}$ and $\{\tilde{c}_{k,\ell,p} : \tilde{c}_{k,\ell,p}\in \R, k,\ell=0,1\}_{p=0,\dots,6}$ are any nontrivial solutions of the linear system induced by \eqref{Corner:1:EQ} with $M=6$,
	$\{J_{m,n}\}_{(m,n)\in \ind_{4}}$, $\{J_{g_1,n}\}_{n=0}^5$ and  $\{J_{g_3,m}\}_{m=0}^5$ are defined in \eqref{Right:corner:1} with $M=6$. 	By the symbolic calculation,
	the linear system in \eqref{Corner:1:EQ} with $M=6$ always has nontrivial solutions.
	 Furthermore,
\begin{itemize}
	\item[(i)] One necessary condition for $\{C_{k,\ell}\}$ in \eqref{stencil:corner:1} to satisfy \eqref{suff:nece:sum:cond} is  $\alpha(y_0)+\beta(x_0)\ge 0$;
	\item[(ii)] There must exist a nontrivial solution of \eqref{Corner:1:EQ} with $M=6$  such that
	$\{C_{k,\ell}\}$ in \eqref{stencil:corner:1}  satisfies the sign condition \eqref{sign:condition} and the sum condition \eqref{sum:condition} for any mesh size $h$ if  $\alpha(y_0)+\beta(x_0)\ge 0$.
\end{itemize}
\end{theorem}
\begin{proof}[\textbf{Proof of \cref{thm:Corner:1}}]
	\begin{figure}[h]
		\centering
				\hspace{29mm}	
		\begin{subfigure}[b]{0.15\textwidth}
		\hspace{-40mm}
			\begin{tikzpicture}[scale = 0.45]
				\node (A) at (0,7) {$u^{(0,0)}$};
				\node (A) at (0,6) {$u^{(0,1)}$};
				\node (A) at (0,5) {$u^{(0,2)}$};
				\node (A) at (0,4) {$u^{(0,3)}$};
				\node (A) at (0,3) {$u^{(0,4)}$};
				\node (A) at (0,2) {$u^{(0,5)}$};
				\node (A) at (0,1) {$u^{(0,6)}$};
				\node (A) at (2,7) {$u^{(1,0)}$};
				\node (A) at (2,6) {$u^{(1,1)}$};
				\node (A) at (2,5) {$u^{(1,2)}$};
				\node (A) at (2,4) {$u^{(1,3)}$};
				\node (A) at (2,3) {$u^{(1,4)}$};
				\node (A) at (2,2) {$u^{(1,5)}$};	
				\node (A) at (4,7) {$u^{(2,0)}$};
				\node (A) at (4,6) {$u^{(2,1)}$};
				\node (A) at (4,5) {$u^{(2,2)}$};
				\node (A) at (4,4) {$u^{(2,3)}$};
				\node (A) at (4,3) {$u^{(2,4)}$};
				\node (A) at (6,7) {$u^{(3,0)}$};
				\node (A) at (6,6) {$u^{(3,1)}$};
				\node (A) at (6,5) {$u^{(3,2)}$};
				\node (A) at (6,4) {$u^{(3,3)}$};
				\node (A) at (8,7) {$u^{(4,0)}$};
				\node (A) at (8,6) {$u^{(4,1)}$};
				\node (A) at (8,5) {$u^{(4,2)}$};
				\node (A) at (10,7) {$u^{(5,0)}$};
				\node (A) at (10,6) {$u^{(5,1)}$};
				\node (A) at (12,7) {$u^{(6,0)}$};
				\draw[ultra thick, blue][->] (9,5) to[left] (13,5);     	
			\end{tikzpicture}
		\end{subfigure}
		\hspace{-0.6cm}
		\begin{subfigure}[b]{0.2\textwidth}
			\hspace{-30mm}
			\vspace{0.8cm}
			\begin{tikzpicture}[scale = 0.45]
				\node (A) at (0,7) {$u^{(0,0)}$};
				\node (A) at (0,6) {$u^{(0,1)}$};
				\node (A) at (2,7) {$u^{(1,0)}$};
				\node (A) at (2,6) {$u^{(1,1)}$};	
				\node (A) at (4,7) {$u^{(2,0)}$};
				\node (A) at (4,6) {$u^{(2,1)}$};
				\node (A) at (6,7) {$u^{(3,0)}$};
				\node (A) at (6,6) {$u^{(3,1)}$};
				\node (A) at (8,7) {$u^{(4,0)}$};
				\node (A) at (8,6) {$u^{(4,1)}$};
				\node (A) at (10,7) {$u^{(5,0)}$};
				\node (A) at (10,6) {$u^{(5,1)}$};
				\node (A) at (12,7) {$u^{(6,0)}$};
				\draw[ultra thick, blue][->] (6,5.5) to[left] (6,3.5);  	      	
			\end{tikzpicture}
		\end{subfigure}
		\begin{subfigure}[b]{0.15\textwidth}
			\hspace{-4cm}
			\vspace{0.2cm}
			\begin{tikzpicture}[scale = 0.45]
				\node (A) at (0,7) {$u^{(0,0)}$};
				\node (A) at (2,7) {$u^{(1,0)}$};	
				\node (A) at (4,7) {$u^{(2,0)}$};
				\node (A) at (6,7) {$u^{(3,0)}$};
				\node (A) at (8,7) {$u^{(4,0)}$};
				\node (A) at (10,7) {$u^{(5,0)}$};
				\node (A) at (12,7) {$u^{(6,0)}$};
			\end{tikzpicture}
		\end{subfigure}
		\hspace{-0.4cm}
		 \begin{subfigure}[b]{0.1\textwidth}
		 	\hspace{-0.5cm}
		 	\begin{tikzpicture}[scale = 0.45]
		 		\node (A) at (0,7) {$u^{(0,0)}$};
		 		\node (A) at (0,6) {$u^{(0,1)}$};
		 		\node (A) at (0,5) {$u^{(0,2)}$};
		 		\node (A) at (0,4) {$u^{(0,3)}$};
		 		\node (A) at (0,3) {$u^{(0,4)}$};
		 		\node (A) at (0,2) {$u^{(0,5)}$};
		 		\node (A) at (0,1) {$u^{(0,6)}$};
		 		\node (A) at (2,7) {$u^{(1,0)}$};
		 		\node (A) at (2,6) {$u^{(1,1)}$};
		 		\node (A) at (2,5) {$u^{(1,2)}$};
		 		\node (A) at (2,4) {$u^{(1,3)}$};
		 		\node (A) at (2,3) {$u^{(1,4)}$};
		 		\node (A) at (2,2) {$u^{(1,5)}$};
		 		\draw[ultra thick, blue][->] (-2,1.5) to[left] (-1,1.5); 	  	
		 	\end{tikzpicture}
		 \end{subfigure}
					\hspace{0.2cm}
				 \begin{subfigure}[b]{0.05\textwidth}
			\begin{tikzpicture}[scale = 0.45]
				\node (A) at (0,7) {$u^{(0,0)}$};
				\node (A) at (0,6) {$u^{(0,1)}$};
				\node (A) at (0,5) {$u^{(0,2)}$};
				\node (A) at (0,4) {$u^{(0,3)}$};
				\node (A) at (0,3) {$u^{(0,4)}$};
				\node (A) at (0,2) {$u^{(0,5)}$};
				\node (A) at (0,1) {$u^{(0,6)}$};
				\draw[ultra thick, blue][->] (-2,4) to[left] (-1,4); 	  	
			\end{tikzpicture}
		\end{subfigure}
		\caption
		{The illustration for \eqref{u:approx:tilde}--\eqref{um0:Em:beta}, \eqref{um0:to:u0n}--\eqref{uIm:2:corner:beta} with $M=6$.}
		\label{Corner:umn}
	\end{figure}
		 In the following statement, \eqref{um0:to:u0n}--\eqref{Corner:1:EQ} are used to derive \eqref{stencil:corner:1}, \eqref{Matrix:Form:1:Corner1}--\eqref{Corner1:Set:Free} are used to derive the necessary and sufficient condition for $\{C_{k,\ell}\}$ in \eqref{stencil:corner:1} to satisfy \eqref{suff:nece:sign:cond} and \eqref{suff:nece:sum:cond}, and prove items (i) and (ii).
	\eqref{upq1} implies (see \cref{Corner:umn}):
	\be\label{um0:to:u0n}
	u^{(m,0)}=\sum_{n=0}^{M}
	u^{(0,n)}\lambda_{m,n} +\sum_{n=0}^{M-1}
	u^{(1,n)}\mu_{m,n}  +\sum_{(i,j)\in \ind_{M-2}}
	f^{(i,j)}\nu_{m,i,j}, \qquad 0\le m \le M,
	\ee
	where $\lambda_{m,n}=a^u_{m,0,0,n}$, $\mu_{m,n}=a^u_{m,0,1,n}$, and $\nu_{m,i,j}=a^f_{m,0,i,j}$.
	By \eqref{u1n:Gamma1} and \eqref{um0:to:u0n},
	\be\label{um0:corner:alpha}
	\begin{split}
		u^{(m,0)}&=\sum_{n=0}^{M}
		u^{(0,n)}\lambda_{m,n} +\sum_{n=0}^{M-1}
		\bigg( \sum_{i=0}^n  {n\choose i}  {\alpha}^{(n-i)}u^{(0,i)} - g_1^{(n)} \bigg) \mu_{m,n}  +\sum_{(i,j)\in \ind_{M-2}}
		f^{(i,j)}\nu_{m,i,j}\\
		&=\sum_{n=0}^{M}
		u^{(0,n)}\lambda_{m,n} +\sum_{i=0}^{M-1} u^{(0,i)}
		\sum_{n=i}^{M-1}  {n\choose i}  {\alpha}^{(n-i)}\mu_{m,n} -\sum_{n=0}^{M-1} g_1^{(n)}   \mu_{m,n}  +\sum_{(i,j)\in \ind_{M-2}}
		f^{(i,j)}\nu_{m,i,j}\\
		&=\sum_{n=0}^{M}
		u^{(0,n)}p_{m,n}  -\sum_{n=0}^{M-1} g_1^{(n)}   \mu_{m,n}  +\sum_{(i,j)\in \ind_{M-2}}
		f^{(i,j)}\nu_{m,i,j}
	\end{split}
	\ee
	where
	$
	p_{m,n}= \lambda_{m,n} +  (1-\delta_{n,M}) \sum_{i=n}^{M-1}  {i\choose n}  {\alpha}^{(i-n)}\mu_{m,i}.
	$
	\eqref{um0:Em:beta} and \eqref{um0:corner:alpha} yield (see \cref{Corner:umn}):
	\be\label{uIm:2:corner:beta}
	\begin{split}
		u(x+x_i^*,y+y_j^*)
		&=   \sum_{n=0}^{M}
		u^{(0,n)} \sum_{m=0}^{M} p_{m,n} \tilde{E}_m(x,y)  -\sum_{n=0}^{M-1} g_1^{(n)}   \sum_{m=0}^{M} \mu_{m,n} \tilde{E}_m(x,y) \\
		& \quad  +\sum_{(m,n)\in \ind_{M-2}}
		f^{(m,n)}\sum_{i=0}^{M} \nu_{i,m,n}   \tilde{E}_i(x,y) -\sum_{m=0}^{M-1}
		g_{3}^{(m)}  \tilde{G}_{M,m,1}(x,y) \\
		&\quad+\sum_{(m,n)\in \ind_{M-2}}
		f^{(m,n)} \tilde{H}_{M,m,n}(x,y) +\bo(h^{M+1}).
	\end{split}
	\ee
	By \eqref{uIm:2:corner:beta}, we define the following $\tilde{u}(x+x_i^*,y+y_j^*)$ for the sake of presentation
	{\footnotesize{
	\be\label{uIm:tilde:corner:beta}
	\begin{split}
		\tilde{u}(x+x_i^*,y+y_j^*)
		&:=   \sum_{n=0}^{M}
		u^{(0,n)} \sum_{m=0}^{M} p_{m,n} \tilde{E}_m(x,y)   +\sum_{(m,n)\in \ind_{M-2}}
		f^{(m,n)}\big( \tilde{H}_{M,m,n}(x,y)+\sum_{i=0}^{M} \nu_{i,m,n}   \tilde{E}_i(x,y)\big) \\
		&  \quad   -\sum_{n=0}^{M-1} g_1^{(n)}   \sum_{m=0}^{M} \mu_{m,n} \tilde{E}_m(x,y)  -\sum_{m=0}^{M-1}
		g_{3}^{(m)}  \tilde{G}_{M,m,1}(x,y)+\bo(h^{M+1}).
	\end{split}
	\ee}}
	Choose $x_i^*=x_0$ and $y_j^*=y_0$,
	by \eqref{u0n:En:alpha} and \eqref{uIm:tilde:corner:beta}, we define
	\be\label{Lh:u:corner:define}
	\begin{split}
		& h^{-1}	\mathcal{L}_h u  := h^{-1} \sum_{k=0}^{1} \sum_{\ell=0}^{1} \hat{C}_{k,\ell} u(x_0 + kh, y_0 + \ell h) +h^{-1} \sum_{k=0}^{1} \sum_{\ell=0}^{1} \tilde{C}_{k,\ell} \tilde{u}(x_0 + kh, y_0 + \ell h),
	\end{split}
	\ee
	where
	\be\label{C:hat:tilde:corner1}
	\begin{split}
		& C_{k,\ell}:=\hat{C}_{k,\ell} + \tilde{C}_{k,\ell}, \quad c_{k,\ell,p}:=\hat{c}_{k,\ell,p} + \tilde{c}_{k,\ell,p}, \quad \hat{C}_{k,\ell}:=\sum_{p=0}^M  \hat{c}_{k,\ell,p}h^p, \quad \tilde{C}_{k,\ell}:=\sum_{p=0}^M  \tilde{c}_{k,\ell,p}h^p,
	\end{split}
	\ee	
	and $\hat{c}_{k,\ell,p}, \tilde{c}_{k,\ell,p}\in \R$.
	Then
	\be\label{Lh:u:corner:1}
	\begin{split}
		& h^{-1}	\mathcal{L}_h u   = 	 h^{-1} \sum_{n=0}^{M} u^{(0,n)} I_{n} + \sum_{(m,n) \in \ind_{	 M-2}} f^{(m,n)} J_{m,n}   + \sum_{n=0}^{	 M-1} g_1^{(n)} J_{g_1,n}  + \sum_{m=0}^{	M-1} g_3^{(m)} J_{g_3,m} +\bo(h^{M}),
	\end{split}
	\ee
	where
	\be\label{In:corner:1}
	\begin{split}
		I_n=\sum_{k=0}^{1} \sum_{\ell=0}^{1} \Big(\hat{C}_{k,\ell}E_n  (kh, \ell h)+\tilde{C}_{k,\ell}\sum_{m=0}^{M} p_{m,n} \tilde{E}_m   (kh, \ell h)\Big),
	\end{split}
	\ee	
	\be\label{Right:corner:1}
	\begin{split}
		& J_{m,n}= h^{-1}\sum_{k=0}^{1} \sum_{\ell=0}^{1} \Big( \hat{C}_{k,\ell}H_{M,m,n}  (kh, \ell h)+ \tilde{C}_{k,\ell} \Big( \tilde{H}_{M,m,n}(kh, \ell h)+\sum_{i=0}^{M} \nu_{i,m,n}   \tilde{E}_i(kh, \ell h)\Big)\Big), \\
		& J_{g_1,n}=-h^{-1}\sum_{k=0}^{1} \sum_{\ell=0}^{1} \Big(\hat{C}_{k,\ell} G_{M,1,n} (kh, \ell h)+\tilde{C}_{k,\ell}\sum_{m=0}^{M} \mu_{m,n} \tilde{E}_m(kh, \ell h)\Big),\\
		& J_{g_3,m}=-h^{-1}\sum_{k=0}^{1} \sum_{\ell=0}^{1}\tilde{C}_{k,\ell}  \tilde{G}_{M,m,1}(kh, \ell h).
	\end{split}
	\ee	
	Let
	\be\label{Lh:uh:corner:1}
	\begin{split}
		h^{-1}	\mathcal{L}_h u_h:=  h^{-1} \sum_{k=0}^{1} \sum_{\ell=0}^{1} C_{k,\ell} (u_h)_{k,\ell}  = \sum_{(m,n) \in \ind_{	 M-2}} f^{(m,n)} J_{m,n}   + \sum_{n=0}^{	 M-1} g_1^{(n)} J_{g_1,n}  + \sum_{m=0}^{	 M-1} g_3^{(m)} J_{g_3,m}.
	\end{split}
	\ee
	Then \eqref{Lh:u:corner:1} and \eqref{Lh:uh:corner:1} result in
	$
	h^{-1}	\mathcal{L}_h (u- u_h)=\bo(h^{M}),
	$
	if
	\be\label{Corner:1:EQ}
	\sum_{k=0}^{1} \sum_{\ell=0}^{1} \Big(\hat{C}_{k,\ell}E_n  (kh, \ell h)+\tilde{C}_{k,\ell}\sum_{m=0}^{M} p_{m,n} \tilde{E}_m   (kh, \ell h)\Big)=\bo(h^{M+1}), \quad  \text{for} \quad n=0,\dots, M.
	\ee
By the symbolic calculation,
	\eqref{Corner:1:EQ} has a nontrivial solution $\{C_{k,\ell}\}$ if and only if $M\le 6$.
	 \eqref{Right:corner:1}--\eqref{Corner:1:EQ} with $M=6$ yield \eqref{stencil:corner:1} in \cref{thm:Corner:1}.

	Next we check the existence of  $\{C_{k,\ell}\}$ in \eqref{stencil:corner:1} to   satisfy the sign condition \eqref{sign:condition} and the sum condition \eqref{sum:condition} for any $h$.
	Similarly to
	 \eqref{regular:putGmn}--\eqref{Matrix:Form:1} or \eqref{Imn:irregular}--\eqref{irregular:simplify2},
	the system of linear equations in \eqref{Corner:1:EQ} can be represented in the following matrix form:
	\be\label{Matrix:Form:1:Corner1}
	A_0C_0=\textbf{0},\
	A_d C_d=b_d,\  d=1,\dots,M,\
	C_d:=(\hat{c}_{0,0,d},
		\hat{c}_{0,1,d},
		\hat{c}_{1,0,d},
		\hat{c}_{1,1,d},
		\tilde{c}_{0,0,d},
		\tilde{c}_{0,1,d},
		\tilde{c}_{1,0,d},
		\tilde{c}_{1,1,d})^{\textsf{T}},
	\ee
	where all $A_d$ with $0\le d\le M$ are constant matrices,
	$b_d$  depends on $\{ C_i : 0\le i \le d-1 \}$, $\{a^{(i,j)}: (i,j) \in \ind_{M-1}\}$, $\{\alpha^{(i)}:  0\le i \le M-1 \}$ and $\{\beta^{(i)}:  0\le i \le M-1 \}$ for $1\le d\le M$.
For example, similar to \eqref{A0:Regular}--\eqref{A1:A7:regular}, $\{A_d: 0\le d\le 6 \}$ in \eqref{Matrix:Form:1:Corner1} with $M=6$ is
	{\small{
		\be\label{A0:Corner1}
		A_0=	\begin{pmatrix}
1& 1& 1& 1& 1& 1& 1& 1\\			
0& 1& 0& 1& 0& 0& 0& 0\\
0& 1/2& -1/2& 0& 0& 1/2& -1/2& 0\\ 	
0& 1/6& 0& -1/3& 0& 0& 0& 0\\
0& 1/24& 1/24& -1/6& 0& 1/24& 1/24& -1/6\\
0& 1/120& 0& -1/30& 0& 0& 0& 0\\ 			
0& 1/720& -1/720& 0& 0& 1/720& -1/720& 0\\
	\end{pmatrix},\quad \text{the size of $A_0$ is } 7\times 8,
\ee}}
\be\label{Ad:Corner1}
\begin{split}
A_1=A_0(1:6,:), \quad A_2=A_0(1:5,:),	\quad A_3=A_0(1:4,:),\quad A_4=A_0(1:3,:),\quad  A_5=A_0(1:2,:),
\end{split}
\ee
and $A_6=A_0(1,:)$, where the submatrix $A_0(1:n,:)$ consists of the first $n$ rows of $A_0$.
	Similarly to the proof of \cref{thm:Robin:Gamma1},  one necessary condition for $\{C_{k,\ell}\}$ in \eqref{stencil:corner:1} to   satisfy \eqref{suff:nece:sum:cond} is $\alpha^{(0)}+\beta^{(0)}\ge 0$.
	Let
	\be\label{Corner1:Set:Free}
	\begin{split}
		& \tilde{c}_{1,1,0}=-1, \qquad \tilde{c}_{0,0,0}=\tilde{c}_{1,0,0}=\tilde{c}_{0,0,1}=\tilde{c}_{1,0,1}=\tilde{c}_{0,0,2}=\tilde{c}_{1,0,2}=0, \\
		&  \tilde{c}_{0,1,3}=  \tilde{c}_{1,1,3}, \quad  \tilde{c}_{0,0,3}=\tilde{c}_{1,0,3}=0, \quad \hat{c}_{1,1,4}=  \tilde{c}_{1,1,4},   \quad  \tilde{c}_{0,1,4}=  2\tilde{c}_{1,1,4}, \quad  \tilde{c}_{0,0,4}=\tilde{c}_{1,0,4}=0,\\
		&    \hat{c}_{1,0,5}=\hat{c}_{1,1,5}= \tilde{c}_{1,0,5}= \tilde{c}_{1,1,5},  \qquad  \tilde{c}_{0,1,5}=  3\tilde{c}_{1,1,5}, \qquad  \tilde{c}_{0,0,5}=0,\\
		&    \hat{c}_{0,1,6}=\hat{c}_{1,0,6}= \hat{c}_{1,1,6}=  \tilde{c}_{0,1,6}= \tilde{c}_{1,0,6}=  \tilde{c}_{1,1,6},  \qquad  \tilde{c}_{0,0,6}=0.
	\end{split}
	\ee
Then	similar to \eqref{Gamma1:ckl0:degree:0}--\eqref{Gamma1:sd:345}, we can prove the item (ii) of \cref{thm:Corner:1}.
\end{proof}	
%
%
\noindent
\textbf{Stencil coefficients $\{C_{k,\ell}\}$ in \cref{thm:Corner:1} forming an M-matrix  for numerical tests:}
To verify the 4-point scheme \eqref{stencil:corner:1} of  \cref{thm:Corner:1} with numerical experiments in \cref{sec:numerical}, we use the unique $\{C_{k,\ell}\}$ by solving $A_0C_0=\textbf{0}$ and $A_dC_d=b_d$ in \eqref{Matrix:Form:1:Corner1} with $M=6$, \eqref{Corner1:Set:Free} and choosing  $\tilde{c}_{1,1,d}$ to be the maximum value such that
\[
\begin{cases}
	\hat{c}_{k,\ell,d}+\tilde{c}_{k,\ell,d}\ge 0, &\quad \mbox{if} \quad (k,\ell)=(0,0),\\
	\hat{c}_{k,\ell,d}+\tilde{c}_{k,\ell,d}\le  0, &\quad \mbox{if} \quad (k,\ell)\ne(0,0),
\end{cases}\quad \text{and} \quad
\sum_{k}\sum_{\ell} (\hat{c}_{k,\ell,d}+ \tilde{c}_{k,\ell,d}) \ge 0, \quad \text{for} \quad d=1,\dots,6.
\]
 By the proof of  \cref{thm:Corner:1},  if  $\alpha(y_0)+\beta(x_0)\ge 0$, then the above unique $\{C_{k,\ell}\}$ must exist and satisfy the sign condition \eqref{sign:condition}  and the sum condition \eqref{sum:condition}  for any $h$.
Similarly, we can obtain  4-point schemes with the sixth-order consistency at $(x_i,y_j)=\overline{\Gamma_1}\cap \overline{\Gamma_4}$, $\overline{\Gamma_2}\cap \overline{\Gamma_3}$ and $\overline{\Gamma_2}\cap \overline{\Gamma_4}$ (see \cref{fig:model:problem}).

\end{document}